\documentclass[12pt]{amsart}

\usepackage{framed}
\usepackage{braket}
\usepackage{amsmath,amssymb,amsthm}
\usepackage{color}
\usepackage{bm}
\usepackage[all]{xy}
\usepackage{amscd}
\usepackage{graphicx}
\usepackage{ascmac}
\usepackage{BOONDOX-frak, mathrsfs}
\usepackage[top=30truemm,bottom=30truemm,left=25truemm,right=25truemm]{geometry}
\usepackage{mathtools}
\mathtoolsset{showonlyrefs=true}
\usepackage{tikz}
\usetikzlibrary{cd}

\makeatletter
\@addtoreset{equation}{section}

\makeatother

\newcommand{\Z}{\mathbb{Z}}
\newcommand{\Q}{\mathbb{Q}}
\newcommand{\F}{\mathbb{F}}
\newcommand{\G}{\mathbb{G}}
\newcommand{\C}{\mathbb{C}}
\newcommand{\bT}{\mathbb{T}}
\newcommand{\bK}{\mathbb{K}}

\newcommand{\bft}{\mathbf{t}}
\newcommand{\btr}{\mathbf{1}}

\newcommand{\Nrm}{\mathrm{N}}

\newcommand{\fO}{\mathfrak{O}}
\newcommand{\fH}{\mathfrak{H}}
\newcommand{\fa}{\mathfrak{a}}
\newcommand{\fb}{\mathfrak{b}}
\newcommand{\fc}{\mathfrak{c}}
\newcommand{\fd}{\mathfrak{d}}
\newcommand{\ff}{\mathfrak{f}}
\newcommand{\fl}{\mathfrak{l}}
\newcommand{\fm}{\mathfrak{m}}
\newcommand{\fn}{\mathfrak{n}}
\newcommand{\fp}{\mathfrak{p}}
\newcommand{\fq}{\mathfrak{q}}
\newcommand{\fr}{\mathfrak{r}}
\newcommand{\ft}{\mathfrak{t}}
\newcommand{\fP}{\mathfrak{P}}
\newcommand{\fQ}{\mathfrak{Q}}
\newcommand{\fR}{\mathfrak{R}}
\newcommand{\fL}{\mathfrak{L}}

\newcommand{\bpsi}{\boldsymbol{\psi}}
\newcommand{\Rc}{\Z_p[G]^{\chi}}

\newcommand{\OO}{\mathcal{O}}

\newcommand{\cZ}{\mathcal{Z}}
\newcommand{\cA}{\mathcal{A}}

\newcommand{\cG}{\mathcal{G}}
\newcommand{\cH}{\mathcal{H}}
\newcommand{\cS}{\mathcal{S}}

\newcommand{\wtil}[1]{\widetilde{#1}}
\newcommand{\ol}[1]{\overline{#1}}
\newcommand{\parenth}[1]{\left( #1 \right)}

\newcommand{\AC}[1]{Q_{#1}}
\newcommand{\coef}[2]{{c \parenth{#1, #2}}}
\newcommand{\coefz}[2]{{c_{#1} \parenth{0, #2}}}

\newcommand{\condu}[1]{\ff_{#1}}
\newcommand{\ordinary}{\mathchar`- \ord}

\newcommand{\detZ}{{\rm det}_{\mathbb{Z}[G]^-}}

\newcommand{\detC}{{\rm det}_{\mathbb{C}[G]^-}}

\newcommand{\GK}{\Gal(K/F)}

\DeclareMathOperator{\Gal}{Gal}
\DeclareMathOperator{\Ind}{Ind}
\DeclareMathOperator{\Coker}{Cok}
\DeclareMathOperator{\Ker}{Ker}
\DeclareMathOperator{\Imag}{Im}
\DeclareMathOperator{\Fitt}{Fitt}

\DeclareMathOperator{\pd}{pd}
\DeclareMathOperator{\ram}{ram}

\DeclareMathOperator{\Cl}{Cl}
\DeclareMathOperator{\ord}{ord}
\DeclareMathOperator{\Hom}{Hom}
\DeclareMathOperator{\Ext}{Ext}

\DeclareMathOperator{\ur}{ur}
\DeclareMathOperator{\ab}{ab}
\DeclareMathOperator{\id}{id}

\DeclareMathOperator{\End}{End}
\DeclareMathOperator{\cyc}{cyc}
\DeclareMathOperator{\cha}{char}
\DeclareMathOperator{\GL}{GL}
\DeclareMathOperator{\prim}{prime}
\DeclareMathOperator{\sgn}{sgn}
\DeclareMathOperator{\bad}{bad}

\DeclareMathOperator{\wild}{wild}
\DeclareMathOperator{\Frac}{Frac}

\DeclareMathOperator{\res}{res}
\DeclareMathOperator{\cl}{cl}
\DeclareMathOperator{\Frob}{Frob}
\DeclareMathOperator{\cusps}{cusps}

\DeclareMathOperator{\ba}{bad}

\let\oldenumerate\enumerate
\renewcommand{\enumerate}{
   \oldenumerate
   \setlength{\itemsep}{1pt}
   \setlength{\parskip}{0pt}
   \setlength{\parsep}{0pt}
}
\let\olditemize\itemize
\renewcommand{\itemize}{
   \olditemize
   \setlength{\itemsep}{1pt}
   \setlength{\parskip}{0pt}
   \setlength{\parsep}{0pt}
}

\theoremstyle{plain}
\newtheorem{thm}{Theorem}[section]
\newtheorem{lem}[thm]{Lemma}
\newtheorem{conj}[thm]{Conjecture}
\newtheorem{prop}[thm]{Proposition}

\newtheorem{claim}[thm]{Claim}

\theoremstyle{definition}
\newtheorem{defn}[thm]{Definition}
\newtheorem{setting}[thm]{Setting}
\newtheorem{rem}[thm]{Remark}

\title[eTNC]
{On the minus component of the equivariant Tamagawa number conjecture for $\mathbb{G}_{\lowercase{m}}$}
\author[M.~Atsuta \& T. Kataoka]{Mahiro Atsuta and Takenori Kataoka}
\address{Faculty of Science and Technology, Keio University.
3-14-1 Hiyoshi, Kohoku-ku, Yokohama, Kanagawa 223-8522, Japan}
\email{tkataoka@math.keio.ac.jp}
\email{atsuta0128@keio.jp}
\date{\today}

\begin{document}

\begin{abstract}
The equivariant Tamagawa number conjecture (hereinafter called the eTNC) predicts close relationships between algebraic and analytic aspects of motives.
In this paper, 
we prove a lot of new cases of the minus component of the eTNC for $\G_m$ and for CM abelian extensions.
One of the main results states that the $p$-component of the eTNC is true when there exists at least one $p$-adic prime that is tamely ramified.
The fundamental strategy is inspired by the work of Dasgupta and Kakde on the Brumer--Stark conjecture.
\end{abstract}

\maketitle

\tableofcontents

\section{Introduction}\label{s:intro}

The relationship between the special values of $L$-functions and algebraic objects is a central theme in number theory.  
The equivariant Tamagawa number conjecture,
which was formulated by Burns and Flach \cite{BuFl01}, is a very general and strong conjecture on such a relationship.
In this paper, 
we focus on the equivariant Tamagawa number conjecture for $\G_m$ (referred to below as the eTNC for short). 

So far, a promising way to prove the eTNC has been to apply a suitable (equivariant) Iwasawa main conjecture, together with a suitable descent theory.
For example, using this strategy, Burns and Greither \cite{BuGr03} and Flach \cite{Flach11} proved the eTNC for an abelian extension of $\Q$.
Also, Bley \cite{Ble06} proved the eTNC for a finite abelian extension of an imaginary quadratic field, under certain hypotheses.
In both cases, the relevant Iwasawa main conjecture could be proved by the theory of Euler systems: the system of cyclotomic units over $\Q$ and the system of elliptic units over an imaginary quadratic field, respectively.
More recently, in \cite{BKS17}, Burns, Kurihara and Sano established a general descent theory for an arbitrary number field to deduce the eTNC from the Iwasawa main conjecture and additional conjectures. 
The result \cite[Theorem 1.1]{BKS17} is strong enough to recover the above-mentioned results of \cite{BuGr03}, \cite{Flach11}, and \cite{Ble06}.
Note also that Burns \cite{Bur11} proved a global function field analogue of the eTNC within a similar framework.

In this paper, we focus on the minus component of the eTNC for a finite abelian CM-extension, which we refer to as the eTNC$^-$.
We will review a formulation of the eTNC$^-$ in \S \ref{ss:form_eTNC}.
Note that we ignore the $2$-components whenever we consider the minus components in this paper.
Then the eTNC$^-$ can be decomposed into the $p$-component eTNC$_p^-$ for each odd prime number $p$.

This eTNC$_p^-$ has been investigated closely, still via the Iwasawa main conjecture.
For instance, Nickel \cite[Theorem 4]{Nic11I} proved the eTNC$_p^-$ under the hypotheses that the relevant $\mu$-invariant vanishes, that the field extension is {\it almost tamely ramified} above $p$ (the definition will be recalled after Theorem \ref{thm:current_main} below), and that another mild condition holds.
Note that the vanishing of the $\mu$-invariant is imposed for the Iwasawa main conjecture to hold.
On the other hand, as a consequence of the above-mentioned result, Burns, Kurihara, and Sano \cite[Corollary 1.2]{BKS17} proved the eTNC$_p^-$ under the hypotheses that
the associated $p$-adic $L$-function has at most one trivial zero, still assuming the vanishing of the $\mu$-invariant.
The condition on the number of the trivial zeros is imposed in order to use the Gross--Stark conjecture proved by Darmon, Dasgupta, and Pollack \cite{DDP11} and Ventullo \cite{Ven15}.
Subsequently, Dasgupta, Kakde, and Ventullo \cite{DKV18} proved a significant part of the Gross--Stark conjecture without the condition on the trivial zeros, but this does not allow us to remove the condition from the result of Burns, Kurihara, and Sano (see \cite[\S 1.C]{BKS17}).

On the other hand, in a recent preprint \cite{DK20}, 
Dasgupta and Kakde {\it unconditionally} proved the Brumer--Stark conjecture (we do not review the statement), except for the $2$-components.
A promising strategy to prove the Brumer--Stark conjecture had been again to descend from the (equivariant) Iwasawa main conjecture (e.g., Greither and Popescu \cite{GP15}).
However, Dasgupta and Kakde do not follow this strategy, and instead they investigate finite abelian CM-extensions directly.
This is why they could avoid any hypotheses for the Iwasawa main conjecture and for the descent procedure.

In more recent work, using the work \cite{DK20},
Nickel \cite{Nickel21} proved the eTNC$_p^-$ as long as all the $p$-adic primes are almost tamely ramified, without assuming the vanishing of the $\mu$-invariant. 
Also, using the work \cite{DK20}, Johnston and Nickel \cite{JN20} proved the relevant Iwasawa main conjecture without assuming the vanishing of the $\mu$-invariant.

In this paper, by adapting the method of Dasgupta and Kakde \cite{DK20} on the Brumer--Stark conjecture, we prove the eTNC$_p^-$ in much more cases.
The main theorem is the following.

\begin{thm}\label{thm:current_main}
Let $H/F$ be a finite abelian CM-extension.
Let $p$ be an odd prime number.
Then the eTNC$_p^-$ for $H/F$ holds if at least one of the following conditions holds.
\begin{itemize}
\item[(i)]
There exists at least one $p$-adic prime of $F$ that is (at most) tamely ramified in $H/F$.
\item[(ii)]
The maximal extension of $F$ in $H$ that is totally split at all the $p$-adic primes is totally real.
\item[(iii)]
We have $H^{\cl, +} (\mu_p) \not\supset H^{\cl}$. 
Here, $H^{\cl}$ denotes the Galois closure of $H$ over $\Q$, $H^{\cl, +}$ its maximal real subfield, and $\mu_p$ the group of $p$-th roots of unity.
\end{itemize}
\end{thm}

We emphasize that the assumption of this theorem is very mild.
See Remark \ref{rem:Wiles_cond} below for an equivalent condition of (iii).
The above-mentioned result of Nickel \cite{Nickel21} is also a part of this theorem.
Recall that a $p$-adic prime $v$ of $F$ is said to be almost tamely ramified in $H/F$ if either $v$ is (at most) tamely ramified in $H / F$ or 
the complex conjugation is in the decomposition subgroup of $v$ in $\Gal(H/F)$.
Therefore, the existence of a {\it single} $p$-adic prime that is almost tamely ramified ensures the condition (i) or (ii).
Note that \cite{Nickel21} also deals with non-commutative extensions, which we do not study in this paper.

Even without assuming (i), (ii), or (iii) of Theorem \ref{thm:current_main},
in this paper we also show that the eTNC$_p^-$ holds after suitably enlarging the coefficient ring.
See Theorem \ref{thm:main_p_2} for this result.

\begin{rem}
After the authors had completed this project, very recently Bullach, Burns, Daoud, and Seo \cite{BBDS} announced {\it an unconditional proof of the eTNC$^-$}.
This is achieved by combining the result of \cite{DK20} with a newly developed general theory of Euler system.
On the other hand, this present paper does not rely on the theory of Euler systems at all and instead directly adapt the strategy of \cite{DK20} (see \S \ref{ss:outline}).
\end{rem}

In \S \ref{ss:form_eTNC}, we formulate the eTNC$_p^-$.
In \S \ref{ss:eTNCp}, we state a more refined result (Theorem \ref{thm:current_main_p}) on character-components of the eTNC$_p^-$, from which we will deduce Theorems \ref{thm:current_main} and \ref{thm:main_p_2}.
In \S \ref{ss:outline}, we give an outline the proof of Theorem \ref{thm:current_main_p}.

\subsection{The formulation of the eTNC}\label{ss:form_eTNC}

In this section, we give a formulation of the eTNC$^-$ that we adopt in this paper.
The equivalence with a more standard formulation will be shown in \S \ref{Ap:eTNC}.

\subsubsection{Notation}\label{sss:not}

Let $H/F$ be a finite abelian CM-extension, which means that $H$ is a CM-field and $F$ is a totally real number field such that the Galois group $G = \Gal(H/F)$ is finite and abelian.

Let $c \in G$ denote the complex conjugation.
We define $\Z[G]^- = \Z[1/2][G]/(1 + c)$. 
For a $\Z[G]$-module $M$, we also define $M^- = \Z[G]^- \otimes_{\Z[G]} M$, which is regarded as a $\Z[G]^-$-module.
When $x$ is an element of $M$, we write $x^-$ for the image of $x$ by the natural homomorphism $M \to M^-$.
Note that we implicitly invert the prime $2$ whenever we study minus components.

We write $S_{\ram}(H/F)$ for the set of places of $F$ that are ramified in $H/F$, including the infinite places.
We also write $S_{\wild}(H/F) \subset S_{\ram}(H/F)$ for the set of finite primes of $F$ that are wildly ramified in $H/F$.
We write $S_{\infty}(F)$ (resp.~$S_p(F)$ for an odd prime number $p$) for the set of infinite places of $F$ (resp.~$p$-adic primes of $F$).
For a finite set $\Sigma$ of places of $F$ such that $\Sigma \supset S_{\infty}(F)$, 
we put $\Sigma_f = \Sigma \setminus S_{\infty}(F)$.

For a finite set $\Sigma'$ of finite primes of $F$, we write $\Cl_H^{\Sigma'}$ for the ray class group of $H$ of modulus $\prod_{w \in \Sigma'_H} w$, where $\Sigma'_H$ denotes the set of primes of $H$ that lie above primes in $\Sigma'$.
We write $\mu(H)$ for the group of roots of unity in $H$ and define
\[
\mu(H)^{\Sigma'} = \{ \zeta \in \mu(H) \mid \zeta \equiv 1 (\bmod{w}), \forall w \in \Sigma'_H\}.
\]
Note that the complex conjugation acts as the inversion on $\mu(H)$, so we have $\mu(H)^{\Sigma', -} = \Z[1/2] \otimes_{\Z} \mu(H)^{\Sigma'}$.

\subsubsection{A Ritter--Weiss type module}\label{sss:Omega_statement}

We introduce an arithmetic module, denoted by $\Omega_{\Sigma}^{\Sigma'}$, that plays a central role in this paper.
The construction and the proof of its basic properties will occupy Section \ref{s:Omega}.
The module is basically the same as what the authors used in preceding work \cite{AK21}, though we have to slightly generalize the situation.
In the work \cite{AK21}, the module was used for computing the Fitting ideals of ray class groups of $H$.
The construction is based on Ritter and Weiss \cite{RW96}, Greither \cite{Gre07}, and Kurihara \cite{Kur20}.


For each finite prime $v$ of $F$, we write $G_v$ (resp.~$I_v$) for the decomposition group (resp.~the inertia group) of $v$ in $G$, and $\varphi_v \in G_v / I_v$ for the arithmetic Frobenius.
We define a finite $\Z[G]$-module $A_v$ by
\begin{equation}\label{eq:A_v}
A_v = \Z[G/I_v]/(1 - \varphi_v^{-1} + \# I_v).
\end{equation}

Now we introduce the module $\Omega_{\Sigma}^{\Sigma'}$, which will be constructed in \S \ref{ss:Omega_constr}.

\begin{prop}\label{prop:Omega}
Let $\Sigma$ and $\Sigma'$ be finite sets of places of $F$ satisfying the following.
\begin{itemize}
\item[(H1)]
$\Sigma \cap \Sigma' = \emptyset$.
\item[(H2)]
$\Sigma \supset S_{\infty}(F)$.
\end{itemize}
Then there exists a finite $\Z[G]^-$-module $\Omega_{\Sigma}^{\Sigma'}$ such that we have an exact sequence
\begin{equation}\label{eq:fund}
0 \to \Cl_H^{\Sigma', -} \to \Omega_{\Sigma}^{\Sigma'} \to \bigoplus_{v \in \Sigma_f} A_v^- \to 0.
\end{equation}
\end{prop}

Note that the sequence \eqref{eq:fund} does not characterize the module $\Omega_{\Sigma}^{\Sigma'}$ up to isomorphisms; we have to determine the extension class associated to \eqref{eq:fund}.
This will be done in \S \ref{ss:ext}.

The following is an important property of the module $\Omega_{\Sigma}^{\Sigma'}$.

\begin{prop}\label{prop:Omega_ct}
In addition to (H1) and (H2), let us suppose the following.
\begin{itemize}
\item[(H3)]
$\mu(H)^{\Sigma', -}$ vanishes.
\item[(H4)]
$\Sigma \cup \Sigma' \supset S_{\ram}(H/F)$ and $\Sigma \supset S_{\wild}(H/F)$.
\end{itemize}
Then we have $\pd_{\Z[G]^-}(\Omega_{\Sigma}^{\Sigma'}) \leq 1$, where in general $\pd$ denotes the projective dimension.
In other words, the $G$-module $\Omega_{\Sigma}^{\Sigma'}$ is cohomologically trivial.
\end{prop}

This proposition will be refined in Proposition \ref{prop:Omega_p_ct} below, which in turn will be proved in \S \ref{ss:ct}.

\subsubsection{A Stickelberger type element}\label{sss:Stickel}

We introduce an analytic element $\theta_{\Sigma}^{\Sigma'}$ that should correspond to $\Omega_{\Sigma}^{\Sigma'}$ under the eTNC$^-$.

For a finite prime $v$ of $F$, we define $N(v)$ as the cardinality of the residue field of $F$ at $v$.
For a $\C$-valued character $\psi$ of $G$, we define the $L$-function $L(\psi, s)$ by
\[
L(\psi, s) 
= \prod_{v \nmid \ff_{\psi}} \parenth{1 - \frac{\psi(v)}{N(v)^s}}^{-1},
\]
where $v$ runs over the finite primes of $F$ that do not divide the conductor $\ff_{\psi}$ of $\psi$.
It is well-known that this infinite product converges absolutely for 
$s \in \C$ whose real part is larger than $1$,
and has an analytic continuation to the whole complex plane $s \in \C$.
We define an element $\omega \in \Q[G]^-$ by
\[
\omega = \sum_{\psi} L(\psi^{-1}, 0) e_{\psi},
\]
where $\psi$ runs over the odd characters of $G$ and $e_{\psi}$ denotes the idempotent associated to $\psi$ defined by
\[
e_{\psi} = \frac{1}{\# G} \sum_{\sigma \in G} \psi(\sigma) \sigma^{-1}.
\]
We actually have $\omega \in \Q[G]^-$, thanks to the Siegel--Klingen theorem.
In other words, the element $\omega$ is characterized by $\psi(\omega) = L(\psi^{-1}, 0)$ for any odd character $\psi$, where by abuse of notation we write $\psi$ for the induced $\Q$-algebra homomorphism from $\Q[G]^-$ to $\C$.

For a finite set $\Sigma'$ of finite primes of $F$, we also define the ``smoothed'' $L$-value
\[
L^{\Sigma'}(\psi, 0) 
= \prod_{v \in \Sigma', v \nmid \ff_{\psi}} \parenth{1 - N(v) \psi(v)} \cdot L(\psi, 0)
\]
and, accordingly, an element
\begin{equation}\label{eq:omega^S}
\omega^{\Sigma'} = \sum_{\psi} L^{\Sigma'}(\psi^{-1}, 0) e_{\psi} \in \Q[G]^-.
\end{equation}

For each finite prime $v$ of $F$, we define a non-zero-divisor $h_v \in \Q[G]$ by 
\begin{equation}\label{eq:hv}
h_v 
= 1 - \frac{\nu_{I_v}}{\# I_v} (\varphi_v^{-1} - \# I_v).
\end{equation}
Here, in general, for a finite group $N$, we write $\nu_N = \sum_{\sigma \in N} \sigma$ for the norm element in a group ring.
Then, in our case, $\nu_{I_v}$ induces a well-defined map $\Q[G/I_v] \to \Q[G]$, which is implicit in the definition of $h_v$.
This element $h_v$ dates back to Greither \cite{Gre07}, and used by Kurihara \cite{Kur20} and the authors \cite{AK21} for computing the Fitting ideals of (duals of) class groups.

Now, for finite sets $\Sigma$ and $\Sigma'$ satisfying the conditions (H1) and (H2), 
we define the Stickelberger element by
\begin{equation}\label{eq:theta_defn}
\theta_{\Sigma}^{\Sigma'} = 
\prod_{v \in \Sigma_f} h_v ^-
\cdot \omega^{\Sigma'} \in \Q[G]^-.
\end{equation}
We will remark on the integrality of this element just after Conjecture \ref{conj:main} below.

\subsubsection{The formulation of the eTNC}\label{sss:eTNC}

Now we can formulate the eTNC$^-$ as follows.
In general, for a commutative noetherian ring $R$, let $\Fitt_R(-)$ denote the initial Fitting ideals for finitely generated $R$-modules.

\begin{conj}[The eTNC$^-$]\label{conj:main}
Let $H/F$ be a finite abelian CM-extension.
Let $\Sigma$ and $\Sigma'$ be finite sets of places of $F$ satisfying (H1), (H2), (H3), and (H4). 
Then we have
\[
\Fitt_{\Z[G]^-}(\Omega_{\Sigma}^{\Sigma'}) 
= (\theta_{\Sigma}^{\Sigma'})
\]
as ideals of $\Z[G]^-$.
\end{conj}

This conjecture implicitly assumes the integrality of $\theta_{\Sigma}^{\Sigma'}$.
Thanks to the work of Deligne and Ribet \cite{DR80} or Cassou-Nogu\`{e}s \cite{CN79}, we have the integrality as long as we slightly strengthen the condition (H3) (see \S \ref{ss:int_Stickel}).
However, unfortunately it seems hard to directly deduce the integrality under the current hypotheses (H1), (H2), (H3), and (H4).

The statement of this conjecture is independent from the choice of the pair $(\Sigma, \Sigma')$ (see \S \ref{sss:Sigma}).
It is clear that Conjecture \ref{conj:main} can be divided into $p$-components for odd prime numbers $p$:
\[
\Fitt_{\Z_p[G]^-}(\Z_p \otimes_{\Z} \Omega_{\Sigma}^{\Sigma'}) 
= (\theta_{\Sigma}^{\Sigma'})
\]
as ideals of $\Z_p[G]^-$.
We refer to this equality as the eTNC$_p^-$.

This formulation of the eTNC$^-$ is essentially what the authors' preceding work \cite{AK21} used.
In fact, the main content of \cite{AK21} is to compute the Fitting ideal of $\Cl_H^{\Sigma', -}$, using the sequence \eqref{eq:fund} and assuming the eTNC$^-$.
Therefore, Theorem \ref{thm:current_main} directly yields a lot of unconditional cases of the main result of \cite{AK21}.
 The equivalence between Conjecture \ref{conj:main} and a more standard formulation of the eTNC will be shown in \S \ref{Ap:eTNC}.

Now we have explained the statement of Theorem \ref{thm:current_main}.
Here we give two remarks on the theorem.
 
\begin{rem}\label{rem:Wiles_cond}
We have a simple equivalent condition for the condition (iii) in Theorem \ref{thm:current_main}.
Let $M_p$ be the maximal $2$-extension of $\Q$ in $\Q(\mu_p)$.
Then, for a CM-field $K$, we have $K^+(\mu_p) \supset K$ if and only if $K \supset M_p$.

This is shown as follows.
Since $[\Q(\mu_p): M_p]$ is odd and $[K:K^+] = 2$, we have $K^+(\mu_p) \supset K$ if and only if $K^+ M_p \supset K$.
Since $M_p/\Q$ is a cyclic extension, so is $K^+ M_p/K^+$, and thus the unique intermediate CM-field of $K^+ M_p/K^+$ is $K^+ M_p$.
Therefore, $K^+ M_p \supset K$ is equivalent to $K^+ M_p = K$, that is, $M_p \subset K$.

Note that this argument also shows the following:
Given a totally real field $k$, there exists at most one CM-field $K$ such that $K^+ = k$ and $K^+(\mu_p) \supset K$; in fact, the unique possible choice is $K = k M_p$ when $k \supset M_p^+$.
For this reason, the condition (iii) is very mild.
\end{rem}

\begin{rem}
Let us illustrate the result of Theorem \ref{thm:current_main} when $F = \Q$.
As already remarked, Burns and Greither \cite{BuGr03} have unconditionally proved the eTNC$^-$ for the case $F = \Q$.
Though Theorem \ref{thm:current_main} shows the validity of the eTNC$_p^-$ for quite a lot of cases as in Remark \ref{rem:Wiles_cond}, unfortunately, we cannot recover the full statement.
For instance, consider the case where $H = \Q(\mu_{m p^n})$ with $n \geq 2$ and $p \nmid m$ with $m \geq 3$
such that $p$ splits completely in $\Q(\mu_m)/\Q$.
Then none of the conditions (i)(ii)(iii) hold.
Indeed, (i) the prime $p$ is wildly ramified in $H/\Q$; (ii) the concerned intermediate field is $\Q(\mu_m)$, which is a CM-field; (iii) $H$ contains $M_p$ introduced in Remark \ref{rem:Wiles_cond}.
\end{rem}

Even without assuming (i), (ii), or (iii) of Theorem \ref{thm:current_main},
we also have the following second main theorem.
For each subgroup $N$ of $G$, using the norm element $\nu_N$ of $N$, we define a $\Z_p[G]$-algebra $\Z_p[G]_{(N)}$ by
\begin{equation}\label{eq:ZpGN}
\Z_p[G]_{(N)} = \Z_p[G/N] \times \Z_p[G]/(\nu_{N}),
\end{equation}
which can be naturally regarded as a subring of $\Q_p[G]$ that contains $\Z_p[G]$.
We put $I_p = \sum_{\fp \in S_p(F)} I_{\fp} \subset G$.

\begin{thm}\label{thm:main_p_2}
Let $H/F$ be a finite abelian CM-extension.
Let $p$ be an odd prime number.
Let $\Sigma$ and $\Sigma'$ be finite sets of places of $F$ satisfying 
(H1), (H2), (H3), and  (H4). 
Let $N$ be a subgroup of $I_p$ that contains the $p$-Sylow subgroup of $I_{\fp}$ for some $\fp \in S_p(F)$.
Then we have
\[
\parenth{\Fitt_{\Z_p[G]^-}(\Z_p \otimes_{\Z} \Omega_{\Sigma}^{\Sigma'})} \cdot \Z_p[G]_{(N)}^-
= \theta_{\Sigma}^{\Sigma'} \cdot \Z_p[G]_{(N)}^-
\]
as ideals of $\Z_p[G]_{(N)}^-$.
In other words, the eTNC$_p^-$ holds if we enlarge the coefficient ring from $\Z_p[G]^-$ to $\Z_p[G]_{(N)}^-$.
\end{thm}

Both Theorems \ref{thm:current_main} and \ref{thm:main_p_2} will be deduced from Theorem \ref{thm:current_main_p} below in \S \ref{ss:pf_thm}.

\subsection{A finer theorem for character-components}\label{ss:eTNCp}

We fix an odd prime number $p$.
In this subsection, we state a finer main theorem (Theorem \ref{thm:current_main_p}) concerning the character-components of the eTNC$_p^-$.

We introduce some general notation.
For a finite abelian group $G$, let $G'$ denote the maximal subgroup of $G$ of order prime to $p$.
For a character $\chi$ of $G'$, put $\OO_{\chi} = \Z_p[\Imag(\chi)]$ regarded as a $\Z_p[G]$-algebra via $\chi$, and let $\Z_p[G]^{\chi} = \Z_p[G] \otimes_{\Z_p[G]} \OO_{\chi}$ be the $\chi$-component of $\Z_p[G]$.
For a finitely generated $\Z_p[G]$-module $M$, 
we put $M^{\chi} = \Z_p[G]^{\chi} \otimes_{\Z[G]} M$, which is regarded as a $\Z_p[G]^{\chi}$-module.
Even more generally, if $M$ is a $\Z[G]$-module whose $p$-adic completion $\widehat{M}$ is finitely generated over $\Z_p[G]$, we define $M^{\chi} = (\widehat{M})^{\chi}$.
For each $x \in M$, we write $x^{\chi}$ for the image of $x$ by the natural map $M \to M^{\chi}$.

Now we return to a finite abelian CM-extension $H/F$ and its Galois group $G = \Gal(H/F)$.
Let $G'$ denote the maximal subgroup of $G = \Gal(H/F)$ whose order is prime to $p$.
We also put $I_v' = I_v \cap G'$ for each finite prime $v$ of $F$.

Let $\Sigma$ and $\Sigma'$ be finite sets of places of $F$ satisfying the conditions (H1) and (H2).
When we focus on the $\chi$-components for a given odd character $\chi$ of $G'$,
the conditions (H3) and (H4) are relaxed to the following.
\begin{itemize}
\item[(H3)$_p^{\chi}$]
$\mu_{p^{\infty}}(H)^{\Sigma', \chi}$ vanishes,
where $\mu_{p^{\infty}}(H)^{\Sigma'}$ denotes the $p$-primary component of $\mu(H)^{\Sigma'}$.
\item[(H4)$_p^{\chi}$]
$\Sigma_f \cup \Sigma' \supset S_{\bad}^{\chi}$ and $\Sigma_f \supset S_{\bad}^{\chi} \cap S_p(F)$.
Here, we write $S_{\bad}^{\chi}$ for the set of finite primes $v$ of $F$ such that $p \mid \# I_v$ and $\chi$ is trivial on $I_v'$.
\end{itemize}
It is easy to see that (H3) (resp.~(H4)) implies (H3)$_p^{\chi}$ (resp.~(H4)$_p^{\chi}$).

Now, under the conditions (H1) and (H2), we have a $\Z_p[G]^{\chi}$-module $\Omega_{\Sigma}^{\Sigma', \chi} = (\Z_p \otimes_{\Z} \Omega_{\Sigma}^{\Sigma'})^{\chi}$.
The following is a refinement of Proposition \ref{prop:Omega_ct}.
It will be proved in \S \ref{ss:ct}.

\begin{prop}\label{prop:Omega_p_ct}
Let $\chi$ be an odd character of $G'$.
Let $\Sigma$ and $\Sigma'$ be finite sets of places of $F$ satisfying (H1), (H2), (H3)$_p^{\chi}$, and (H4)$_p^{\chi}$.
Then we have $\pd_{\Z_p[G]^{\chi}}(\Omega_{\Sigma}^{\Sigma', \chi}) \leq 1$.
\end{prop}

By Proposition \ref{prop:Omega_p_ct},
 as long as the conditions (H1), (H2), (H3)$_p^{\chi}$, and (H4)$_p^{\chi}$ hold, we naturally expect the $\chi$-component of the eTNC$_p^-$ (see Conjecture \ref{conj:main}): 
\[
\Fitt_{\Z_p[G]^{\chi}}(\Omega_{\Sigma}^{\Sigma', \chi}) 
= (\theta_{\Sigma}^{\Sigma', \chi})
\]
as ideals of $\Z_p[G]^{\chi}$.
(Recall that the integrality of the Stickelberger element is not available in general; see \S \ref{ss:int_Stickel}.)

%

Now we can state the refined version of the main theorem of this paper, from which Theorems \ref{thm:current_main} and \ref{thm:main_p_2} will follow (see \S \ref{ss:pf_thm}).
As before, we put $I_p = \sum_{\fp \in S_p(F)} I_{\fp} \subset G$ and define $\nu_{I_p} \in \Z_p[G]$ as the norm element.

\begin{thm}\label{thm:current_main_p}
Let $H/F$ be a finite abelian CM-extension.
Let $p$ be an odd prime number.
Let $G'$ denote the maximal subgroup of $G = \Gal(H/F)$ of order prime to $p$.
Let $\chi$ be an odd character of $G'$ and put $H^{\chi} = H^{\Ker(\chi)}$.
Let $\Sigma$ and $\Sigma'$ be finite sets of places of $F$ satisfying (H1), (H2), (H3)$_p^{\chi}$, 
and (H4)$_p^{\chi}$. 
Then the following are true.
\begin{itemize}
\item[(1)]
Unless the decomposition groups of all the $p$-adic primes in $H^{\chi}/F$ are $p$-groups and all the $p$-adic primes are ramified in $H^\chi/F$,
we have
\[
\Fitt_{\Z_p[G]^{\chi}}(\Omega_{\Sigma}^{\Sigma', \chi}) 
\subset (\theta_{\Sigma}^{\Sigma' , \chi})_{\Z_p[G]^{\chi}}
\]
as principal (fractional) ideals of $\Z_p[G]^{\chi}$.
\item[(2)]
In any case, we have
\[
\Fitt_{\Z_p[G]^{\chi}}(\Omega_{\Sigma}^{\Sigma', \chi}) 
\subset (\theta_{\Sigma}^{\Sigma', \chi})_{\Z_p[G]^{\chi}} + (\nu_{I_p}^{\chi})_{\Frac(\Z_p[G]^{\chi})}
\]
in the total fraction ring $\Frac(\Z_p[G]^{\chi})$ of $\Z_p[G]^{\chi}$.
\end{itemize}
\end{thm}

Although this theorem concerns only the single inclusion, it is a standard fact that the analytic class number formula allows us to deduce the equality from it.
See \S \ref{ss:divisibility} for the details.

The existence of $\nu_{I_p}$ in the claim (2) is exactly why the full statement of the eTNC$_p^-$ cannot be deduced.
The reason why we need the element $\nu_{I_p}$ will be explained in \S \ref{ss:outline}.
Nevertheless, from the claim (2) we can deduce Theorem \ref{thm:current_main} (iii).
For that purpose, we will adapt the ``avoiding trivial zero'' argument of Wiles \cite{Wil90}.

\subsection{An outline of the proof and the organization of this paper}\label{ss:outline}

As already mentioned, the proof of Theorem \ref{thm:current_main_p} is very much inspired by the work \cite{DK20} of Dasgupta and Kakde on the Brumer--Stark conjecture.

Before explaining the method, let us explain a crucial difference between the eTNC$_p^-$ and the work \cite{DK20}.
The main theorem \cite[Theorem 3.3]{DK20} of that paper states the equality
\begin{equation}\label{eq:DK_main}
\Fitt_{\Z_p[G]^-}(\nabla_{\Sigma}^{\Sigma'}(H)_p^-)
= (\Theta_{\Sigma}^{\Sigma'})
\end{equation}
as ideals of $\Z_p[G]^-$, where $\nabla_{\Sigma}^{\Sigma'}(H)$ is the transpose Selmer module (studied in \cite{BKS16}) and $\Theta_{\Sigma}^{\Sigma'}$ is a Stickelberger element (we do not review the precise definitions here; cf.~\S \ref{Ap:eTNC}).
The formula \eqref{eq:DK_main} is indeed of the same form as the eTNC$_p^-$ formulated as Conjecture \ref{conj:main}.
However, on the algebraic side, the module $\nabla_{\Sigma}^{\Sigma'}(H)_p^-$ may be infinite in general, while our module $\Omega_{\Sigma}^{\Sigma'}$ is always finite.
Correspondingly, on the analytic side, the element $\Theta_{\Sigma}^{\Sigma'}$ may be a zero-divisor, while our element $\theta_{\Sigma}^{\Sigma'}$ is always a non-zero-divisor.
It is even possible that both sides of \eqref{eq:DK_main} are zero.
Nevertheless, Dasgupta and Kakde used \eqref{eq:DK_main} (for various intermediate extensions of $H/F$) in an ingenious way to deduce the Brumer--Stark conjecture.
However, there seems to be little hope to deduce the eTNC$_p^-$ from \eqref{eq:DK_main}.

Generally speaking, we have two different strategies to prove relationships (e.g., the eTNC and the Iwasawa main conjecture) between algebraic objects and analytic objects.
One is the theory of Euler systems.
For instance, the Iwasawa main conjecture over $\Q$ can be proved using the Euler system of cyclotomic units, which led to the proof by Burns and Greither \cite{BuGr03} of the eTNC when the base field is $\Q$.
The theory of Euler systems has been developed for decades by many contributors.
The other general strategy has its origin in the work of Ribet \cite{Rib76} on the converse of the Herbrand theorem.
A key idea is to construct a suitable cuspform and then study the associated Galois representation.
The strategy is sophisticated by a lot of subsequent work, and as a milestone Wiles \cite{Wil90} proved the Iwasawa main conjecture over totally real fields.
It is also well-known that these two strategies give the opposite divisibilities between algebraic objects and analytic objects.
When we are dealing with $\G_m$, thanks to the analytic class number formula, a single divisibility is enough for the equality.

This latter strategy is the basis of the work of Dasgupta and Kakde \cite{DK20}.
They construct a suitable cuspform and study the associated Galois representation.
A key point we need to care about is that we have to work over a group ring like $\Z_p[G]$, not over a domain like $\C$.

Now let us explain the organization of this paper and at the same time the outline the proof of Theorem \ref{thm:current_main_p}.
It closely follows \cite{DK20}, but both the Ritter--Weiss type modules and the Stickelberger elements are different, so we need appropriate modifications throughout.

\begin{itemize}
\item[\S \S \ref{s:Omega} -- \ref{s:ded}.]
These sections are essentially preliminaries concerning the formulation of the eTNC$^-$ and the statements of the main theorems of this paper.
Firstly, \S \ref{s:Omega} is devoted to the construction and the proof of basic properties of $\Omega_{\Sigma}^{\Sigma'}$.
Then in \S \ref{s:func}, we show functorial properties of those modules.
They enable us to show the independency of the eTNC$^-$ from the choice of $(\Sigma, \Sigma')$, and also allow us to reduce the proof of the main theorems to the case of the faithful odd characters $\chi$.
In \S \ref{ss:int_Stickel}, we study the integrality of the Stickelberger element $\theta_{\Sigma}^{\Sigma'}$.
In \S \S \ref{ss:divisibility} -- \ref{ss:pf_thm}, we deduce Theorems \ref{thm:current_main} and \ref{thm:main_p_2} from Theorem \ref{thm:current_main_p}.
In \S \ref{Ap:eTNC}, we show that our eTNC$^-$ is equivalent to more standard formulations.
\item[\S \S \ref{s:Eis}--\ref{s:cuspform}.]
The remaining sections are devoted to the proof of Theorem \ref{thm:current_main_p} (for a particular choice of $(\Sigma, \Sigma')$ and faithful characters $\chi$).
In \S \ref{s:Eis}, after reviewing basic notation on Hilbert modular forms, we introduce the Eisenstein series and modify them so that the relevant $L$-values appear in the constant terms.
Then in \S \ref{s:cuspform}, using the Eisenstein series and applying the work of Silliman \cite{Sil20}, we construct an appropriate cuspform $F_k(\bpsi)$.
\item[\S \S \ref{ss:hom_Hecke}--\ref{ss:strategy}.]
Using the Galois representation associated to the cuspform $F_k(\bpsi)$, we define $\Z_p[G]^{\chi}$-modules $\ol{B_p}$, $\ol{B_0}$, and $\ol{B_1}$ that fit into an exact sequence
\[
0 \to \ol{B_0} \to \ol{B_p} \to \ol{B_1} \to 0.
\]
Compared to the work \cite{DK20}, the construction needs an additional idea because the Hecke action on the cuspform is much more complicated.
\item[\S \ref{ss:abc}.]
We construct a commutative diagram of $\Z_p[G]^{\chi^{-1}}$-modules
\[
\xymatrix{
0 \ar[r]
& \Cl_H^{\Sigma', \chi^{-1}} \ar[r] \ar@{->>}[d]
& \Omega_{\Sigma}^{\Sigma', \chi^{-1}} \ar[r] \ar@{->>}[d]
& \bigoplus_{v \in \Sigma_f} A_v^{\chi^{-1}} \ar[r] \ar@{->>}[d]
& 0\\
0 \ar[r]
& \ol{B_0}^{\sharp} \ar[r]
& \ol{B_p}^{\sharp} \ar[r]
& \ol{B_1}^{\sharp} \ar[r]
& 0,
}
\]
which corresponds to \cite[Equation (125)]{DK20}.
Here, the upper sequence is the $\chi^{-1}$-component of \eqref{eq:fund}, and in the lower sequence
the superscript $\sharp$ denotes the twist by the involution of the group ring that inverts every group element.
This diagram is constructed by comparing the extension classes of the both sequences.
As a consequence, we obtain an inclusion $\Fitt_{\Z_p[G]^{\chi^{-1}}}(\Omega_{\Sigma}^{\Sigma', \chi^{-1}}) \subset \Fitt_{\Z_p[G]^{\chi^{-1}}}(\ol{B_p}^{\sharp})$.
\item[\S \ref{ss:Fitt_B}.]
We show that $\Fitt_{\Z_p[G]^{\chi}}(\ol{B_p})$ is contained in the right hand side of the claims of Theorem \ref{thm:current_main_p}.
This is accomplished by a hard computation based on the idea of \cite{DK20}.
This is where we need the norm element $\nu_{I_p}$ in general.
The origin of $\nu_{I_p}$ is Theorem \ref{thm:cuspform} on $F_k(\bpsi)$; it is multiplied to a weight $k$ modular form.
In this step, we have to compute Fourier coefficients of $F_k(\bpsi)$ (after a certain Hecke action), and the authors found it difficult to manage the coefficients of the weight $k$ modular form.
\end{itemize}

\section*{Acknowledgments}\label{s:ack}

We are sincerely grateful to Masato Kurihara for his encouragements.
We also thank Andreas Nickel for informative correspondences.
The second author is supported by JSPS KAKENHI Grant Number 19J00763.

\section{Construction of the Ritter--Weiss type module}\label{s:Omega}

In this section, we construct the module $\Omega_{\Sigma}^{\Sigma'}$ satisfying Propositions \ref{prop:Omega}, \ref{prop:Omega_ct}, and \ref{prop:Omega_p_ct}.
The construction closely follows Kurihara \cite{Kur20} and the work \cite{AK21} of the authors, which in turn relies on work of Ritter and Weiss \cite{RW96} and Greither \cite{Gre07}.
The construction is more or less standard and has nothing essentially new (we may also refer to the exposition of Dasgupta and Kakde \cite[\S A]{DK20}).
However, the actual construction is necessary to obtain a concrete expression of the extension class associated to the sequence \eqref{eq:fund} as in \S \ref{ss:ext}, and also to show functorial properties in \S \ref{s:func}.
This is why we include a complete construction of $\Omega_{\Sigma}^{\Sigma'}$ in this section.

In \S \ref{ss:translation}, we review Gruenberg's translation functor \cite[\S 10.5]{Gru76}.
Then in \S \ref{ss:local} and \S \ref{ss:global}, following the work \cite{RW96} by Ritter and Weiss, we review the construction of local and global arithmetic modules by applying the translation functor to the fundamental classes in class field theory.
Using certain local and global compatibility that we review in \S \ref{ss:local_global}, we construct the module $\Omega_{\Sigma}^{\Sigma'}$ satisfying Proposition \ref{prop:Omega} in \S \ref{ss:Omega_constr}.
We prove Proposition \ref{prop:Omega_p_ct} (and thus Proposition \ref{prop:Omega_ct}) on the cohomological triviality in \S \ref{ss:ct}.
Finally in \S \ref{ss:ext} we obtain a concrete expression of the extension class associated to the sequence \eqref{eq:fund}.

\subsection{Gruenberg's translation functor}\label{ss:translation}

In this subsection, we review Gruenberg's translation functor.
Let $G$ be any finite group (in the applications we only need the case where $G$ is abelian).

Let $A$ be a (left) $\Z[G]$-module.
It is well-known that the cohomology group $H^2(G, A)$ can be identified with the set of equivalence classes of group extensions
\begin{equation}\label{eq:group_ext}
1 \to A \to X \to G \to 1
\end{equation}
which are compatible with the $G$-module structure on $A$.
Here, the last restraint means that the $G$-module structure on $A$ coincides with the structure defined by $g \cdot a = \wtil{g} a \wtil{g}^{-1}$ for $g \in G$ and $a \in A$, where $\wtil{g} \in X$ is a lift of $g$.

In general we define the augmentation ideal by $\Delta G = \Ker(\Z[G] \to \Z)$, so we have an exact sequence
\begin{equation}\label{eq:aug}
0 \to \Delta G \to \Z[G] \to \Z \to 0.
\end{equation}
It is also well-known that $\Ext^1_{\Z[G]}(\Delta G, A)$ can be identified with the set of the equivalence classes of $\Z[G]$-module extensions
\begin{equation}\label{eq:module_ext}
0 \to A \to Y \to \Delta G \to 0.
\end{equation}

Then Gruenberg's translation functor \cite[\S 10.5]{Gru76} (see also \cite[\S 2]{RW96}) can be formally defined as follows.

\begin{defn}\label{defn:translation}
For a finite group $G$ and a $\Z[G]$-module $A$, we have natural isomorphisms
\[
\bft: H^2(G, A) \simeq \Ext^2_{\Z[G]}(\Z, A) \simeq \Ext^1_{\Z[G]}(\Delta G, A),
\]
where the second isomorphism is induced by the exact sequence \eqref{eq:aug}.
We regard $\bft$ as a functor which sends group extensions \eqref{eq:group_ext} to module extensions \eqref{eq:module_ext}, and call $\bft$ the translation functor.
\end{defn}

We will need the following concrete expression of the translation functor.
See \cite[\S 2, Lemma 3]{RW96} for a proof; actually the following description is usually used as the definition of the translation functor, and the equivalence with the above definition is proved in the article.



\begin{prop}\label{prop:translation}
Let $A$ be a $\Z[G]$-module.
Let 
\begin{equation}\label{eq:trans_1}
1 \to A \overset{\alpha}{\to} X \overset{\beta}{\to} G \to 1
\end{equation}
 be a group extension which is compatible with the $G$-module structure on $A$.
We define a $\Z[G]$-module $Y$ by
\[
Y = \Z[G] \otimes_{\Z[X]} \Delta X.
\]
Here, the right $\Z[X]$-module structure on $\Z[G]$ is defined by $\beta$.
Then we have an exact sequence
\begin{equation}\label{eq:trans_2}
0 \to A \overset{\alpha'}{\to} Y \overset{\beta'}{\to} \Delta G \to 0,
\end{equation}
where we define $\alpha'$ by $\alpha'(a) = 1 \otimes (\alpha(a) - 1)$ for $a \in A$, and $\beta'$ as the induced map from $\Delta X \to \Delta G$ induced by $\beta$.
Moreover, the translation functor $\bft$ sends the extension \eqref{eq:trans_1} to the extension \eqref{eq:trans_2}.
\end{prop}


\subsection{Local consideration}\label{ss:local}

Let $H_w/F_v$ be a finite Galois extension of non-archimedean local fields of mixed characteristic.
Let $G_w = \Gal(H_w/F_v)$ be the Galois group and $I_w \subset G_w$ the inertia group.
Of course we have an extension $H/F$ of number fields in mind, but in this subsection we deal with only local fields.


By local class field theory, we have the local fundamental class denoted by
\[
u_{H_w/F_v} \in H^2(G_w, H_w^{\times}).
\]
Let
\[
\ol{u}_{H_w/F_v} \in H^2(G_w, \Z).
\]
be the image of $u_{H_w/F_v}$ by the homomorphism induced by the normalized valuation map $\ord_w: H_w^{\times} \to \Z$.

\begin{defn}\label{defn:V_W}
By applying the translation functor $\bft$ to the local fundamental class $u_{H_w/F_v}$, we define a $\Z[G_w]$-module $V_w$ which fits in a $\Z[G_w]$-module extension 
\begin{equation}\label{eq:V}
0 \to H_w^{\times} \to V_w \to \Delta G_w \to 0
\end{equation}
that represents $\bft(u_{H_w/F_v}) \in \Ext^1_{\Z[G_w]}(\Delta G_w, H_w^{\times})$.
Similarly, we define a $\Z[G_w]$-module $W_w$ with an exact sequence
\begin{equation}\label{eq:W}
0 \to \Z \to W_w \to \Delta G_w \to 0
\end{equation}
that represents $\bft(\ol{u}_{H_w/F_v}) \in \Ext^1_{\Z[G_w]}(\Delta G_w, \Z)$.
\end{defn}

By the definition, the exact sequences \eqref{eq:V} and \eqref{eq:W} fit in a commutative diagram with exact rows and columns
\begin{equation}\label{eq:VW}
\xymatrix{
& U_{H_w} \ar@{=}[r] \ar@{^{(}->}[d] &
U_{H_w} \ar@{^{(}->}[d] & &\\
0 \ar[r] &
H_w^{\times} \ar[r] \ar@{->>}[d] &
V_w \ar[r] \ar@{->>}[d] &
\Delta G_w \ar[r] \ar@{=}[d] &
0 \\
0 \ar[r] &
\Z \ar[r] &
W_w \ar[r] &
\Delta G_w \ar[r] &
0.
}
\end{equation}
Here, $U_{H_w}$ denotes the unit group of the local field $H_w$.

For a more concrete description of $W_w$, we introduce the Weil groups.
Let $H_w^{\ur}$ be the maximal unramified extension of $H_w$.
Let $W(H_w^{\ur}/H_w) \subset \Gal(H_w^{\ur}/H_w)$ be the Weil group, that is, the subgroup generated by the arithmetic Frobenius $\varphi_w$.
We similarly define the Weil group $W(F_v^{\ur}/F_v) \subset \Gal(F_v^{\ur}/F_v)$ generated by $\varphi_v$.

We then define the Weil group of the extension $H_w^{\ur}/F_v$ by
\[
W(H_w^{\ur}/F_v) = \{\sigma \in \Gal(H_w^{\ur}/F_v) \mid \ol{\sigma} \in W(F_v^{\ur}/F_v)\},
\]
where $\ol{\sigma}$ denotes the natural image of $\sigma$ in $\Gal(F_v^{\ur}/F_v)$.
Then the natural restriction homomorphism $\Gal(H_w^{\ur}/F_v) \to \Gal(F_v^{\ur}/F_v) \times G_w$ induces the first isomorphism of
\begin{align}\label{eq:Weil_str}
W(H_w^{\ur}/F_v)
& \simeq \{(h, g) \in W(F_v^{\ur}/F_v) \times G_w \mid \text{$\ol{h} = \ol{g}$ in $G_w/I_w$}\}\\
& \simeq \{(n, g) \in \Z \times G_w \mid \text{$\varphi_v^n = \ol{g}$ in $G_w/I_w$}\},
\end{align}
where the overlines denote the images to $G_w/I_w$, but we simply write $\varphi_v$ for $\ol{\varphi_v} \in G_w/I_w$ by abuse of notation.

We now have the following fact.

\begin{prop}\label{prop:W_Weil}
The group extension 
\[
1 \to W(H_w^{\ur}/H_w) \to W(H_w^{\ur}/F_v) \to G_w \to 1,
\]
which is obtained as the restriction of the natural exact sequence $1 \to \Gal(H_w^{\ur}/H_w) \to \Gal(H_w^{\ur}/F_v) \to G_w \to 1$, 
represents the class $\ol{u}_{H_w/F_v} \in H^2(G_w, \Z)$ (we identify $W(H_w^{\ur}/H_w)$ with $\Z$ by sending $\varphi_w$ to $1$).
\end{prop}

By Propositions \ref{prop:translation} and \ref{prop:W_Weil}, we obtain an isomorphism
\begin{equation}\label{eq:W_desc}
W_w \simeq \Z[G_w] \otimes_{\Z[W(H_w^{\ur}/F_v)]} \Delta W(H_w^{\ur}/F_v).
\end{equation}

This description of $W_w$ gives the following more concrete one.

\begin{prop}\label{prop:W_str}
We have an isomorphism of $\Z[G_w]$-modules
\[
W_w \simeq 
\{(x, y) \in \Delta G_w \oplus \Z[G_w/I_w] 
\mid \text{$\ol{x} = (1 - \varphi_v^{-1}) y$ in $\Z[G_w/I_w]$} \},
\]
where $\ol{x}$ denotes the image of $x$ in $\Z[G_w/I_w]$.
Moreover, under this isomorphism, the exact sequence \eqref{eq:W} is described as follows: 
the map $\Z \to W_w$ sends $1$ to $(0, \nu_{G_w/I_w})$, and the map $W_w \to \Delta G_w$ sends $(x, y)$ to $x$.
Here, $\nu_{G_w/I_w} = \sum_{i = 1}^{\#(G_w/I_w)} \varphi_v^{i} \in \Z[G_w/I_w]$ is the norm element.
\end{prop}

\begin{proof}
Let us write $W_w'$ for the right hand side of the displayed isomorphism.
We will construct a commutative diagram with exact rows of the form
\[
\xymatrix{
	0 \ar[r]
	& \Z \ar[r] \ar@{=}[d]
	& W_w \ar[r] \ar@{..>}[d]
	& \Delta G_w \ar[r] \ar@{=}[d]
	& 0\\
	0 \ar[r]
	& \Z \ar[r]
	& W_w' \ar[r]
	& \Delta G_w \ar[r]
	& 0.
}
\]
Here, the upper sequence is \eqref{eq:W}, and the lower sequence is the one described in the statement.
It is easy to see that the lower sequence is also exact.
Then it is enough to construct the dotted arrow which makes the diagram commutative.

Using the identification \eqref{eq:Weil_str}, 
we define a $\Z$-homomorphism
\[
\Delta W(H_w^{\ur}/F_v) \to \Delta G_w \oplus \Z[G_w/I_w]
\]
 by sending $(n, g) - 1$ to $(g - 1, \sum_{i=1}^{n} \varphi_v^{i})$ if $n > 0$, to $(g - 1, 0)$ if $n = 0$, and to $(g - 1, - \sum_{i=0}^{-n - 1} \varphi_v^{-i})$ if $n < 0$ (in any case the second component is defined as a formal expansion of $(\varphi_v^{n} - 1)/(1 - \varphi_v^{-1})$).
It can directly checked that this is actually a $\Z[W(H_w^{\ur}/F_v)]$-homomorphism and the image is contained in $W_w'$.
Thus, by \eqref{eq:W_desc}, a $\Z[G_w]$-homomorphism $W_w \to W_w'$ is induced.

Let us check that the diagram is commutative.
By the description in Proposition \ref{prop:translation}, the map $\Z \to W_w$ sends $1$ to $\varphi_w - 1 = (\#(G_w/I_w), \id_{G_w}) - 1$ since $\varphi_w = \varphi_v^{\# (G_w/I_w)}$.
Then this is sent to $(0, \sum_{i = 1}^{\#(G_w/I_w)} \varphi_v^{i}) = (0, \nu_{G_w/I_w})$ in $W_w'$, so the left square is commutative.
Each element $(n, g) - 1 \in \Delta W(H_w^{\ur}/F_v)$ is sent by the map $W_w \to W_w' \to \Delta G_w$ to $g - 1$, so the right square is also commutative.
\end{proof}

The following proposition is elementary but of critical importance.

\begin{prop}\label{prop:str_W_2}
We make use of the isomorphism in Proposition \ref{prop:W_str}.
The following are true.
\begin{itemize}
\item[(1)]
If $H_w/F_v$ is unramified, then we have an isomorphism
\[
\iota_w: W_w \simeq \Z[G_w],
\]
which sends $(x, y)$ to $y$.
\item[(2)]
In any case we have an exact sequence
\[
0 \to W_w \overset{f_w}{\to} \Z[G_w] \to \Z[G_w/I_w]/(1 - \varphi_v^{-1} + \# I_w) \to 0,
\]
where $f_w$ sends $(x, y)$ to $x + \nu_{I_w} y$.
\end{itemize}
\end{prop}

\begin{proof}
The claim (1) is easy.
See the paper \cite[Lemma 3.4]{AK21} of the authors for the proof of the claim (2).
\end{proof}

\begin{rem}\label{rem:fw_var}
As a variant of Proposition \ref{prop:str_W_2}(2), given a nonzero integer $c_w \in \Z$, we have an exact sequence
\begin{equation}\label{eq:W_fund_var}
0 \to W_w \overset{\ol{f_w}}{\to} \Z[G_w] \to \Z[G_w/I_w]/(1 - \varphi_v^{-1} + c_w \# I_w) \to 0,
\end{equation}
where $\ol{f_w}$ sends $(x, y)$ to $x + c_w \nu_{I_w} y$.
Moreover, when we work over $\Z_p$ instead of $\Z$, an analogous construction is available for a nonzero $p$-adic integer $c_w \in \Z_p$.
This remark will be used in later in \S \ref{s:func}.
\end{rem}

\subsection{Global consideration}\label{ss:global}

Let $H/F$ be a finite Galois extension of number fields.
Let $G = \Gal(H/F)$ be the Galois group.
The argument in this subsection is in parallel with \S \ref{ss:local}.

Let $C_H$ be the idele class group of $H$.
By global class field theory, we have the global fundamental class
\[
u_{H/F} \in H^2(G, C_H).
\]
Let $\Sigma'$ be a finite set of finite primes of $F$.
Let $\Cl_H^{\Sigma'}$ be the ray class group of $H$ of modulus $\prod_{w \in \Sigma'_H} w$, where $\Sigma'_H$ denotes the set of primes of $H$ that lie above primes in $\Sigma'$.
We define
\[
\ol{u}_{H/F} \in H^2(G, \Cl_H^{\Sigma'})
\]
as the image of $u_{H/F}$ by the homomorphism $H^2(G, C_H) \to H^2(G, \Cl_H^{\Sigma'})$ induced by the natural map $C_H \to \Cl_H^{\Sigma'}$.

Similarly as in the local case, we introduce the following.

\begin{defn}\label{defn:O_H}
By applying the translation functor $\bft$, we define a $\Z[G]$-module $\fO$ which fits in a $\Z[G]$-module extension 
\begin{equation}\label{eq:O}
0 \to C_H \to \fO \to \Delta G \to 0
\end{equation}
that represents $\bft(u_{H/F}) \in \Ext^1_{\Z[G]}(\Delta G, C_H)$.
Similarly, for a finite set $\Sigma'$ of finite primes of $F$, we define a $\Z[G]$-module $\fH^{\Sigma'}$ which fits in an exact sequence
\begin{equation}\label{eq:H}
0 \to \Cl_H^{\Sigma'} \to \fH^{\Sigma'} \to \Delta G \to 0
\end{equation}
that represents $\bft(\ol{u}_{H/F}) \in \Ext^1_{\Z[G]}(\Delta G, \Cl_H^{\Sigma'})$.
\end{defn}

By the definition, the sequences \eqref{eq:O} and \eqref{eq:H} fit in a commutative diagram with exact rows
\[
\xymatrix{
0 \ar[r] &
C_H \ar[r] \ar@{->>}[d] &
\fO \ar[r] \ar@{->>}[d] &
\Delta G \ar[r] \ar@{=}[d] &
0 \\
0 \ar[r] &
\Cl_H^{\Sigma'} \ar[r] &
\fH^{\Sigma'} \ar[r] &
\Delta G \ar[r] &
0.
}
\]

We need more concrete description of $\fH^{\Sigma'}$.
Let $H^{\Sigma', \ab}$ be the ray class field of $H$ of modulus $\prod_{w \in \Sigma'_H} w$, so we have an isomorphism $\Cl_H^{\Sigma'} \simeq \Gal(H^{\Sigma', \ab}/H)$ by the Artin reciprocity map.

We now have the following fact.

\begin{prop}\label{prop:H_Weil}
The group extension $1 \to \Gal(H^{\Sigma', \ab}/H) \to \Gal(H^{\Sigma', \ab}/F) \to G \to 1$ represents the class $\ol{u}_{H/F} \in H^2(G, \Cl_H^{\Sigma'})$.
\end{prop}

By Propositions \ref{prop:translation} and \ref{prop:H_Weil}, we have
\begin{equation}\label{eq:H_desc}
\fH^{\Sigma'} \simeq
\Z[G] \otimes_{\Z[\Gal(H^{\Sigma', \ab}/F)]} \Delta \Gal(H^{\Sigma', \ab}/F).
\end{equation}

\subsection{Compatibility between the local and global diagrams}\label{ss:local_global}

Let $H/F$ be as in \S \ref{ss:global}.
For each finite prime $w$ of $H$, let $I_w \subset G_w \subset G$ be the inertia group and the decomposition group, respectively.
We can apply the results in \S \ref{ss:local} to the extension $H_w/F_v$, where $v$ denotes the prime of $F$ lying below $w$.

Since we have appropriate compatibility between the local and the global fundamental classes,
we obtain a commutative diagram concerning the sequences \eqref{eq:V} and \eqref{eq:O}:
\begin{equation}\label{eq:local_to_global}
\xymatrix{
0 \ar[r] &
H_w^{\times} \ar[r] \ar[d] &
V_w \ar[r] \ar[d] &
\Delta G_w \ar[r] \ar@{^{(}->}[d] &
0 \\
0 \ar[r] &
C_H \ar[r] &
\fO \ar[r] &
\Delta G \ar[r] &
0.
}
\end{equation}
Similarly, for a finite set $\Sigma'$ of finite primes of $F$ with $v \not\in \Sigma'$, we have a commutative diagram concerning the sequences \eqref{eq:W} and \eqref{eq:H}:
\begin{equation}\label{eq:local_to_global2}
\xymatrix{
0 \ar[r] &
\Z \ar[r] \ar[d] &
W_w \ar[r] \ar[d]_{z_w} &
\Delta G_w \ar[r] \ar@{^{(}->}[d] &
0 \\
0 \ar[r] &
\Cl_H^{\Sigma'} \ar[r] &
\fH^{\Sigma'} \ar[r] &
\Delta G \ar[r] &
0.
}
\end{equation}
See \cite[page 169]{RW96} for a three-dimensional diagram which incorporates \eqref{eq:local_to_global} and \eqref{eq:local_to_global2}.

The following computation will be used later.
Let 
\[
\bpsi: G_F \twoheadrightarrow G \hookrightarrow \Z[G]^{\times}
\]
be the tautological character.
Here, $G_F = \Gal(\ol{F}/F)$ denotes the absolute Galois group of $F$.
We also write $G_{F_v}$ for the absolute Galois group of $F_v$, and $I_{F_v}$ for its inertia subgroup.
These are regarded as subgroups of $G_F$ by fixing a place of $\ol{F}$ above $w$.

\begin{prop}\label{prop:compute_local_global}
We suppose $v \not \in \Sigma'$.
Let us recall the descriptions of $W_w$ in Proposition \ref{prop:W_str} and of $\fH^{\Sigma'}$ in \eqref{eq:H_desc}.
For each $\sigma \in G_{F}$, we write $\ol{\sigma} \in \Gal(H^{\Sigma', \ab}/F)$ for the image of $\sigma$ by the restriction map.
Then the middle vertical arrow, denoted by $z_w$, of \eqref{eq:local_to_global2} is described as follows.
\begin{itemize}
\item[(1)]
For each $\sigma \in I_{F_v}$,
we have
\[
z_w((\bpsi(\sigma) - 1, 0)) = 1 \otimes (\ol{\sigma} - 1).
\]
\item[(2)]
Let $\wtil{\varphi_v} \in G_{F_v}$ be a lift of the arithmetic Frobenius.
Then we have
\[
z_w((\bpsi(\wtil{\varphi_v}) - 1, \varphi_v))
= 1 \otimes (\ol{\wtil{\varphi_v}} - 1).
\]
\end{itemize}
\end{prop}

\begin{proof}
We first move from the description of Proposition \ref{prop:W_str} to one using \eqref{eq:W_desc}.
Note that, for each $\sigma \in G_{F_v}$, we have $\bpsi(\sigma) \in G_w$.
Then, thanks to the description \eqref{eq:Weil_str}, 
for $\sigma \in I_{F_v}$ as in the claim (1) (resp.~$\wtil{\varphi_v} \in G_{F_v}$ as in the claim (2)),
we have an element $s := (0, \bpsi(\sigma))$ (resp.~$s := (1, \bpsi(\wtil{\varphi_v}))$) of $W(H_w^{\ur}/F_v)$.
Then, by the proof of Proposition \ref{prop:W_str}, we can directly check that the element $1 \otimes (s - 1)$ of $W_w$ using \eqref{eq:W_desc} corresponds to the element $(\bpsi(\sigma) - 1, 0)$ (resp.~$(\bpsi(\wtil{\varphi_v}) - 1, \varphi_v)$).

The map $W(H_w^{\ur}/F_v) \to \Gal(H^{\Sigma', \ab}/F)$ sends $s$ to $\ol{\sigma}$ (resp.~$\ol{\wtil{\varphi_v}}$).
Therefore, the map $z_w$, which is induced by the natural map between the right components of \eqref{eq:W_desc} and \eqref{eq:H_desc}, sends $1 \otimes (s - 1)$ to $1 \otimes (\ol{\sigma} - 1)$ (resp.~$1 \otimes (\ol{\wtil{\varphi_v}} - 1)$).
This completes the proof.
\end{proof}

\subsection{Construction of the Ritter--Weiss type module}\label{ss:Omega_constr}

In this subsection, we construct the module $\Omega_{\Sigma}^{\Sigma'}$ satisfying the exact sequence \eqref{eq:fund} in Proposition \ref{prop:Omega}.
Let $H/F$ be a finite abelian CM-extension and put $G = \Gal(H/F)$.
Let $\Sigma$ and $\Sigma'$ be finite sets of places of $F$ satisfying (H1) and (H2).

Let us take an auxiliary finite set $S'$ of places of $F$ such that
\begin{itemize}
\item
$S' \cap \Sigma' = \emptyset$, 
\item
$S' \supset \Sigma$, 
\item
each element of $S' \setminus \Sigma$ is unramified in $H/F$, 
\item
$\Cl_{H, S'}^{\Sigma'} = 0$, where $\Cl_{H, S'}^{\Sigma'}$ denotes the quotient of $\Cl_H^{\Sigma'}$ by the prime ideals in $(S'_f)_H$, and 
\item
$\bigcup_{v \in S' \setminus \Sigma} G_v = G$.
\end{itemize}
Moreover, we assume that there exists an element $v_0 \in S' \setminus \Sigma$ such that 
\[
G_{v_0} = \Gal(H/H^+).
\]
The role that $v_0$ plays will become clear later.
The existence of such an $S'$ can be checked in a straightforward way using Chebotarev's density theorem.

We write $S'_f = S' \setminus S_{\infty}(F)$.
We introduce the following $\Z[G]$-modules:
\begin{align}
J^{\Sigma'} 
&= \prod_{v \in S_{\infty}(F)} H_v^{\times}
\times \prod_{v \in S'_f} U_{H_v} 
\times \prod_{v \in \Sigma'} U_{H_v}^1 
\times \prod_{v \not \in S' \cup \Sigma'} U_{H_v},\\
V_{S'}^{\Sigma'} 
&= \prod_{v \in S_{\infty}(F)} H_v^{\times}
\times \prod_{v \in S'_f} V_v
\times \prod_{v \in \Sigma'} U_{H_v}^1 
\times \prod_{v \not \in S' \cup \Sigma'} U_{H_v},\\
W_{S'}
&= V_{S'}^{\Sigma'} / J^{\Sigma'} = \prod_{v \in S'_f} W_v.
\end{align}
Here, we define $V_v = \bigoplus_{w \mid v} V_w$ and $W_v = \bigoplus_{w \mid v} W_w$ by using the modules introduced in \S \ref{ss:local}.
We also write $H_v = H \otimes_F F_v \simeq \prod_{w \mid v} H_w$, and $U_{H_v}$ and $U_{H_v}^1$ denote the unit group and the principal unit group of $H_v$, respectively.

Then, by combining \eqref{eq:local_to_global} for $w \in (S_f')_H$ with the tautological maps, we obtain a commutative diagram
\begin{equation}\label{eq:theta}
\xymatrix{
0 \ar[r] &
J^{\Sigma'} \ar[r] \ar[d]_{\theta_J} &
V_{S'}^{\Sigma'} \ar[r] \ar[d]_{\theta_V} &
W_{S'} \ar[r] \ar[d]_{\theta_W} &
0 \\
0 \ar[r] &
C_H \ar[r] &
\fO \ar[r] &
\Delta G \ar[r] &
0.
}
\end{equation}

\begin{lem}[{\cite[Lemma A.1]{DK20}}]\label{lem:theta_surj}
The homomorphism $\theta_V$ in \eqref{eq:theta} is surjective.
\end{lem}

\begin{proof}
We put
\begin{align}
J_{S'}^{\Sigma'} 
&= \prod_{v \in S_{\infty}(F)} H_v^{\times}
\times \prod_{v \in S'_f} H_v^{\times} 
\times \prod_{v \in \Sigma'} U_{H_v}^1 
\times \prod_{v \not \in S' \cup \Sigma'} U_{H_v},\\
\ol{W}_{S'}
&= V_{S'}^{\Sigma'} / J_{S'}^{\Sigma'} = \prod_{v \in S'_f} \Ind_{G_v}^G \Delta G_v,
\end{align}
where $\Ind_{G_v}^G$ denotes the induction from $\Z[G_v]$-modules to $\Z[G]$-modules.
Then, similarly as \eqref{eq:theta}, we obtain a commutative diagram
\begin{equation}\label{eq:theta_prime}
\xymatrix{
0 \ar[r] &
J_{S'}^{\Sigma'} \ar[r] \ar[d]_{\theta_J'} &
V_{S'}^{\Sigma'} \ar[r] \ar[d]_{\theta_V} &
\ol{W}_{S'} \ar[r] \ar[d]_{\theta_W'} &
0 \\
0 \ar[r] &
C_H \ar[r] &
\fO \ar[r] &
\Delta G \ar[r] &
0.
}
\end{equation}
By the choice of $S'$, the cokernels of $\theta'_J$ and of $\theta'_W$ vanish (the former follows from $\Cl_{H, S'}^{\Sigma'} = 0$ and the latter from $\bigcup_{v \in S' \setminus \Sigma} G_v = G$).
Hence $\theta_V$ is surjective by the snake lemma.
\end{proof}

Let $\OO_H^{\times}$ be the unit group of $H$.
We put
\[
\OO_H^{\times, \Sigma'}
= \{ x \in \OO_H^{\times} \mid x \equiv 1 (\bmod w), \forall w \in \Sigma'_H\}.
\]
By applying the snake lemma to \eqref{eq:theta}, thanks to Lemma \ref{lem:theta_surj}, we obtain an exact sequence
\begin{equation}\label{eq:snake_theta}
0 \to \OO_H^{\times, \Sigma'} \to \Ker(\theta_V) \to \Ker(\theta_W) \overset{\delta_{\theta}}{\to} \Cl_H^{\Sigma'} \to 0.
\end{equation}
Since we have $(\OO_H^{\times, \Sigma'})^- = \mu(H)^{\Sigma', -}$ by the Dirichlet unit theorem, the minus component of \eqref{eq:snake_theta} yields
\begin{equation}\label{eq:snake_theta^-}
0 \to \mu(H)^{\Sigma', -} \to \Ker(\theta_V)^- \to \Ker(\theta_W)^- \overset{\delta_{\theta}^-}{\to}  \Cl_H^{\Sigma', -} \to 0.
\end{equation}

We next investigate $\Ker(\theta_W)^-$.
For that purpose, we use the element $v_0 \in S' \setminus \Sigma$ such that $G_{v_0} = \Gal(H/H^+)$.
We write $c \in \Gal(H/H^+)$ for the complex conjugation.

\begin{lem}\label{lem:ker_theta_W}
The natural projection map $W_{S'} \to \bigoplus_{v \in S'_f \setminus \{v_0\}} W_v$ induces an isomorphism
\[
\Ker(\theta_W)^- 
\simeq \bigoplus_{v \in S'_f \setminus \{v_0\}} W_v^-.
\]
\end{lem}

\begin{proof}
For each prime $w_0 \mid v_0$ of $H$, by the sequence \eqref{eq:W}, we have isomorphisms
\[
W_{w_0}^- \simeq (\Delta G_{w_0})^- \simeq \Z[G_{w_0}]^-.
\]
Note that, under the description in Proposition \ref{prop:W_str}, this isomorphism sends $\frac{1}{2} (1 - c, 1) \in W_{w_0}$ to the identity element $\frac{1}{2}(1 - c) \in \Z[G_{w_0}]^-$.
Thus we have an isomorphism $W_{v_0}^- \simeq \Z[G]^-$.
This isomorphism is the component of $\theta_W^-$ at $v_0$, so the lemma follows.
\end{proof}

\begin{defn}\label{defn:f_hom}
We define an injective homomorphism 
\[
f: \Ker(\theta_W)^- \hookrightarrow \bigoplus_{v \in S'_f \setminus \{v_0\}} \Z[G]^-
\]
as the composite map of the isomorphism in Lemma \ref{lem:ker_theta_W} and the map 
\[
\bigoplus_{v \in S'_f \setminus \{v_0\}} W_v^- \to \bigoplus_{v \in S'_f \setminus \{v_0\}} \Z[G]^-
\]
which is defined, by using Proposition \ref{prop:str_W_2}, as $f_v^-$ at the components for $v \in \Sigma_f$ and as $\iota_v^-$ at the components for $v \in S'_f \setminus (\Sigma_f \cup \{v_0\})$.
\end{defn}

\begin{defn}\label{defn:Omega}
We define a $\Z[G]^-$-module $\Omega_{\Sigma}^{\Sigma'}$ as the cokernel of the composite map
\begin{equation}\label{eq:defn_Omega}
\Ker(\theta_V)^- \to \Ker(\theta_W)^- \overset{f}{\hookrightarrow} \bigoplus_{v \in S'_f \setminus \{v_0\}} \Z[G]^-,
\end{equation}
where the first map is the middle map of the sequence \eqref{eq:snake_theta^-}.
\end{defn}

By Proposition \ref{prop:str_W_2}, we have $\Coker(f) \simeq \bigoplus_{v \in \Sigma_f} A_v^-$, where $A_v$ is defined in \eqref{eq:A_v}.
Then the exact sequence \eqref{eq:fund} follows easily from \eqref{eq:snake_theta^-}.
This finishes the proof of Proposition \ref{prop:Omega} (the construction of $\Omega_{\Sigma}^{\Sigma'}$).

Though we omit the detail, the construction of $\Omega_{\Sigma}^{\Sigma'}$ does not depend, up to isomorphisms, on the choice of $S'$ and $v_0$.
This may be deduced from the determination of the extension class as we will discuss in \S \ref{ss:ext}.

The following proposition will be useful for investigation of the homomorphism $\delta_{\theta}$ in the exact sequence \eqref{eq:snake_theta} (defined by the snake lemma).

\begin{prop}[{\cite[Theorem 5]{RW96}}]\label{prop:snake_desc}
Let us consider the commutative diagram
\begin{equation}\label{eq:local_to_global3}
\xymatrix{
0 \ar[r] &
\bigoplus_{v \in S'_f} \Ind_{G_v}^G \Z \ar[r] \ar[d] &
\bigoplus_{v \in S'_f} W_v \ar[r] \ar[d] &
\bigoplus_{v \in S'_f} \Ind_{G_v}^G \Delta G_v \ar[r] \ar@{^{(}->}[d] &
0 \\
0 \ar[r] &
\Cl_H^{\Sigma'} \ar[r] &
\fH^{\Sigma'} \ar[r] &
\Delta G \ar[r] &
0
}
\end{equation}
induced by \eqref{eq:local_to_global2} for $w \in (S'_f)_H$.
Then the image of $\Ker(\theta_W)$ under the middle vertical arrow is contained in the image of $\Cl_H^{\Sigma'}$, and the induced homomorphism $\Ker(\theta_W) \to \Cl_H^{\Sigma'}$ coincides with the homomorphism $\delta_{\theta}$ in \eqref{eq:snake_theta}.
\end{prop}

\subsection{Cohomological triviality}\label{ss:ct}

In this subsection, we prove Proposition \ref{prop:Omega_p_ct}, from which Proposition \ref{prop:Omega_ct} follows.
A key observation is the following facts from local and global class field theory.

\begin{prop}\label{prop:VO_ct}
The following are true.
\begin{itemize}
\item[(1)]
In the situation of \S \ref{ss:local}, the module $V_w$ is cohomologically trivial over $G_w$.
\item[(2)]
In the situation of \S \ref{ss:global}, the module $\fO$ is cohomologically trivial over $G$.
\end{itemize}
\end{prop}





From now on, let us fix an odd prime number $p$.
Recall the convention introduced in \S \ref{ss:eTNCp} on the character components.
For a $\Z$-module $M$, we write $\widehat{M}$ for its $p$-adic completion.
For example, let us consider a non-archimedean local field $H_w$ with mixed characteristic.
When the residue characteristic of $H_w$ is $p$, we have $\widehat{U_{H_w}} \simeq U_{H_w}^1$.
On the other hand, when the residue characteristic of $H_w$ is not $p$, we have $\widehat{U_{H_w}} \simeq \widehat{U_{H_w}/U_{H_w}^1} \simeq \Z_p \otimes_{\Z} U_{H_w}/U_{H_w}^1$.
In both cases, the $\Z_p$-module $\widehat{U_{H_w}}$ is finitely generated.


\begin{lem}\label{lem:U_ct}
We consider the situation in \S \ref{ss:local} (we do not restrict the residue characteristic).
Let us suppose that $G_w = \Gal(H_w/F_v)$ is abelian.
Let $G_w' \subset G_w$ denote the maximal subgroup of order prime to $p$.
Let $\chi$ be a character of $G'_w$.
Suppose either $p \nmid \# I_w$ or $\chi$ is non-trivial on $G_w'$.
Then $U_{H_w}^{\chi}$ is cohomologically trivial over $G_w$.
\end{lem}

\begin{proof}
We make use of the results in \S \ref{ss:local}.
By the middle vertical exact sequence in \eqref{eq:VW} and Proposition \ref{prop:VO_ct}(1), the cohomological triviality of $U_{H_w}^{\chi}$ is equivalent to that of $W_w^{\chi}$, which in turn is equivalent to that of
\[
(\Z_p[G_w/I_w]/(1 - \varphi_v^{-1} + \# I_w))^{\chi}
\]
 by Proposition \ref{prop:str_W_2}(2).
When $p \nmid \# I_w$, then the $\Z_p[G_w]$-module $\Z_p[G_w/I_w]$ is projective, from which we can easily deduce the claim.
Suppose $p \mid \# I_w$ and $\chi$ is non-trivial on $G_w'$.
Then we have
\[
\F_p \otimes_{\Z_p} (\Z_p[G_w/I_w]/(1 - \varphi_v^{-1} + \# I_w))^{\chi}
\simeq (\F_p[G_w/I_w]/(1 - \varphi_v^{-1}))^{\chi}
\simeq (\F_p)^{\chi}
= 0.
\]
By Nakayama's lemma, $(\Z_p[G_w/I_w]/(1 - \varphi_v^{-1} + \# I_w))^{\chi}$ vanishes.
\end{proof}

Now we are ready to prove Proposition \ref{prop:Omega_p_ct}.
\begin{proof}[Proof of Proposition \ref{prop:Omega_p_ct}]
By the sequence \eqref{eq:snake_theta^-}, the homomorphism $\Ker(\theta_V)^{\chi} \to \Ker(\theta_W)^{\chi}$ is injective under (H3)$_p^{\chi}$.
Thus, by Definition \ref{defn:Omega}, it is enough to show that $\Ker(\theta_V)^{\chi}$ is cohomologically trivial over $G$.
Since $\theta_V$ is surjective by Lemma \ref{lem:theta_surj}, thanks to Proposition \ref{prop:VO_ct}, it remains only to show that $(U_{H_v}^1)^{\chi}$ is cohomologically trivial for $v \in \Sigma'$ and $U_{H_v}^{\chi}$ is cohomologically trivial for $v \not \in S' \cup \Sigma'$.
As $(U_{H_v}^1)^{\chi} = 0$ for $v \not \in S_p(F)$ and $(U_{H_v}^1)^{\chi} = U_{H_v}^{\chi}$ for $v \in S_p(F)$, these are combined to that $U_{H_v}^{\chi}$ is cohomologically trivial for $v \not \in S' \cup (\Sigma' \setminus S_p(F))$.
By Lemma \ref{lem:U_ct}, for a finite prime $v$ of $F$, $U_{H_v}^{\chi}$ is cohomologically trivial if $v \not \in S_{\bad}^{\chi}$.
By (H4)$_p^{\chi}$, we have $S_{\bad}^{\chi} \subset \Sigma_f \cup (\Sigma' \setminus S_p(F)) \subset S'_f \cup (\Sigma' \setminus S_p(F))$.
This completes the proof.
\end{proof}

This proof shows that the conclusion of Proposition \ref{prop:Omega_p_ct} holds even if we weaken the condition (H4)$_p^{\chi}$ by replacing $S_{\bad}^{\chi}$ by the smaller set of finite primes $v$ of $F$ such that $p \mid \# I_v$ and $\chi$ is trivial on $G_v'$, not only on $I_v'$.
However, for the purpose of this paper it is more convenient to use the current condition (H4)$_p^{\chi}$.

\subsection{The extension class}\label{ss:ext}

We keep the notation in \S \ref{ss:Omega_constr}.
The goal of this subsection is to obtain a description of the element $\eta_v$ introduced below.
This corresponds to \cite[\S A.5]{DK20}.

\begin{defn}\label{defn:eta}
We define
\[
\eta \in \Ext^1_{\Z[G]^-} \parenth{\bigoplus_{v \in \Sigma_f} A_v^-, \Cl_H^{\Sigma', -}}
\]
as the element which corresponds to the extension \eqref{eq:fund}.
For each $v \in \Sigma_f$, we define $\eta_v \in \Ext^1_{\Z[G]^-} (A_v^-, \Cl_H^{\Sigma', -})$ as the component of $\eta$ at $v$.
\end{defn}

\begin{defn}\label{defn:phi_cocycle}
Recall that, by \cite[Lemma 6.3]{DK20}, the restriction map induces an isomorphism 
\[
H^1(G_F, \Cl_H^{\Sigma', -}) \simeq H^1(G_H, \Cl_H^{\Sigma', -})^G = \Hom_G(G_H, \Cl_H^{\Sigma', -}).
\]
We regard the natural homomorphism
\[
\phi_H: G_H \twoheadrightarrow \Gal(H^{\Sigma', \ab}/H) \simeq \Cl_H^{\Sigma'} \twoheadrightarrow \Cl_H^{\Sigma', -},
\]
where the isomorphism is given by the Artin reciprocity map, as an element of $\Hom_G(G_H, \Cl_H^{\Sigma', -})$.
We then define a cocycle class $[\phi] \in H^1(G_F, \Cl_H^{\Sigma', -})$ so that the restriction of $[\phi]$ equals to $\phi_H$.
Note that we defined only the class $[\phi]$, not $\phi$.
\end{defn}

\begin{lem}[{\cite[\S A.5, (159)]{DK20}}]\label{lem:desc_phi}
Let $\wtil{c} \in G_F$ be a lift of the complex conjugation $c \in G$.
We define a map $\phi: G_F \to \Cl_H^{\Sigma', -}$ by 
\[
\phi(g) = \phi_H(g \wtil{c}g^{-1}\wtil{c}^{-1})^{1/2}.
\]
Then $\phi$ is a cocycle and represents the cocycle class $[\phi] \in H^1(G_F, \Cl_H^{\Sigma', -})$.
\end{lem}

\begin{proof}
Since $H/F$ is abelian, the element $g \wtil{c}g^{-1}\wtil{c}^{-1}$ is contained in $G_H$, so the map $\phi$ is well-defined.
It is easy to see that $\phi$ is a cocycle.
Finally, when $g \in G_H$, by the definition of the action of the complex conjugation, we have $\phi(g) = \phi_H(g)$.
\end{proof}

For each finite prime $v$ of $F$, we define an ideal $J_v$ of $\Z[G]$ as the annihilator ideal of the module $A_v$ defined in \eqref{eq:A_v}.
Since $A_v$ is a cyclic module, we have an exact sequence
\begin{equation}\label{eq:J_v}
0 \to J_v \to \Z[G] \to A_v \to 0,
\end{equation}
which induces an exact sequence
\begin{equation}\label{eq:Jv_ext}
\Hom_{\Z[G]^-}(\Z[G]^-, \Cl_H^{\Sigma', -}) 
\to \Hom_{\Z[G]^-}(J_v^-, \Cl_H^{\Sigma', -}) 
\overset{\delta_v}{\to} \Ext^1_{\Z[G]^-}(A_v^-, \Cl_H^{\Sigma', -}) \to 0.
\end{equation}

Let
\[
\bpsi: G_F \twoheadrightarrow G \hookrightarrow \Z[G]^{\times}
\]
be the tautological character.
Then $J_v$ is generated by 
\[
\bpsi(\sigma) - 1
\]
for $\sigma \in I_{F_v}$ and 
\[
\bpsi(\wtil{\varphi_v}) - 1 + \nu_{I_v} \bpsi(\wtil{\varphi_v})
\]
 for a lift $\wtil{\varphi_v} \in G_{F_v}$ of $\varphi_v \in \Gal(F_v^{\ur}/F_v)$.
Therefore, a $\Z[G]$-homomorphism from $J_v$ is determined by the values of these elements.

\begin{prop}\label{prop:ext_desc}
Let $v \in \Sigma_f$.
Let us choose a lift $\wtil{c} \in G_F$ of the complex conjugation $c \in G$ and define a cocycle $\phi: G_F \to \Cl_H^{\Sigma', -}$ as in Lemma \ref{lem:desc_phi}.
Then there exists an element 
\[
\wtil{\eta_v} \in \Hom_{\Z[G]^-}(J_v^-, \Cl_H^{\Sigma', -})
\]
 which satisfies the following.
\begin{itemize}
\item[(a)]
We have $\delta_v(\wtil{\eta_v}) = \eta_v$, where $\delta_v$ is the homomorphism in \eqref{eq:Jv_ext}.
\item[(b)]
For any $\sigma \in I_{F_v}$, 
we have $\wtil{\eta_v}(\bpsi(\sigma) - 1) = \phi(\sigma)$.
\item[(c)]
For any lift $\wtil{\varphi_v} \in G_{F_v}$ of the arithmetic Frobenius, 
we have $\wtil{\eta_v}(\bpsi(\wtil{\varphi_v}) - 1 + \nu_{I_v} \bpsi(\wtil{\varphi_v})) = \phi(\wtil{\varphi_v})$.
\end{itemize}
\end{prop}

\begin{proof}
In order to distinguish with unspecified $v$'s, we write $v_1$ for the fixed $v$.
We define $\wtil{x} \in \bigoplus_{v \in S'_f \setminus \{v_0\}} \Z[G]^-$ as the element whose component at $v_1$ is $1$ and whose other components are $0$.
We then define $x \in \Omega_{\Sigma}^{\Sigma'}$ as the image of $\wtil{x}$ (recall Definition \ref{defn:Omega}).
By the definitions, the map in \eqref{eq:fund} sends $x$ to the element of $\bigoplus_{v \in \Sigma_f} A_v^-$ whose component at $v_1$ is the class of $1$ and whose other components are zero.
Then we have a commutative diagram with exact rows
\[
\xymatrix{
	0 \ar[r]
	& J_{v_1}^- \ar[r] \ar[d]_{\wtil{\eta_{v_1}}}
	& \Z_p[G]^- \ar[r] \ar[d]
	& A_{v_1} \ar[r] \ar@{^(->}[d]
	& 0\\
	0 \ar[r]
	& \Cl_H^{\Sigma', -} \ar[r] 
	& \Omega_{\Sigma}^{\Sigma'} \ar[r]
	& \bigoplus_{v \in \Sigma_f} A_v^- \ar[r]
	& 0,
}
\]
where the upper sequence is the minus component of \eqref{eq:J_v}, the lower sequence is \eqref{eq:fund}, the right vertical arrow is the injection into the $v_1$-component, the middle vertical arrow is the $\Z_p[G]^-$-homomorphism that sends $1$ to $x$, and the left vertical arrow (named $\wtil{\eta_{v_1}}$) is the induced one.

We first check that $\wtil{\eta_{v_1}}$ satisfies the condition (a).
From the above diagram we obtain a commutative diagram
\[
\xymatrix{
	\Hom_{\Z_p[G]^-}(J_{v_1}^-, \Cl_H^{\Sigma', -}) \ar[r]^{\delta_{v_1}}
	& \Ext^1_{\Z_p[G]^-}(A_{v_1}^-, \Cl_H^{\Sigma', -})\\
	\Hom_{\Z_p[G]^-}(\Cl_H^{\Sigma', -}, \Cl_H^{\Sigma', -}) \ar[r] \ar[u]
	& \Ext^1_{\Z_p[G]^-}(\bigoplus_{v \in \Sigma_f} A_{v}^-, \Cl_H^{\Sigma', -}), \ar@{->>}[u]
}
\]
where $\delta_{v_1}$ is the same map as in \eqref{eq:Jv_ext}.
By the construction, the identity map on $\Cl_H^{\Sigma', -}$ as an element of the left bottom module goes to $\wtil{\eta_{v_1}}$ of the left upper module, and to $\eta$ of the right bottom module.
The right vertical arrow is nothing but the projection to the $v_1$-component.
Therefore, the condition (a) follows from the commutativity of the square diagram.

It remains to show the conditions (b) and (c).
We recall the homomorphism $f$ in Definition \ref{defn:f_hom}:
\[
f: \Ker(\theta_W)^- 
\simeq \bigoplus_{v \in S'_f \setminus \{v_0\}} W_v^-
\hookrightarrow \bigoplus_{v \in S'_f \setminus \{v_0\}} \Z[G]^-,
\]
where the final arrow is $f_v^-$ at $v \in \Sigma_f$ and $\iota_v^-$ at the others.
For each $\sigma$ in (b) (resp. $\wtil{\varphi} := \wtil{\varphi_{v_1}}$ in (c)), let $y_{\sigma}$ (resp. $y_{\wtil{\varphi}}$) be the element of $\Ker(\theta_W)^-$ which is sent by $f$ to $(\bpsi(\sigma) - 1)\wtil{x}$ (resp. $(\bpsi(\wtil{\varphi}) - 1 + \nu_{I_{v_1}} \bpsi(\wtil{\varphi})) \wtil{x}$).
Let 
\[
\delta_{\theta}^-: \Ker(\theta_W)^- \to \Cl_H^{\Sigma', -}
\]
 be the snake map in \eqref{eq:snake_theta^-}.
By the definition of $\wtil{\eta_{v_1}}$, we then have $\wtil{\eta_{v_1}}(\bpsi(\sigma) - 1) = \delta_{\theta}^-(y_{\sigma})$ (resp. $\wtil{\eta_{v_1}}(\bpsi(\wtil{\varphi}) - 1 + \nu_{I_{v_1}} \bpsi(\wtil{\varphi})) = \delta_{\theta}^-(y_{\wtil{\varphi}})$).

Note that, by the definition of $f_{v_1}$, we have $f_{v_1}(\bpsi(\sigma) - 1, 0) = \bpsi(\sigma) - 1$ (resp.~$f_{v_1}(\bpsi(\wtil{\varphi}) - 1, \varphi) = \bpsi(\wtil{\varphi}) - 1 + \nu_{I_{v_1}} \bpsi(\wtil{\varphi})$ with $\varphi := \varphi_{v_1}$).
Also recall that we have an isomorphism $W_{v_0}^- \simeq \Z[G]^-$ which sends $\frac{1}{2} (1 - c, 1)$ to the identity element $\frac{1}{2}(1 - c)$ (see Lemma \ref{lem:ker_theta_W}).
Therefore, the element $y_{\sigma}$ (resp. $y_{\wtil{\varphi}}$) is the element whose component is $(\bpsi(\sigma) - 1, 0)$ (resp. $(\bpsi(\wtil{\varphi}) - 1, \varphi)$) at $v_1$, $-(\bpsi(\sigma) - 1)\frac{1}{2} (1 - c, 1)$  (resp. $-(\bpsi(\wtil{\varphi}) - 1)\frac{1}{2} (1 - c, 1)$) at $v_0$, and zero at the others.

Now we apply Propositions \ref{prop:compute_local_global} and \ref{prop:snake_desc} and, as a consequence, we obtain
\[
\delta_{\theta}^-(y_{\sigma}) 
= 1 \otimes \left(\frac{1}{2}(1 - \ol{\wtil{c}}) (\ol{\sigma} - 1) + (\ol{\sigma} - 1) \frac{1}{2}(\ol{\wtil{c}} - 1) \right)
\]
(resp.
\[
\delta_{\theta}^-(y_{\wtil{\varphi}})
= 1 \otimes \left(\frac{1}{2}(1 - \ol{\wtil{c}})(\ol{\wtil{\varphi}} - 1) + (\ol{\wtil{\varphi}} - 1) \frac{1}{2}(\ol{\wtil{c}} - 1) \right))
\]
as elements of $\fH^{\Sigma'}$ via \eqref{eq:H_desc}.
Here, the overlines mean the restrictions to $\Gal(H^{\Sigma', \ab}/F)$ and the coefficients $\frac{1}{2}(1 - \ol{\wtil{c}})$ comes from projection to the minus component.
This element can be computed as
\[
\delta_{\theta}^-(y_{\sigma}) 
= 1 \otimes \frac{1}{2}(\ol{\sigma} \ol{\wtil{c}} - \ol{\wtil{c}} \ol{\sigma})
= 1 \otimes \frac{1}{2}(\ol{\sigma} \ol{\wtil{c}} \ol{\sigma}^{-1} \ol{\wtil{c}}^{-1} - 1) \ol{\wtil{c}} \ol{\sigma}
= 1 \otimes \frac{1}{2}(\ol{\sigma} \ol{\wtil{c}} \ol{\sigma}^{-1} \ol{\wtil{c}}^{-1} - 1)
\]
(resp.
\[
\delta_{\theta}^-(y_{\wtil{\varphi}}) 
= 1 \otimes \frac{1}{2}(\ol{\wtil{\varphi}} \ol{\wtil{c}} - \ol{\wtil{c}} \ol{\wtil{\varphi}})
= 1 \otimes \frac{1}{2}(\ol{\wtil{\varphi}} \ol{\wtil{c}} \ol{\wtil{\varphi}}^{-1} \ol{\wtil{c}}^{-1} - 1) \ol{\wtil{c}} \ol{\wtil{\varphi}}
= 1 \otimes \frac{1}{2}(\ol{\wtil{\varphi}} \ol{\wtil{c}} \ol{\wtil{\varphi}}^{-1} \ol{\wtil{c}}^{-1} - 1)).
\]
Therefore, via the inclusion $\Cl_H^{\Sigma', -} \hookrightarrow \fH^{\Sigma'}$, 
the element $\delta_{\theta}^-(y_{\sigma})$ (resp.~$\delta_{\theta}^-(y_{\wtil{\varphi}})$) of $\fH^{\Sigma'}$ coincides with $\phi_H(\sigma \wtil{c} \sigma^{-1} \wtil{c}^{-1})^{1/2} = \phi(\sigma) \in \Cl_H^{\Sigma', -}$ (resp. $\phi_H(\wtil{\varphi} \wtil{c} \wtil{\varphi}^{-1} \wtil{c}^{-1})^{1/2} = \phi(\wtil{\varphi}) \in \Cl_H^{\Sigma', -}$).
This completes the proof.
\end{proof}

\section{Compatibility of the eTNC}\label{s:func}

In this section, we prove two functorial properties of our module $\Omega_{\Sigma}^{\Sigma'}$.
These properties are necessary for proving the compatibility of the eTNC and also for reducing the proof of the main theorems of this paper to special cases.

In \S \ref{ss:func_state}, we state the propositions that are to be proved in this section.
The proof is given in \S \ref{ss:pf_state} after preliminary discussion in \S \S \ref{ss:Fitt} -- \ref{ss:p_variant}.

\subsection{Statements}\label{ss:func_state}

Let $H/F$ be a finite abelian CM-extension and put $G = \Gal(H/F)$.
We fix an odd prime number $p$ and write $G'$ for the maximal subgroup of $G$ of order prime to $p$.

\subsubsection{Varying $(\Sigma, \Sigma')$}\label{sss:Sigma}

For each finite prime $v$ of $F$, let us put
\begin{equation}\label{eq:hv'}
h_v' = 1 - N(v) \frac{\nu_{I_v}}{\# I_v} \varphi_v^{-1} \in \Q[G].
\end{equation}
Note that, for a pair $(\Sigma, \Sigma')$ of finite sets of places of $F$ satisfying (H1) and (H2), we have
\begin{equation}\label{eq:theta_h_h'}
\theta_{\Sigma}^{\Sigma'} 
= \prod_{v \in \Sigma_f} h_v^- \cdot \prod_{v \in \Sigma'} h_v^{\prime, -} \cdot \omega
\end{equation}
(see \eqref{eq:theta_defn}).
This immediately gives us formulas for the variance of $\theta_{\Sigma}^{\Sigma'}$ as $(\Sigma, \Sigma')$ varies (we do not write them concretely).
The following proposition is their algebraic counterparts.

\begin{prop}\label{prop:Om_var}
Let $\chi$ be an odd character of $G'$.
Let $(\Sigma, \Sigma')$ be a pair of finite sets of places of $F$ satisfying (H1), (H2), (H3)$_p^{\chi}$, and (H4)$_p^{\chi}$.
Then the following hold.
\begin{itemize}
\item[(1)]
For a finite prime $v \not \in \Sigma \cup \Sigma'$, we have
\[
\Fitt_{\Z_p[G]^{\chi}}({\Omega_{\Sigma \cup \{v\}}^{\Sigma', \chi}})
= {h_v^{\chi}} \Fitt_{\Z_p[G]^{\chi}}({\Omega_{\Sigma}^{\Sigma', \chi}}).
\]
\item[(2)]
For a finite prime $v \not \in \Sigma \cup \Sigma'$, we have
\[
\Fitt_{\Z_p[G]^{\chi}}({\Omega_{\Sigma}^{\Sigma' \cup \{v\}, \chi}})
= h_v^{\prime, \chi} \Fitt_{\Z_p[G]^{\chi}}({\Omega_{\Sigma}^{\Sigma', \chi}}).
\]
\item[(3)]
For a finite prime $v \in \Sigma_f \setminus (S_{\bad}^{\chi} \cap S_p(F))$ (note that $(\Sigma \setminus \{v\}, \Sigma' \cup \{v\})$ again satisfies the conditions (H1), (H2), (H3)$_p^{\chi}$, and (H4)$_p^{\chi}$), we have
\[
\Fitt_{\Z_p[G]^{\chi}}({\Omega_{\Sigma \setminus \{v\}}^{\Sigma' \cup \{v\}, \chi}})
= \frac{h_v^{\prime, \chi}}{{h_v^{\chi}}}\Fitt_{\Z_p[G]^{\chi}}({\Omega_{\Sigma}^{\Sigma', \chi}}).
\]
\end{itemize}
\end{prop}
The proof will be given in \S \ref{ss:pf_state}.
The proof of the claims (1) and (2) is not hard as we have simple relations between the modules concerned.
On the other hand, the proof of the claim (3) is much harder.
This is because we cannot bypass the claim (1) or (2) since the condition (H4)$_p^{\chi}$ fails for the pair $(\Sigma \setminus \{v\}, \Sigma')$ if $v \in S_{\bad}^{\chi} \setminus (S_p(F) \cap S_{\bad}^{\chi})$. 

Proposition \ref{prop:Om_var} implies that the $\chi$-component of the eTNC$_p^-$ is independent from the choice of $(\Sigma, \Sigma')$.
Moreover, the statement of Theorem \ref{thm:current_main_p} is also independent from $(\Sigma, \Sigma')$.

\subsubsection{Varying $H$}\label{sss:H}

Let $K$ be an intermediate CM-field of the original finite abelian CM-extension $H/F$.
In order to clarify the field concerned, we write $h_{v, H}$, $\theta_{\Sigma, H}^{\Sigma'}$, and $\Omega_{\Sigma, H}^{\Sigma'}$ for the objects $h_{v}$, $\theta_{\Sigma}^{\Sigma'}$, and $\Omega_{\Sigma}^{\Sigma'}$ defined for the extension $H/F$.
Then we also have $h_{v, K}$, $\theta_{\Sigma, K}^{\Sigma'}$, and $\Omega_{\Sigma, K}^{\Sigma'}$ defined for $K/F$ instead of $H/F$.

We study the behavior of the analytic and algebraic objects with respect to the natural map
\[
\pi_{H/K}: \Z[G] \to \Z[\GK].
\]
By abuse of notation, we also write $\pi_{H/K}$ for the induced map $\Q[G] \to \Q[\GK]$.

First we observe the behavior on the analytic side.
Let $(\Sigma, \Sigma')$ be a pair of finite sets of places of $F$ satisfying (H1) and (H2).
By \eqref{eq:theta_h_h'}, we immediately obtain
\begin{equation}\label{eq:theta_bar}
\pi_{H/K}(\theta_{\Sigma, H}^{\Sigma'})
= \parenth{\prod_{v \in \Sigma_f} \frac{\pi_{H/K}(h_{v, H})}{h_{v, K}}}
\theta_{\Sigma, K}^{\Sigma'}
\end{equation}
as elements of $\Q[\Gal(K/F)]$.

Note that we have $\pi_{H/K}(h_{v, H}) = h_{v, K}$ if and only if $v$ is unramified in $H/K$.
This is the subtle point of our formulation of the eTNC using $\theta_{\Sigma}^{\Sigma'}$ and $\Omega_{\Sigma}^{\Sigma'}$.
Since the definitions of $A_v$ in \eqref{eq:A_v} and of $h_v$ in \eqref{eq:hv} depend on (the order of) the inertia subgroup, they do not enjoy the most desirable compatibility with respect to field extensions.
This seems to be inevitable in order to make $A_v$ a finite module and $h_v$ a non-zero-divisor.

Now the following is the algebraic counterpart of \eqref{eq:theta_bar}.
As usual, let $G'$ and $\GK'$ be the maximal subgroup of $G$ and $\GK$ of order prime to $p$, respectively.

\begin{prop}\label{prop:fld_var}
Let $\chi$ be an odd character of $\GK'$, which is also regarded as an odd character of $G'$ via the natural map $G' \twoheadrightarrow \GK'$.
Let $(\Sigma, \Sigma')$ be a pair of finite sets of places of $F$ satisfying (H1), (H2), (H3)$_p^{\chi}$, and (H4)$_p^{\chi}$ for $H/F$.
Note that those conditions also hold for $K/F$.
Then we have
\[
\Fitt_{\Z_p[\GK]^{\chi}} \parenth{(\Omega_{\Sigma, H}^{\Sigma', \chi})_{\Gal(H/K)}}
= \parenth{\prod_{v \in \Sigma_f} \frac{\pi_{H/K}(h_{v, H}^{\chi})}{h_{v, K}^{\chi}}}
\Fitt_{\Z_p[\GK]^{\chi}} (\Omega_{\Sigma, K}^{\Sigma', \chi})
\]
as ideals of $\Z_p[\GK]^{\chi}$.
\end{prop}

The proof will be given in \S \ref{ss:pf_state}.

By Propositions \ref{prop:Om_var} and \ref{prop:fld_var}, in order to prove the full statement of Theorem \ref{thm:current_main_p}, we may suppose that $\chi$ is faithful and may choose any pair $(\Sigma, \Sigma')$ (satisfying the conditions (H1), (H2), (H3)$_p^{\chi}$, and (H4)$_p^{\chi}$).
In practice, the choice will be as in Setting \ref{setting}.

\subsection{Preliminaries on Fitting ideals}\label{ss:Fitt}

This subsection is an algebraic preliminary to proving Propositions \ref{prop:Om_var} and \ref{prop:fld_var}.
Most of the contents are just reformulations of materials that are known to experts of this field.

Let $\cG$ be a profinite group which is isomorphic to the product of a finite abelian group and $\Z_p^d$ for some $d \geq 0$.
We consider the completed group ring $R = \Z_p[[\cG]]$.

It is well-known that, for a finitely generated torsion $R$-module $M$ such that $\pd_R(M) \leq 1$, the Fitting ideal $\Fitt_R(M)$ is a principal ideal of $R$ which is generated by a non-zero-divisor.

\begin{defn}\label{defn:Fitt_ind}
Let $M$ and $M'$ be finitely generated $R$-modules such that $\pd_R(M) \leq 1$ and $\pd_R(M') \leq 1$.
Let $\phi: M \to M'$ be an $R$-homomorphism whose kernel and cokernel are both torsion over $R$.
We shall define a principal invertible fractional ideal $\Fitt_R(\phi)$ of $R$, which we call the Fitting ideal of $\phi$, as follows.

We take a projective $R$-module $F$ and an injective homomorphism $\psi: F \to M$ whose cokernel is torsion (there do exist such an $F$ and a $\psi$).
Note that then $\phi \circ \psi: F \to M'$ is also injective with torsion cokernel.
The situation is described by a commutative diagram with exact rows
\[
\xymatrix{
	0 \ar[r]
	& F \ar[r]^{\psi} \ar@{=}[d]
	& M \ar[r] \ar[d]_{\phi}
	& \Coker(\psi) \ar[r] \ar[d]
	& 0\\
	0 \ar[r]
	& F \ar[r]_{\phi \circ \psi}
	& M' \ar[r]
	& \Coker(\phi \circ \psi) \ar[r]
	& 0.	
}
\]
Then we define
\[
\Fitt_R(\phi) = \Fitt_R(\Coker(\psi))^{-1} \Fitt_R(\Coker(\phi \circ \psi)).
\]
Since $\Coker(\psi)$ and $\Coker(\phi \circ \psi)$ are both torsion and of $\pd_R \leq 1$, the Fitting ideals in the right hand are principal and invertible.
We omit the proof of the independency from the choices.
\end{defn}

\begin{lem}\label{lem:quasi_add}
For $\phi: M \to M'$ as in Definition \ref{defn:Fitt_ind}, we have
\[
\Fitt_R(\Coker(\phi)) = \Fitt_R(\phi) \Fitt_R(E_R^1(\Ker(\phi)))
\]
as (not necessarily principal) ideals of $R$.
Here, we put $E^1_R(M) = \Ext^1_R(M, R)$ for an $R$-module $M$.
\end{lem}

\begin{proof}
For $\psi$ as in Definition \ref{defn:Fitt_ind}, we have an exact sequence
\[
0 \to \Ker(\phi) \to \Coker(\psi) \to \Coker(\phi \circ \psi) \to \Coker(\phi) \to 0.
\]
Then the lemma is a reformulation of a formula that is often used in this field (see \cite[Remark 4.8]{Kata_05} and the references mentioned there).
\end{proof}

\begin{rem}
If we assume $\pd_R(\Coker(\phi)) \leq 1$, then we also have  $\pd_R(\Ker(\phi)) \leq 1$ and Lemma \ref{lem:quasi_add} implies
\[
\Fitt_R(\phi) = \Fitt_R(\Ker(\phi))^{-1} \Fitt_R(\Coker(\phi)),
\]
which we may regard as a definition of $\Fitt_R(\phi)$ in this case.
However, we will have to deal with the case where the assumption does not hold.
\end{rem}

We omit the proof of the following two elementary lemmas.

\begin{lem}\label{lem:Fitt_comp}
Let $M$, $M'$, and $M''$ be finitely generated $R$-modules with $\pd_R \leq 1$.
Let $\phi: M \to M'$ and $\phi': M' \to M''$ be $R$-homomorphisms whose kernels and cokernels are all torsion over $R$.
Then the kernel and the cokernel of $\phi' \circ \phi$ are also torsion and we have
\[
\Fitt_R(\phi' \circ \phi)
= \Fitt_R(\phi) \Fitt_R(\phi').
\]
\end{lem}

\begin{lem}\label{lem:Fitt_diagram}
Let us consider a commutative diagram with exact rows
\[
\xymatrix{
	0 \ar[r]
	& M_1 \ar[r] \ar[d]_{\phi_1}
	& M_2 \ar[r] \ar[d]_{\phi_2}
	& M_3 \ar[r] \ar[d]_{\phi_3}
	& 0\\
	0 \ar[r]
	& M_1' \ar[r]
	& M_2' \ar[r]
	& M_3' \ar[r]
	& 0,
}
\]
where all modules are finitely generated $R$-modules with $\pd_R \leq 1$.
We assume that the kernels and cokernels of $\phi_1, \phi_2, \phi_3$ are all torsion over $R$.
Then we have
\[
\Fitt_R(\phi_2) 
= \Fitt_R(\phi_1) \Fitt_R(\phi_3).
\]
\end{lem}

We will also need the following descent property.
Let $\Gamma$ be a closed subgroup of $\cG$ which is isomorphic to $\Z_p$, and we put $\cH = \cG/\Gamma$.
Note that, for each $\Z_p[[\cG]]$-module $M$ with $\pd_{\Z_p[[\cG]]}(M) \leq 1$ and $M^{\Gamma} = 0$, we have $\pd_{\Z_p[[\cH]]}(M_{\Gamma}) \leq 1$.
Let $\pi_{\Gamma}: \Z_p[[\cG]] \to \Z_p[[\cH]]$ be the natural map.
Let us define a multiplicative subset $\cS$ of $\Z_p[[\cG]]$ as the set of elements which are sent by $\pi_{\Gamma}$ to non-zero-divisors of $\Z_p[[\cH]]$.
Then, by localization, $\pi_{\Gamma}$ induces an algebra homomorphism
\[
\cS^{-1} \Z_p[[\cG]] \to \Frac(\Z_p[[\cH]]),
\]
where $\Frac(\Z_p[[\cH]])$ denotes the ring of fractions of $\Z_p[[\cH]]$.
We write $\pi_{\Gamma}$ for this induced homomorphism again.

\begin{prop}\label{prop:descent}
Let $\phi: M \to M'$ be as in Definition \ref{defn:Fitt_ind}.
%
We assume that $M^{\Gamma} = 0$ and $(M')^{\Gamma} = 0$.
We moreover suppose that the kernel and the cokernel of the induced $\Z_p[[\cH]]$-homomorphism $\phi_{\Gamma}: M_{\Gamma} \to M'_{\Gamma}$ are both torsion over $\Z_p[[\cH]]$, so Definition \ref{defn:Fitt_ind} applies to $\phi_{\Gamma}$.
Then we have 
\[
\Fitt_{\Z_p[[\cH]]}(\phi_{\Gamma})
= \pi_{\Gamma} \parenth{\Fitt_{\Z_p[[\cG]]}(\phi)}.
\]
\end{prop}

\begin{proof}
Let us take a projective module $\ol{F}$ over $\Z_p[[\cH]]$ and an injective homomorphism $\ol{\psi}: \ol{F} \to M_{\Gamma}$ whose cokernel is torsion over $\Z_p[[\cH]]$.
Let $F$ be a projective module over $\Z_p[[\cG]]$ such that we have an identification $\ol{F} = F_{\Gamma}$.
By the projectivity of $F$, we can construct a homomorphism $\psi: F \to M$ over $\Z_p[[\cG]]$ whose base change coincides with $\ol{\psi}$.
Since we have $\Coker(\psi)_{\Gamma} \simeq \Coker(\ol{\psi})$, which is assumed to be torsion over $\Z_p[[\cH]]$, the module $\Coker(\psi)$ is also torsion over $\Z_p[[\cG]]$.
By observing the ranks, we also see that $\psi$ is injective.
Therefore, we can use $\psi$ to compute $\Fitt_R(\phi)$.
Then the basic property of Fitting ideals shows
\[
\Fitt_{\Z_p[[\cH]]}(\Coker(\ol{\psi})) 
= \Fitt_{\Z_p[[\cH]]}(\Coker(\psi)_{\Gamma}) 
= \pi_{\Gamma} \parenth{\Fitt_{\Z_p[[\cG]]}(\Coker(\psi))}.
\]
We also have a similar formula for $\phi \circ \psi$ instead of $\psi$, and those imply the proposition.
\end{proof}

\subsection{A minor variant of the Ritter--Weiss type module}\label{ss:p_variant}

As usual, let $H/F$ be a finite abelian CM-extension and $p$ an odd prime number.
Let $\Sigma$ and $\Sigma'$ be finite sets of places of $F$ satisfying (H1) and (H2).

In this subsection, given a family of scalars $(c_v)_{v \in \Sigma_f}$, using Remark \ref{rem:fw_var}, we construct a variant $\ol{\Omega_{\Sigma}^{\Sigma'}}$ of $\Z_p \otimes \Omega_{\Sigma}^{\Sigma'}$.
Note that the case where $c_v = 1$ for all $v$ would recover the original.
The variant will be useful for the proof of Propositions \ref{prop:Om_var} and \ref{prop:fld_var}.


Let us suppose that we are given a family $(c_v)_{v \in \Sigma_f} \subset \Z_p \setminus \{0\}$.
See \eqref{eq:cv2} and \eqref{eq:cv1} below for practical choices.
The choice of $(c_v)_{v \in \Sigma_f}$ will be implicit in the notation; the objects with overlines do depend on the choice of $(c_v)_v$.

For each $v \in \Sigma_f$, 
as a variant of \eqref{eq:A_v}, we define a $\Z_p[G]$-module $\ol{A_v}$ by
\[
\ol{A_v} = \Z_p[G/I_v]/(1 - \varphi_v^{-1} + c_v \# I_v).
\]
Then, as a variant of Proposition \ref{prop:str_W_2}(2) (see Remark \ref{rem:fw_var}), we have an exact sequence of $\Z_p[G]$-modules
\begin{equation}\label{eq:VW_var}
0 \to \Z_p \otimes W_v \overset{\ol{f_v}}{\to} \Z_p[G] \to \ol{A_v} \to 0,
\end{equation}
where $\ol{f_v}$ sends $(x, y)$ to $x + c_v \nu_{I_v} y$.


We choose $S'$ (and $v_0$) as in \S \ref{ss:Omega_constr} and use the same notation.
Then, as a variant of Definition \ref{defn:f_hom}, we can define a $\Z_p[G]^-$-homomorphism
\[
\ol{f}: \Z_p \otimes \Ker(\theta_W)^- \hookrightarrow \bigoplus_{v \in S'_f \setminus \{v_0\}} \Z_p[G]^-
\]
as the composite map of the isomorphism in Lemma \ref{lem:ker_theta_W} and the map 
\[
\Z_p \otimes \bigoplus_{v \in S'_f \setminus \{v_0\}} W_v^- \to \bigoplus_{v \in S'_f \setminus \{v_0\}} \Z_p[G]^-
\]
which is defined as $\ol{f_v}^-$ at the components for $v \in \Sigma_f$ and as $\iota_v^-$ at the components for $v \in S' \setminus (\Sigma \cup \{v_0\})$.

Now, as a variant of Definition \ref{defn:Omega}, we define a $\Z_p[G]^-$-module $\ol{\Omega_{\Sigma}^{\Sigma'}}$ as the cokernel of the composite map
\begin{equation}\label{eq:defn_Omega_bar}
\Z_p \otimes \Ker(\theta_V)^- 
\to \Z_p \otimes \Ker(\theta_W)^- 
\overset{\ol{f}}{\hookrightarrow} \bigoplus_{v \in S'_f \setminus \{v_0\}} \Z_p[G]^-,
\end{equation}
where the first map is the middle map of the sequence \eqref{eq:snake_theta^-}.
Then, as a variant of \eqref{eq:fund}, we have an exact sequence
\begin{equation}\label{eq:fund_bar}
0 \to \Z_p \otimes \Cl_H^{\Sigma', -} 
\to \ol{\Omega_{\Sigma}^{\Sigma'}}
\to \bigoplus_{v \in \Sigma_f} \ol{A_v}^-
\to 0.
\end{equation}
We write $\ol{\Omega_{\Sigma}^{\Sigma', \chi}}$ for the $\chi$-component of $\ol{\Omega_{\Sigma}^{\Sigma'}}$ for each odd character $\chi$ of $G'$.

As a variant of Proposition \ref{prop:Omega_p_ct}, 
if the conditions (H3)$_p^{\chi}$ and (H4)$_p^{\chi}$ hold, then we have $\pd_{\Z_p[G]^{\chi}}(\ol{\Omega_{\Sigma}^{\Sigma', \chi}}) \leq 1$.

We investigate the difference between the Fitting ideals of $\Omega_{\Sigma}^{\Sigma', \chi}$ and of $\ol{\Omega_{\Sigma}^{\Sigma', \chi}}$.
For a finite prime $v$ not lying above $p$, as a variant of $h_v$ defined in \eqref{eq:hv}, we put
\[
\ol{h_v}
= 1 - \frac{\nu_{I_v}}{\# I_v} (\varphi_v^{-1} - c_v \# I_v)
\in \Q_p[G].
\]

\begin{prop}\label{prop:variance_bar}
Let $\chi$ be an odd character of $G'$ and suppose that the conditions (H3)$_p^{\chi}$ and (H4)$_p^{\chi}$ hold for $(\Sigma, \Sigma')$.
Then we have
\[
\Fitt_{\Z_p[G]^{\chi}}(\Omega_{\Sigma}^{\Sigma', \chi})
= 
\parenth{\prod_{v \in \Sigma_f} 
\frac{h_v^{\chi}}{\ol{h_v}^{\chi}}}
\Fitt_{\Z_p[G]^{\chi}}(\ol{\Omega_{\Sigma}^{\Sigma', \chi}}).
\]
\end{prop}

\begin{proof}
First, for each $v \in \Sigma_f$, let us observe that the following diagram commutes:
\[
\xymatrix{
	& \Z_p[G] \ar[dr]^{\times \ol{h_v}}& \\
	\Z_p \otimes W_v \ar[ru]^-{f_v} \ar[rd]_-{\ol{f_v}}
	& & \frac{1}{\# I_v} \Z_p[G]\\
	& \Z_p[G] \ar[ru]_-{\times h_v}&
}
\]
The commutativity says
\[
\ol{h_v} (x + \nu_{I_v} y) 
= h_v \parenth{x + c_v \nu_{I_v} y}
\]
for $(x, y) \in W_v$.
We can directly compute
\[
h_v - \ol{h_v}
= \frac{\nu_{I_v}}{\# I_v} \cdot \# I_v (1 - c_v)
= \nu_{I_v} (1 - c_v)
\]
and
\[
\nu_{I_v} (c_v h_v - \ol{h_v})
= \nu_{I_v} (c_v - 1) \parenth{1 - \frac{\nu_{I_v}}{\# I_v} \varphi_v^{-1}}
= (c_v - 1) \nu_{I_v} \parenth{1 - \varphi_v^{-1}}.
\]
Then the desired formula follows from the relation $\nu_{I_v} x = \nu_{I_v} (1 - \varphi_v^{-1}) y$ by the description of $W_v$.

By definition, $\Fitt_{\Z_p[G]^{\chi}}(\Omega_{\Sigma}^{\Sigma', \chi})$ (resp. $\Fitt_{\Z_p[G]^{\chi}}(\ol{\Omega_{\Sigma}^{\Sigma', \chi}})$) equals to the Fitting ideal of the $\chi$-component of the homomorphism \eqref{eq:defn_Omega} (resp. \eqref{eq:defn_Omega_bar}), in the sense of Definition \ref{defn:Fitt_ind}.
We consider a commutative diagram
\[
\xymatrix{
	& \bigoplus_{v \in S'_f \setminus \{v_0\}} \Z_p[G]^- \ar[rd]^-{\ol{h}} &\\
	\Ker(\theta_W)^- \ar[ru]_-{f} \ar[rd]_-{\ol{f}}& &
	\bigoplus_{v \in S'_f \setminus \{v_0\}} \frac{1}{\# I_v} \Z_p[G]^- \\
	& \bigoplus_{v \in S'_f \setminus \{v_0\}} \Z_p[G]^- \ar[ru]_-{{h}},
& }
\]
Here, $h$ and $\ol{h}$ are defined as $\times h_v^-$ and $\times \ol{h_v}^-$ for $v \in \Sigma_f$ respectively, and as the identity (the inclusion) for the other $v$'s.
Then this diagram also commutes thanks to the above observation.
It then follows from Lemma \ref{lem:Fitt_comp} that 
\[
\Fitt_{\Z_p[G]^{\chi}}(\ol{h}^{\chi})
\Fitt_{\Z_p[G]^{\chi}}(\Omega_{\Sigma}^{\Sigma', \chi})
= \Fitt_{\Z_p[G]^{\chi}}(h^{\chi})
\Fitt_{\Z_p[G]^{\chi}}(\ol{\Omega_{\Sigma}^{\Sigma', \chi}}),
\]
which is equivalent to
\[
\parenth{\prod_{v \in \Sigma_f} \ol{h_v}^{\chi} } \Fitt_{\Z_p[G]^{\chi}}(\Omega_{\Sigma}^{\Sigma', \chi})
= \parenth{\prod_{v \in \Sigma_f} h_v^{\chi} } \Fitt_{\Z_p[G]^{\chi}}(\ol{\Omega_{\Sigma}^{\Sigma', \chi}}).
\]
This shows the proposition.
\end{proof}

\subsection{The proof of Propositions \ref{prop:Om_var} and \ref{prop:fld_var}}\label{ss:pf_state}

First we prove Proposition \ref{prop:fld_var}.

\begin{proof}[Proof of Proposition \ref{prop:fld_var}]
We define $(c_v)_{v \in \Sigma_f} \subset \Z_p \setminus \{0\}$ by
\begin{equation}\label{eq:cv2}
c_v = \# I_{v, H/K}
\end{equation}
for each $v$, where we write $I_{v, H/K}$ for the inertia group of $H/K$.
We define $\ol{h_{v, K}}$ and $\ol{\Omega_{\Sigma, K}^{\Sigma'}}$ using this choice of $(c_v)_{v \in \Sigma_f}$ as in \S \ref{ss:p_variant}.
By the choice \eqref{eq:cv2}, concerning the natural map $\pi_{H/K}: \Z[G] \to \Z[\GK]$, we have
\begin{equation}\label{eq:hv_var}
\pi_{H/K}(h_{v, H}) = \ol{h_{v, K}}.
\end{equation}
This formula is the motivation for the choice \eqref{eq:cv2}; see the discussion after \eqref{eq:theta_bar}.
Then by Proposition \ref{prop:variance_bar}, in order to prove Proposition \ref{prop:fld_var}, it is enough to show an isomorphism
\begin{equation}\label{eq:bc_Omega}
\ol{\Omega_{\Sigma, K}^{\Sigma', \chi}} 
\simeq (\Omega_{\Sigma, H}^{\Sigma', \chi})_{\Gal(H/K)}.
\end{equation}
Let us write $\theta_{V, H}$ and $\theta_{V, K}$ for the homomorphisms $\theta_V$ in \eqref{eq:theta} for the extension $H/F$ and $K/F$, respectively.
We write $\theta_{W, H}$ and $\theta_{W, K}$ similarly.
Then we have the following commutative diagram:
\[
\xymatrix{
	(\Z_p \otimes \Ker(\theta_{V, H})^-)_{\Gal(H/K)} \ar[r] \ar[d]_{\simeq}
	& (\Z_p \otimes \Ker(\theta_{W, H})^-)_{\Gal(H/K)} \ar[r]^-{f_H} \ar[d]
	& \bigoplus_{v \in S'_f \setminus \{v_0\}} \Z_p[\Gal(H/F)]^-_{\Gal(H/K)} \ar[d]_{\simeq}\\
	\Z_p \otimes \Ker(\theta_{V, K})^- \ar[r] 
	& \Z_p \otimes \Ker(\theta_{W, K})^- \ar[r]_-{\ol{f_{K}}}
	& \bigoplus_{v \in S'_f \setminus \{v_0\}} \Z_p[\Gal(K/F)]^-.
}
\]
Here, the left vertical isomorphism follows from \cite[Lemma B.1]{DK20} and the right square is commutative because of the choice \eqref{eq:cv2} of $(c_v)_{v \in \Sigma_f}$.
Then we obtain an isomorphism \eqref{eq:bc_Omega}.
This completes the proof of Proposition \ref{prop:fld_var}.
\end{proof}

In the rest of this subsection, we aim at proving Proposition \ref{prop:Om_var}.
The proof makes use of the local considerations as in Lemmas \ref{lem:Fitt_loc} and \ref{lem:Fitt_loc2}.

\begin{lem}\label{lem:Fitt_loc}
Let $\chi$ be an odd character of $G'$.
Let $v$ be a finite prime of $F$.
We suppose either $p \nmid \# I_v$ or $\chi$ is non-trivial on $I_v'$.
Then the following are true.
\begin{itemize}
\item[(1)]
We have $\Fitt_{\Z_p[G]^{\chi}} ({A_v^{\chi}}) = ({h_v^{\chi}})$.
\item[(2)]
We have $\Fitt_{\Z_p[G]^{\chi}} \parenth{U_{H_v}/U_{H_v}^1}^{\chi} = (h_v^{\prime, \chi})$.
\end{itemize}
\end{lem}

\begin{proof}
(1)
If $\chi$ is non-trivial on $I_v'$, the both sides are the unit ideal by the definitions of ${A_v}$ and of $h_v$.
Let us assume that $p \nmid \# I_v$ and show 
\[
\Fitt_{\Z_p[G]} ({A_v})
= ({h_v})
\]
as ideals of $\Z_p[G]$.
Since $p \nmid \# I_v$, we may decompose this equality with respect to the characters of $I_v$.
For the nontrivial characters of $I_v$, the components of the both sides are the unit ideal.
For the trivial character of $I_v$, the components of the both sides are the same as $(1 - \varphi_v^{-1} + \# I_v)$, so the equality holds.

(2)
Similarly, if $\chi$ is non-trivial on $I_v'$, the both sides are the unit ideal.
Let us assume that $p \nmid \# I_v$ and show 
\[
\Fitt_{\Z_p[G]} \parenth{\Z_p \otimes U_{H_v}/U_{H_v}^1} = \parenth{1 - N(v) \frac{\nu_{I_v}}{\# I_v} \varphi_v^{-1}}
\]
as ideals of $\Z_p[G]$.
When $v \mid p$, we have $\Z_p \otimes U_{H_v}/U_{H_v}^1 = 0$ and $1 - N(v) \frac{\nu_{I_v}}{\# I_v} \varphi_v^{-1}$ is a unit, so the both sides are again the unit ideal.
When $v \nmid p$, we have 
\[
\Z_p \otimes U_{H_v}/U_{H_v}^1 \simeq \mu_{p^{\infty}}(H_v) \simeq \Z_p[G/I_v]/(1 - N(v) \varphi_v^{-1}),
\]
from which we can deduce the claim by considering the decomposition with respect to characters of $I_v$.
\end{proof}

\begin{proof}[Proof of Proposition \ref{prop:Om_var}(1)(2)]
As we need to use the construction in \S \ref{ss:Omega_constr} for various $\Sigma$ and $\Sigma'$, we write $\theta_{S'}^{\Sigma'}$ for $\theta_V$ in \eqref{eq:theta}.

(1)
We may choose $S'$ so that $v \not \in S'$.
Let us use $S'$ for the construction of ${\Omega_{\Sigma}^{\Sigma'}}$, and $S' \cup \{v\}$ for ${\Omega_{\Sigma \cup \{v\}}^{\Sigma'}}$.
By definition, we have an exact sequence 
\[
0 \to V_{S'}^{\Sigma'} \to V_{S' \cup \{v\}}^{\Sigma'} \to W_v \to 0,
\]
 which induces the upper exact sequence in a commutative diagram
\[
\xymatrix{
	0 \ar[r]
	& \Z_p \otimes \Ker(\theta_{S'}^{\Sigma'})^- \ar[r] \ar@{^(->}[d]
	& \Z_p \otimes \Ker(\theta_{S' \cup \{v\}}^{\Sigma'})^- \ar[r] \ar@{^(->}[d]
	& \Z_p \otimes W_v^- \ar[r] \ar@{^(->}[d]
	&0\\
	0 \ar[r]
	& \bigoplus_{v \in S_f' \setminus \{v_0\}} \Z_p[G]^- \ar[r]
	& \bigoplus_{v \in (S_f' \cup \{v\}) \setminus \{v_0\}} \Z_p[G]^- \ar[r]
	& \Z_p[G]^- \ar[r]
	& 0.
}\]
Here, the right vertical arrow is $f_v^-$ and the left and the middle arrows are those defining $\Omega_{\Sigma}^{\Sigma'}$ and $\Omega_{\Sigma \cup \{v\}}^{\Sigma'}$, respectively.
By applying the snake lemma, we obtain an exact sequence
\[
0 
\to {\Omega_{\Sigma}^{\Sigma'}} 
\to {\Omega_{\Sigma \cup \{v\}}^{\Sigma'}}
\to {A_v}^-
 \to 0.
\]
Since $v \not \in S_{\bad}^{\chi}$ by (H4)$_p^{\chi}$, we can apply Lemma \ref{lem:Fitt_loc}(1) and we obtain the claim.

(2)
We may choose $S'$ so that $v \not \in S'$ and use $S'$ for the constructions of both ${\Omega_{\Sigma}^{\Sigma'}}$ and ${\Omega_{\Sigma}^{\Sigma' \cup\{v\}}}$.
By definition, we have an exact sequence 
\[
0 \to V_{S'}^{\Sigma' \cup \{v\}} \to V_{S'}^{\Sigma'} \to U_{H_v}/U_{H_v}^1 \to 0,
\]
 which induces the upper exact sequence in a commutative diagram
\[
\xymatrix{
	0 \ar[r]
	& \Z_p \otimes \Ker(\theta_{S'}^{\Sigma' \cup \{v\}})^- \ar[r] \ar@{^(->}[d]
	& \Z_p \otimes \Ker(\theta_{S'}^{\Sigma'})^- \ar[r] \ar@{^(->}[d]
	& \Z_p \otimes (U_{H_v}/U_{H_v}^1)^- \ar[r]
	&0\\
	& \bigoplus_{v \in S_f' \setminus \{v_0\}} \Z_p[G]^- \ar@{=}[r]
	& \bigoplus_{v \in S_f' \setminus \{v_0\}} \Z_p[G]^-.
	&&	
}\]
By applying the snake lemma, we obtain an exact sequence
\[
0 \to \Z_p \otimes \parenth{U_{H_v}/U_{H_v}^1}^- 
\to {\Omega_{\Sigma}^{\Sigma' \cup \{v\}}} 
\to {\Omega_{\Sigma}^{\Sigma'}}
 \to 0.
\]
Then we can apply Lemma \ref{lem:Fitt_loc}(2) and obtain the claim.
\end{proof}

As already remarked, Proposition \ref{prop:Om_var}(3) is the harder part and in order to prove it we make use of the variants with overlines introduced in \S \ref{ss:p_variant}.
Thanks to Proposition \ref{prop:variance_bar}, for any choice of $(c_v)_{v \in \Sigma_f} \subset \Z_p \setminus \{0\}$, 
the proposition is equivalent to the corresponding statement for objects with overlines; simply add the overlines on all $\Omega$'s and all $h_v$'s for $v \in \Sigma_f$ (not on $h_v'$'s).
Indeed it is convenient to take
\begin{equation}\label{eq:cv1}
c_v = 
\begin{cases}
(N(v)^{-1} - 1)/\# I_v & (v \nmid p)\\
1 & (v \mid p).
\end{cases}
\end{equation}
By local class field theory, $c_v$ is actually $p$-adically integral.
Accordingly we construct objects with overlines as in \S \ref{ss:p_variant}.


\begin{lem}\label{lem:Fitt_loc2}
Let $v$ be a finite prime of $F$ such that $v \nmid p$.
Let us consider the composite map
\[
\phi_v: \widehat{V_v} \to \Z_p \otimes W_v \to \Z_p[G],
\]
where the first map is the surjection in \eqref{eq:VW} and the second is the injection $\ol{f_v}$.
Here, recall that $\widehat{(-)}$ denotes the $p$-adic completion, so we have $\widehat{W_v} \simeq \Z_p \otimes W_v$.
Then we have
\[
\Fitt_{\Z_p[G]}(\phi_v) 
= (1).
\]
\end{lem}

\begin{proof}
Let us fix a prime $w$ of $H$ lying above $v$.
It is enough to study the $w$-component $\phi_w$ of $\phi_v$.
The kernel and the cokernel of $\phi_w$ can be determined by the middle vertical sequence in \eqref{eq:VW} and the sequence \eqref{eq:VW_var}, respectively.
We then have an exact sequence
\[
0 \to \mu_{p^{\infty}}(H_w) \to \widehat{V_w} \overset{\phi_w}{\to} \Z_p[G_w] \to \mu_{p^{\infty}}(H_w) \to 0.
\]
Here, we used a non-canonical isomorphism between $\ol{A_w} = \Z_p[G_w]/(N(v)^{-1} - \varphi_v^{-1})$ and $\mu_{p^{\infty}}(H_w)$.
When $H_w$ does not contain a non-trivial $p$-th power root of unity, $\phi_w$ is an isomorphism, so the assertion is clear.
In the rest of the proof, we assume that $H_w$ contains a non-trivial $p$-th power roots of unity.
A hard point is that the module $\mu_{p^{\infty}}(H_w)$ is not cohomologically trivial in general, so the Fitting ideals do not behave well.
To overcome such a difficulty, we use an idea from Iwasawa theory; we go up the unramified $\Z_p$-tower and then descend to the finite extension.

For each integer $n \geq 0$, we consider the unramified extension $H_{w, n}/H_w$ of degree $p^n$.
We can apply the constructions in \S \ref{ss:local} to the extension $H_{w, n}/F_v$, and then we obtain $\Z_p[\Gal(H_{w, n}/F_v)]$-modules $V_{w, n}$ and $W_{w, n}$.
Then, as in (the proof of) \cite[Lemma B.1]{DK20}, we have natural isomorphisms 
\[
V_{w, n} 
\simeq (V_{w, n'})_{\Gal(H_{w, n'}/H_{w, n})}
\]
for each $n' \geq n$.
We define $V_{w, \infty} = \varprojlim_n \widehat{V_{w, n}}$.
Then, by Proposition \ref{prop:VO_ct}(1), we can see that $V_{w, \infty}$ is cohomologically trivial over $G_{w, \infty} = \Gal(H_{w, \infty}/F_v)$ and moreover we have $(V_{w, \infty})_{\Gal(H_{w, \infty}/H_w)} \simeq \widehat{V_w}$.

We construct a homomorphism $\phi_{w, n}: \widehat{V_{w, n}} \to \Z_p[G_{w, n}]$ in an analogous way.
It is straightforward to show that the homomorphisms $\phi_{w, n}$ are compatible;
for $n' \geq n \geq 0$, we have a commutative diagram
\[
\xymatrix{
	\widehat{V_{w, n'}} \ar@{->>}[d] \ar@{->>}[r]
	& \Z_p \otimes W_{w, n'} \ar@{->>}[d] \ar@{^{(}->}[r]^-{\ol{f_{v, n'}}}
	& \Z_p[G_{w, n'}] \ar@{->>}[d]\\
	\widehat{V_{w, n}} \ar@{->>}[r]
	& \Z_p \otimes W_{w, n} \ar@{^{(}->}[r]_-{\ol{f_{v, n}}}
	& \Z_p[G_{w, n}].
}
\]
Here, we use the same $c_v$ as in \eqref{eq:cv1} for every $n \geq 0$.
The commutativity of the right square is verified as the extension $H_{n'}/H_n$ is unramified.
As a consequence, we obtain a homomorphism $\phi_{w, \infty}$ which fits in the following diagram with exact rows:
\[
\xymatrix{
	0 \ar[r]
	& \Z_p(1) \ar[r] \ar@{->>}[d]
	& V_{w, \infty} \ar[r]^-{\phi_{w, \infty}} \ar@{->>}[d]
	& \Z_p[[G_{w, \infty}]] \ar@{->>}[d] \ar[r]
	& \Z_p(1) \ar[r]	\ar@{->>}[d]
	& 0\\
	0 \ar[r]
	& \mu_{p^{\infty}}(H_w) \ar[r]
	& \widehat{V_w} \ar[r]_-{\phi_w} 
	& \Z_p[G_w] \ar[r]
	& \mu_{p^{\infty}}(H_w) \ar[r]
	& 0.
}
\]
Here, we again used a non-canonical isomorphism between the cokernel of $\phi_{w, \infty}$ and $\Z_p(1)$.
It follows from the upper sequence and Lemma \ref{lem:quasi_add} that
\[
\Fitt_{\Z_p[[G_{w, \infty}]]}(\Z_p(1)) = \Fitt_{\Z_p[[G_{w, \infty}]]}(\phi_{w, \infty}) \Fitt_{\Z_p[[G_{w, \infty}]]}(\Z_p(1)).
\]
The ideal $\Fitt_{\Z_p[[G_{w, \infty}]]}(\Z_p(1))$ is not principal in general.
Nevertheless, since the $\mu$-invariant of $\Z_p(1)$ is zero, this formula implies $\Fitt_{\Z_p[[G_{w, \infty}]]}(\phi_{w, \infty}) = (1)$ (see, e.g., \cite[Lemma 3.12]{GKK_09}).
Finally, by applying Proposition \ref{prop:descent} to $\phi_{w, \infty}$, we obtain the proposition.
\end{proof}

Now we are ready to prove Proposition \ref{prop:Om_var}(3).

\begin{proof}[Proof of Proposition \ref{prop:Om_var}(3)]
First we assume that $v \mid p$.
By the hypothesis $v \in \Sigma_f \setminus (S_{\bad}^{\chi} \cap S_p(F))$, we then have $v \not \in S_{\bad}^{\chi}$.
Then the condition (H4)$_p^{\chi}$ still holds for $(\Sigma \setminus \{v\}, \Sigma')$.
Thus we can apply the claims (1) and (2) to show
\begin{align}
\Fitt_{\Z_p[G]^{\chi}}({\Omega_{\Sigma \setminus \{v\}}^{\Sigma' \cup \{v\}, \chi}})
& = h_v^{\prime, \chi} \Fitt_{\Z_p[G]^{\chi}}({\Omega_{\Sigma \setminus \{v\}}^{\Sigma', \chi}})\\
& = \frac{h_v^{\prime, \chi}}{h_v^{\chi}} \Fitt_{\Z_p[G]^{\chi}}({\Omega_{\Sigma}^{\Sigma', \chi}}).
\end{align}

Now we assume that $v \nmid p$.
As in the proof of (1) and (2), we have a commutative diagram with exact rows
\[
\xymatrix{
	0 \ar[r]
	& \Z_p \otimes \Ker(\theta_{S' \setminus \{v\}}^{\Sigma' \cup \{v\}})^- \ar[r] \ar@{^(->}[d]
	& \Z_p \otimes \Ker(\theta_{S'}^{\Sigma'})^- \ar[r] \ar@{^(->}[d]
	& \widehat{V_v}^- \ar[r] \ar[d]^{\phi_v}
	&0\\
	0 \ar[r]
	& \bigoplus_{v \in (S_f' \setminus \{v\}) \setminus \{v_0\}} \Z_p[G]^- \ar[r]
	& \bigoplus_{v \in S_f' \setminus \{v_0\}} \Z_p[G]^- \ar[r]
	& \Z_p[G]^- \ar[r]
	& 0.
}\]
Here, the right vertical arrow $\phi_v$ is the same as in Lemma \ref{lem:Fitt_loc2}.
The left and the middle vertical arrows are those denoted by $\ol{f}$, whose cokernels are $\ol{\Omega_{\Sigma \setminus \{v\}}^{\Sigma' \cup \{v\}}}$ and $\ol{\Omega_{\Sigma}^{\Sigma'}}$, respectively.
Then by Lemmas \ref{lem:Fitt_diagram} and \ref{lem:Fitt_loc2}, we obtain
\[
\Fitt_{\Z_p[G]^{\chi}}(\ol{\Omega_{\Sigma \setminus \{v\}}^{\Sigma' \cup \{v\}, \chi}})
 = \Fitt_{\Z_p[G]^{\chi}}(\ol{\Omega_{\Sigma}^{\Sigma', \chi}}).
 \]
 By Proposition \ref{prop:variance_bar}, this can be rewritten as
\[
\Fitt_{\Z_p[G]^{\chi}}({\Omega_{\Sigma \setminus \{v\}}^{\Sigma' \cup \{v\}, \chi}})
 = \frac{\ol{h_v}^{\chi}}{h_v^{\chi}} \Fitt_{\Z_p[G]^{\chi}}({\Omega_{\Sigma}^{\Sigma', \chi}}).
 \]
 Since $\ol{h_v} = h_v'$ by \eqref{eq:cv1}, this completes the proof.
\end{proof}

%

\section{Reduction of the main theorems}\label{s:ded}

In \S \ref{ss:int_Stickel}, we show the integrality of the Stickelberger element.
In \S \ref{ss:divisibility}, we show how to deduce the main theorems on the eTNC$_p^-$ from single inclusions.
Then in \S \ref{ss:pf_thm}, we deduce Theorems \ref{thm:current_main} and \ref{thm:main_p_2} from Theorem \ref{thm:current_main_p}.
We will begin the actual proof of Theorem \ref{thm:current_main_p} from the next section.
In \S \ref{Ap:eTNC}, we show that our eTNC$^-$ is equivalent to more standard formulations.

\subsection{The integrality of the Stickelberger element}\label{ss:int_Stickel}

Let $H/F$ be a finite abelian CM-extension with $G = \Gal(H/F)$.
Let $(\Sigma, \Sigma')$ be a pair of finite sets of places of $F$ satisfying the conditions (H1) and (H2).
In this subsection, we show the integrality of the Stickelberger element $\theta_{\Sigma}^{\Sigma'}$ under certain hypotheses.

Before the main discussion, let us introduce conditions (H3)$_p$ and (H4)$_p$ for each odd prime number $p$:
\begin{itemize}
\item[(H3)$_p$]
$\mu_{p^{\infty}}(H)^{\Sigma'}$ vanishes.
\item[(H4)$_p$]
$\Sigma_f \cup \Sigma' \supset S_{\bad}$ and $\Sigma_f \supset S_{\bad} \cap S_p(F)$.
Here, we write $S_{\bad}$ for the set of finite primes $v$ of $F$ such that $p \mid \# I_v$ and $H^{I_v'}$ is a CM-field.
\end{itemize}
It is easy to see that (H3) (resp.~(H4)) implies (H3)$_p$ (resp.~(H4)$_p$) (for all odd prime $p$).
Moreover, the condition (H3)$_p$ (resp.~(H4)$_p$) is equivalent to that (H3)$_p^{\chi}$ (resp.~(H4)$_p^{\chi}$) holds for all odd characters $\chi$ of $G'$ (the maximal subgroup of $G$ of order prime to $p$).
In order to show the equivalence between (H4)$_p$ and (H4)$_p^{\chi}$, it is enough to show $S_{\bad} = \bigcup_{\chi} S_{\bad}^{\chi}$, where $\chi$ runs over the odd characters of $G'$.
This follows immediately from Lemma \ref{lem:H4} below (applied to $N = I_v'$).

Then, by Proposition \ref{prop:Omega_p_ct}, we have $\pd_{\Z_p[G]^-}(\Z_p \otimes_{\Z} \Omega_{\Sigma}^{\Sigma'}) \leq 1$ as long as the pair $(\Sigma, \Sigma')$ satisfies (H3)$_p$ and (H4)$_p$.
Therefore, the eTNC$_p^-(H/F)$ is expected under the condition.

Now we begin the discussion on the integrality of $\theta_{\Sigma}^{\Sigma'}$.
The essential ingredient is the result of Deligne and Ribet \cite{DR80} or Cassou-Nogu\`{e}s \cite{CN79}.
In order to apply it, we have to strengthen the condition (H3), (H3)$_p$, and (H3)$_p^{\chi}$ (when $p$ is an odd prime number and $\chi$ is an odd character of $G'$) to the following:
\begin{itemize}
\item[(H3')]
$\mu(H)^{\Sigma_{\ur}', -}$ vanishes, 
where we put $\Sigma_{\ur}' = \Sigma' \setminus (S_{\ram}(H/F) \cap \Sigma')$.
\item[(H3')$_p$]
$\mu_{p^{\infty}}(H)^{\Sigma_{\ur}'}$ vanishes,
where $\mu_{p^{\infty}}(H)^{\Sigma_{\ur}'}$ is the $p$-primary component of $\mu(H)^{\Sigma_{\ur}'}$.
\item[(H3')$_p^{\chi}$]
$\mu(H)^{\Sigma_{\ur}', \chi} = \mu_{p^{\infty}}(H)^{\Sigma_{\ur}', \chi}$ vanishes.
\end{itemize}
Clearly these three conditions respectively imply (H3), (H3)$_p$, and (H3)$_p^{\chi}$.
Moreover, (H3') is equivalent to that (H3')$_p$ holds for any $p$, and (H3')$_p$ is equivalent to that (H3')$_p^{\chi}$ holds for any odd character $\chi$ of $G'$.

Now we have the following integrality property.

\begin{prop}\label{prop:theta_int}
Let $\Sigma$ and $\Sigma'$ be finite sets of places of $F$ satisfying 
the conditions (H1) and (H2).
\begin{itemize}
\item[(1)]
Let $p$ be an odd prime number and let $\chi$ be an odd character of $G'$.
Suppose the conditions (H3')$_p^{\chi}$ and (H4)$_p^{\chi}$. 
Then we have $\theta_{\Sigma}^{\Sigma', \chi} \in \Z_p[G]^{\chi}$.
\item[(2)]
Let $p$ be an odd prime number.
Suppose the conditions (H3')$_p$ and (H4)$_p$. 
Then we have $\theta_{\Sigma}^{\Sigma'} \in \Z_p[G]^{-}$.
\item[(3)]
Suppose the conditions (H3') and (H4). 
Then we have $\theta_{\Sigma}^{\Sigma'} \in \Z[G]^{-}$.
\end{itemize}
\end{prop}

\begin{proof}
It is clear that the claim (1) implies the claims (2) and (3).
Let us prove (1).
Put 
\[
\Theta_{\Sigma}^{\Sigma'}(H/F) 
= \prod_{v \in \Sigma_f} \parenth{1 - \frac{\nu_{I_v}}{\# I_v} \varphi_v^{-1}}
 \cdot \omega^{\Sigma'}
 \in \Q[G]^-.
\]
Then, as in \cite[Formula (2) and Remark 3.6]{DK20} (also see \cite[Proposition 2.1]{Nickel21}), 
the celebrated theorem of Deligne and Ribet \cite{DR80} or Cassou-Nogu\`{e}s \cite{CN79} implies that
\[
\Theta_{\Sigma}^{\Sigma'}(H/F)^{\chi} \in \Z_p[G]^{\chi}
\]
as long as the conditions (H1), (H2), (H3')$_p^{\chi}$, and (H4)$_p^{\chi}$ hold.
We shall deduce the proposition from this integrality of $\Theta_{\Sigma}^{\Sigma'}(H/F)^{\chi}$.

By \eqref{eq:hv} and \eqref{eq:theta_defn}, we obtain
\[
\theta_{\Sigma}^{\Sigma'} 
= \sum_{J \subset \Sigma_f} \parenth{\prod_{v \in J} \nu_{I_v}} \prod_{v \in \Sigma_f \setminus J} \parenth{1 - \frac{\nu_{I_v}}{\# I_v} \varphi_v^{-1}} \omega^{\Sigma'},
\]
where $J$ runs over all (possibly empty) subsets of $\Sigma_f$.
Then it is enough to show the integrality of the $\chi$-component of the term in the right hand side for each $J$.
For each $J$, let us define an intermediate field $H^J$ of $H/F$ by $\Gal(H/H^J) = \sum_{v \in J} I_v \subset \Gal(H/F)$.
Then we have
\[
\parenth{\prod_{v \in J} \nu_{I_v}} \prod_{v \in \Sigma_f \setminus J} \parenth{1 - \frac{\nu_{I_v}}{\# I_v} \varphi_v^{-1}} \omega^{\Sigma'}
= \frac{\prod_{v \in J} \# I_v}{[H: H^J]} \nu_{H/H^J} \Theta_{\Sigma \setminus J}^{\Sigma'}(H^J/F),
\]
where $\nu_{H/H^J}$ denotes the norm element of $\Gal(H/H^J)$.
If $\chi$ is non-trivial on $\Gal(H/H^J)$, then we have $\nu_{H/H^J}^{\chi} = 0$, so we can ignore this case.
Let us suppose that $\chi$ is trivial on $\Gal(H/H^J)$.
Then $\chi$ can be regarded as an odd character of $H^J/F$, and the conditions (H3')$_p^{\chi}$ and (H4)$_p^{\chi}$ still holds for the extension $H^J/F$ and the pair $(\Sigma \setminus J, \Sigma')$. 
Therefore, we have the integrality of $\Theta_{\Sigma \setminus J}^{\Sigma'}(H^J/F)^{\chi}$, from which the conclusion follows.
\end{proof}

\subsection{An application of the analytic class number formula}\label{ss:divisibility}

Let $H/F$ be a finite abelian CM-extension and put $G = \Gal(H/F)$.
Let $p$ be an odd prime number.
For a pair $(\Sigma, \Sigma')$ of finite sets of places of $F$ satisfying (H1) and (H2), we simply write
\[
{}_p\Omega_{\Sigma}^{\Sigma'} = \Z_p \otimes_{\Z} \Omega_{\Sigma}^{\Sigma'}.
\]
The aim of this subsection is to prove the following proposition, which allows us to deduce the conclusions of Theorems \ref{thm:current_main} and \ref{thm:main_p_2} from single inclusions.

\begin{prop}\label{prop:divisibility3}
Let $\Sigma$ and $\Sigma'$ be finite sets of places of $F$ satisfying (H1), (H2), (H3')$_p$, and (H4)$_p$.
\begin{itemize}
\item[(1)]
If we have
\[
\Fitt_{\Z_p[G]^-}({}_p\Omega_{\Sigma}^{\Sigma'}) \subset (\theta_{\Sigma}^{\Sigma'})
\]
 as ideals of $\Z_p[G]^-$, then the equality also holds.
\item[(2)]
Let $N$ be a subgroup of $G$ and recall the algebra $\Z_p[G]_{(N)}$ introduced before Theorem \ref{thm:main_p_2}.
If we have
\[
\Fitt_{\Z_p[G]^-}({}_p\Omega_{\Sigma}^{\Sigma'}) \cdot \Z_p[G]^-_{(N)}
\subset \theta_{\Sigma}^{\Sigma'} \cdot \Z_p[G]^-_{(N)}
\]
 as ideals of $\Z_p[G]^-_{(N)}$, then the equality also holds.
\end{itemize}
\end{prop}

For the proof of this proposition, we apply a standard method using the analytic class number formula (e.g.,~Greither \cite[Theorem 4.11]{Grei00}), following \cite[\S 2]{DK20}.
Note that the situation is actually easier than \cite{DK20} because we have only to deal with finite modules and non-zero-divisors.

In Lemmas \ref{lem:Fitt_ideal} and \ref{lem:Fitt_ideal2} below, we consider a general finite abelian group $G$.
Let $G'$ its maximal subgroup of order prime to $p$, and $\chi$ a character of $G'$. 
We consider the set of $\ol{\Q_p}$-valued characters of $G$ whose restrictions to $G'$ coincide with $\chi$.
We have an equivalence relation $\sim$ on the set defined as $\psi_1 \sim \psi_2$ if and only if $\psi_2 = \sigma \psi_1$ for some $\sigma \in \Gal(\ol{\Q_p}/\Q_p)$.
We fix a complete system of representatives of the equivalence classes modulo $\sim$, and write $\Psi_{\chi}$ for the set of representatives.
Then we have a natural injective algebra homomorphism with finite cokernel
\[
\Z_p[G]^{\chi} \hookrightarrow \prod_{\psi \in \Psi_\chi} \OO_{\psi},
\]
where we put $\OO_{\psi} = \Z_p[\Imag(\psi)]$.
Let $\Psi$ be any subset of $\Psi_\chi$.
We define an algebra $\Z_p[G]^{\Psi}$ as the image of the natural map
\[
\Z_p[G]^{\chi} \to \prod_{\psi \in \Psi} \OO_{\psi}.
\]
Note that $\Z_p[G]^{\Psi}$ is a local ring unless $\Psi$ is empty.

\begin{lem}\label{lem:Fitt_ideal}
Let $N$ be a finite $\Z_p[G]^{\Psi}$-module with $\pd_{\Z_p[G]^{\Psi}}(N) \leq 1$.
Then we have
\[
\# N 
= \prod_{\psi \in \Psi} \# \parenth{ \OO_{\psi}/\Fitt_{\OO_{\psi}}(N \otimes_{\Z_p[G]^{\Psi}} \OO_{\psi})}.
\]
\end{lem}

\begin{proof}
See \cite[Lemmas 2.4 and 2.5]{DK20}.
\end{proof}

\begin{lem}\label{lem:Fitt_ideal2}
Let $N_1$ and $N_2$ be finite $\Z_p[G]^{\Psi}$-modules with $\pd_{\Z_p[G]^{\Psi}}(N_i) \leq 1$ for $i = 1, 2$.
Then we have $\Fitt_{\Z_p[G]^{\Psi}}(N_1) = \Fitt_{\Z_p[G]^{\Psi}}(N_2)$ 
if and only if both $\Fitt_{\Z_p[G]^{\Psi}}(N_1) \subset \Fitt_{\Z_p[G]^{\Psi}}(N_2)$ and $\# N_1 = \# N_2$ hold.
\end{lem}

\begin{proof}
If $\Psi$ is empty, then the claim is trivial, so let us assume that $\Psi$ is non-empty.
The ``only if'' part immediately follows from Lemma \ref{lem:Fitt_ideal}.
For the ``if'' part, let $\lambda \in \Z_p[G]^{\Psi}$ be an element such that $\Fitt_{\Z_p[G]^{\Psi}}(N_1) = \lambda \Fitt_{\Z_p[G]^{\Psi}}(N_2)$.
Then by Lemma \ref{lem:Fitt_ideal} and $\# N_1 = \# N_2$, we must have $\psi(\lambda) \in \OO_{\psi}^{\times}$ for any $\psi \in \Psi$.
This implies $\lambda \in (\Z_p[G]^{\Psi})^{\times}$ since, for any $\psi \in \Psi$, the algebra homomorphism $\Z_p[G]^{\Psi} \to \OO_{\psi}$ is a local homomorphism.
\end{proof}

We now return to the arithmetic setting.
Note that we have an isomorphism
\[
{}_p\Omega_{\Sigma}^{\Sigma'}
\simeq \prod_{\chi} \Omega_{\Sigma}^{\Sigma', \chi},
\]
where $\chi$ runs over all equivalence classes of odd characters of $G'$.
By using the analytic class number formula, we obtain the following.

\begin{prop}\label{prop:order_equal}
Let $\Sigma$ and $\Sigma'$ be finite sets of places of $F$ satisfying (H1), (H2), (H3')$_p$, and (H4)$_p$.
Then we have
\[
\# {}_p\Omega_{\Sigma}^{\Sigma'}
= \# (\Z_p[G]^- / (\theta_{\Sigma}^{\Sigma'})).
\]
\end{prop}

\begin{proof}
We take a finite extension of $\Q_p$ that contains the values of all the $\ol{\Q_p}$-valued characters of $G$, and let $\OO$ be its ring of integers.
Then we have a natural homomorphism
\[
\OO[G]^- \to \prod_{\psi} \OO,
\]
is injective with finite cokernel.
Here, and in the rest of the proof, $\psi$ runs over the odd characters of $G$.

We aim at showing
$\# (\OO \otimes_{\Z} \Omega_{\Sigma}^{\Sigma'})
= \# (\OO[G]^- / (\theta_{\Sigma}^{\Sigma'})).$
By the sequence \eqref{eq:fund}, we have 
\[
\# (\OO \otimes \Omega_{\Sigma}^{\Sigma'})
= \prod_{v \in \Sigma_f} \# (\OO \otimes A_v^-) \cdot \# (\OO \otimes \Cl_H^{\Sigma', -}).
\]
On the other hand, we have
\[
\# \parenth{\OO[G]^- / (\theta_{\Sigma}^{\Sigma'})}
= \prod_{\psi} \# \parenth{\OO / (\psi(\theta_{\Sigma}^{\Sigma'}))}.
\]
By \eqref{eq:theta_defn}, this implies
\[
\# \parenth{\OO[G]^- / (\theta_{\Sigma}^{\Sigma'})}
= \prod_{v \in \Sigma_f} \prod_{\psi} \# \parenth{\OO / (\psi(h_v))}
\times \prod_{\psi} 
\# \parenth{\OO / \parenth{\psi \parenth{\omega^{\Sigma'}}}}.
\]

We claim that
\[
\# (\OO \otimes A_v^-) = \prod_{\psi} \# \parenth{\OO / (\psi(h_v))}
\]
holds for each $v \in \Sigma_f$.
By the definition of $A_v$, we have
\[
\# (\OO \otimes A_v^-) = \prod_{\psi} \# (\OO/(1 - \psi(\varphi_v)^{-1} + \# I_v)),
\]
where $\psi$ runs over the odd characters that are trivial on $I_v$.
By the definition of $h_v$, we have $\psi(h_v) = 1 - \psi(\varphi_v)^{-1} + \# I_v$ if $\psi$ is trivial on $I_v$, and $\psi(h_v) = 1$ otherwise.
Then we obtain the claim.

The analytic class number formula (cf. \cite[Lemma 2.1]{DK20}) implies
\begin{align}
\# (\OO \otimes \Cl_H^{\Sigma', -})
= \prod_{\psi} 
\# \parenth{\OO / \parenth{L^{\Sigma'}(\psi^{-1}, 0)}}.
\end{align}
By the definition of $\omega^{\Sigma'}$, we have
$L^{\Sigma'}(\psi^{-1}, 0) = \psi \parenth{\omega^{\Sigma'}}$
for each odd character $\psi$ of $G$.

Incorporating these formulas, we obtain the proposition.
\end{proof}

Now we are ready to prove Proposition \ref{prop:divisibility3}.

\begin{proof}[Proof of Proposition \ref{prop:divisibility3}]
(1)
Suppose that $\Fitt_{\Z_p[G]^-}({}_p\Omega_{\Sigma}^{\Sigma'}) \subset (\theta_{\Sigma}^{\Sigma'})$ holds.
Then by Lemma \ref{lem:Fitt_ideal}, for all odd characters $\chi$ of $G'$, we have
\[
 \# (\Z_p[G]^{\chi} / (\theta_{\Sigma}^{\Sigma', \chi})) \mid \# \Omega_{\Sigma}^{\Sigma', \chi}.
\]
Then by Proposition \ref{prop:order_equal}, the divisibility of the orders must be an equality for all $\chi$.
Now we can apply Lemma \ref{lem:Fitt_ideal2} and obtain the claim (1).

(2)
It is easy to see (e.g., by Lemma \ref{lem:Fitt_ideal})
\[
\# \parenth{\Z_p[G]^-/(\theta_{\Sigma}^{\Sigma'})}
= \# \parenth{\Z_p[G]^-_{(N)}/(\theta_{\Sigma}^{\Sigma'})}
\]
and
\[
\# {}_p\Omega_{\Sigma}^{\Sigma'}
= \# \parenth{\Z_p[G]^-_{(N)} \otimes_{\Z_p[G]^-} {}_p\Omega_{\Sigma}^{\Sigma'}}.
\]
By Proposition \ref{prop:order_equal}, these formulas imply
\begin{equation}\label{eq:ord1}
\# \parenth{\Z_p[G]^-_{(N)}/(\theta_{\Sigma}^{\Sigma'})}
= \# \parenth{\Z_p[G]^-_{(N)} \otimes_{\Z_p[G]^-} {}_p\Omega_{\Sigma}^{\Sigma'}}.
\end{equation}


For each odd character $\chi$ of $G'$, 
we introduce $\Psi_{\chi}$ as before and 
let $\Psi_{\chi, 1}$ (resp.~$\Psi_{\chi, 2}$) be the subset of $\Psi_\chi$ whose elements are trivial on $N$ (resp.~non-trivial on $N$). 
Then we have 
\[
(\Z_p[G/N])^{\chi}
\simeq \Z_p[G]^{\Psi_{\chi, 1}}
\]
and
\[
(\Z_p[G]/(\nu_N))^{\chi}
\simeq \Z_p[G]^{\Psi_{\chi, 2}}.
\]
By the assumption and Lemma \ref{lem:Fitt_ideal}, we have
\[
 \# \parenth{\Z_p[G]^{\Psi_{\chi, i}} / (\theta_{\Sigma}^{\Sigma', \Psi_{\chi, i})}}
 \mid \# \parenth{\Z_p[G]^{\Psi_{\chi, i}} \otimes_{\Z_p[G]^-} {}_p\Omega_{\Sigma}^{\Sigma'}}
\]
for $i \in \{1, 2\}$.
Here, $\theta_{\Sigma}^{\Sigma', \Psi_{\chi, i}}$ denotes the image of $\theta_{\Sigma}^{\Sigma'}$ to $\Z_p[G]^{\Psi_{\chi, i}}$.
Then by \eqref{eq:ord1}, the divisibility of the orders must be an equality for all $\chi$ and $i$.
Now we can apply Lemma \ref{lem:Fitt_ideal2} and obtain the claim (2).
\end{proof}

\subsection{Deducing Theorems \ref{thm:current_main} and \ref{thm:main_p_2} from Theorem \ref{thm:current_main_p}}\label{ss:pf_thm}

In this subsection, assuming the validity of Theorem \ref{thm:current_main_p}, we prove Theorems \ref{thm:current_main} and \ref{thm:main_p_2}.
Let $H/F$ be a finite abelian CM-extension and $p$ an odd prime number.
Put $G = \Gal(H/F)$ and define $G'$ as usual.
For each $(\Sigma, \Sigma')$, we keep the notation ${}_p\Omega_{\Sigma}^{\Sigma'} = \Z_p \otimes_{\Z} \Omega_{\Sigma}^{\Sigma'}$.

\subsubsection{The proof of Theorem \ref{thm:main_p_2}}\label{sss:pf_thm1}

We first prove Theorem \ref{thm:main_p_2} from Theorem \ref{thm:current_main_p}

\begin{proof}[Proof of Theorem \ref{thm:main_p_2}]
Thanks to Proposition \ref{prop:Om_var}, the statement of Theorem \ref{thm:main_p_2} is independent from the choice of $(\Sigma, \Sigma')$.
Therefore, we may assume that the conditions (H1), (H2), (H3')$_p$, and (H4)$_p$ hold (we do not need to assume (H3) or (H4)).

By Proposition \ref{prop:divisibility3}(2), we have only to show
\[
\Fitt_{\Z_p[G]^-}({}_p\Omega_{\Sigma}^{\Sigma'}) \cdot \Z_p[G]_{(N)}^-
\subset \theta_{\Sigma}^{\Sigma'} \cdot \Z_p[G]_{(N)}^-.
\]
By the definition of $\Z_p[G]_{(N)}$, this is equivalent to both
\begin{equation}\label{eq:a10}
\Fitt_{\Z_p[G]^-}({}_p\Omega_{\Sigma}^{\Sigma'}) \cdot \Z_p[G/N]^-
\subset \theta_{\Sigma}^{\Sigma'} \cdot \Z_p[G/N]^-
\end{equation}
and
\begin{equation}\label{eq:a11}
\Fitt_{\Z_p[G]^-}({}_p\Omega_{\Sigma}^{\Sigma'}) \cdot \parenth{\Z_p[G]^-/(\nu_{N}^-)}
\subset \theta_{\Sigma}^{\Sigma'} \cdot \parenth{\Z_p[G]^-/(\nu_{N}^-)}
\end{equation}
hold.

Let us show \eqref{eq:a11}.
For any odd character $\chi$ of $G'$, 
recall that we have the integrality of $\theta_{\Sigma}^{\Sigma', \chi}$ by Proposition \ref{prop:theta_int}.
Therefore, by $(\nu_{I_p}^{\chi})_{\Frac(\Z_p[G]^{\chi})} \cap \Z_p[G]^{\chi} = (\nu_{I_p}^{\chi})_{\Z_p[G]^{\chi}}$, Theorem \ref{thm:current_main_p}(2) implies
\[
\Fitt_{\Z_p[G]^{\chi}}(\Omega_{\Sigma}^{\Sigma', \chi}) 
\subset (\theta_{\Sigma}^{\Sigma', \chi}, \nu_{I_p}^{\chi})
\]
as ideals of $\Z_p[G]^{\chi}$.
By the assumption $N \subset I_p$, we have $(\nu_{I_p}) \subset (\nu_N)$.
Then we obtain \eqref{eq:a11}.

Let us check \eqref{eq:a10}.
If $H^{N}$ is totally real (i.e., the complex conjugation lies in $N$),
then $\Z_p[G/N]^- = 0$, so the claim is trivial.
Otherwise, by \eqref{eq:theta_bar} and Proposition \ref{prop:fld_var}, the claim is equivalent to
\[
\parenth{\Fitt_{\Z_p[G/N]^-}({}_p\Omega_{\Sigma, H^{N}}^{\Sigma'})}
\subset (\theta_{\Sigma, H^{N}}^{\Sigma'})
\]
as ideals of $\Z_p[G/N]^-$.
By the choice of $N$, at least one $p$-adic prime is (at most) tamely ramified in $H^{N}/F$, so
this is true by Theorem \ref{thm:current_main_p}(1).
\end{proof}

\subsubsection{The proof of Theorem \ref{thm:current_main}(i)(ii)}\label{ss:deduction(i)}

We first observe elementary lemmas.

\begin{lem}\label{lem:H4}
Let $N \subset G'$ be a subgroup.
Then $H^N$ is a CM-field if and only if there exists an odd character $\chi$ of $G'$ which is trivial on $N$.
\end{lem}

\begin{proof}
The field $H^N$ is a CM-field if and only if $N$ does not contain the complex conjugation $c \in G$.
If there exists an odd character $\chi$ of $G'$ which is trivial on $N$, then $\chi(c) = -1$ implies that $c \not \in N$.
If $c \not \in N$, we may find a character $\chi$ of $G'$ which is trivial on $N$ and $\chi(c) = -1$.
\end{proof}

\begin{lem}\label{lem:deduction1}
The following are equivalent.
\begin{itemize}
\item[(a)]
For any odd character $\chi$ of $G'$, the following is false: 
the decomposition group of all $p$-adic primes in $H^{\chi}/F$ are $p$-groups
and all $p$-adic primes are ramified in $H^{\chi}/F$
(recall that $H^{\chi} = H^{\Ker(\chi)}$).
\item[(b)]
Either (i) or (ii) of Theorem \ref{thm:current_main}
holds. 
\end{itemize}
\end{lem}

\begin{proof}
Let us show the equivalence of $\neg$(a) and $\neg$(b), where $\neg$ denotes the negation.
Put $F_p = H^{G'}$.
Let $\chi$ be an odd character of $G'$ and $\fp$ a $p$-adic prime of $F$.
Then the condition that ``the decomposition group of $\fp$ in $H^{\chi}/F$ is a $p$-group and $\fp$ is ramified in $H^{\chi}/F$'' 
is equivalent to that ``$\fp$ is totally split in $H^{\chi}/F_p$ and $\fp$ is ramified in $F_p/F$''.
Therefore, $\neg$(a) is equivalent to that both ($\alpha$) and ($\beta$) hold, where:
\begin{itemize}
\item[($\alpha$)]
Any $p$-adic prime is ramified in $F_p/F$.
\item[($\beta$)]
There exists an odd character $\chi$ of $G'$ such that any $p$-adic prime is totally split in $H^{\chi}/F_p$.
\end{itemize}
Clearly ($\alpha$) is equivalent to $\neg$(i).
It is then enough to show that ($\beta$) is equivalent to $\neg$(ii).
Let $G_p = \sum_{\fp \in S_p(F)}G_{\fp} \subset G$ be the sum of the decomposition groups and put $G_p' = G_p \cap G'$.
Then the condition ($\beta$) is equivalent to that there exists an odd character $\chi$ of $G'$ that is trivial on $G_p'$.
By Lemma \ref{lem:H4}, we see that ($\beta$) is equivalent to that $H^{G_p'}$ is a CM-field.
Since $G_p / G_p'$ is a $p$-group, $H^{G_p'}$ is a CM-field if and only if so is $H^{G_p}$.
Thus ($\beta$) is equivalent to $\neg$(ii).
This completes the proof.
\end{proof}

\begin{proof}[Proof of Theorem \ref{thm:current_main}(i)(ii)]

Assume either (i) or (ii) in Theorem \ref{thm:current_main} is true.
In that case, by Lemma \ref{lem:deduction1}, we may apply Theorem \ref{thm:current_main_p}(1) for any odd character $\chi$ of $G'$, so we obtain one inclusion of the eTNC$_p^{\chi}$.
Then by Proposition \ref{prop:divisibility3}(1), we obtain the eTNC$_p^-$.
\end{proof}

\subsubsection{The proof of Theorem \ref{thm:current_main}(iii)}\label{ss:deduction(iii)}

The proof of Theorem \ref{thm:current_main}(iii) is not so direct as the others.
We assume the condition (iii), that is, $H^{\cl, +} (\mu_p) \not\supset H^{\cl}$.
The key idea is the technique of Wiles \cite{Wil90} on avoiding the trivial zeros.
The technique makes use of the following lemma (in which the field $F$ is not concerned).

\begin{lem}[{\cite[Proposition 4.1]{Grei00}}]\label{lem:trivial_zero}
Assume that $H^{\cl, +} (\mu_p) \not\supset H^{\cl}$.
Then, for any positive integer $n$,
there are infinitely many prime numbers $l$ satisfying the following conditions. 
\begin{itemize}
\item[(1)]
$H / \Q$ is unramified at any prime above $l$.
\item[(2)]
$l \equiv 1 \pmod{p^n}$.
\item[(3)]
No prime of $H^+$ above $l$ splits in $H$. 
\item[(4)]
The prime $p$ is inert in the extension $E_l/\Q$,
where $E_l$ is the subfield of $\Q (\mu_l)$ such that $[E_l : \Q] = p^n$. 
\end{itemize}
\end{lem}

\begin{proof}
For the readers' convenience, let us explain the key ideas.
We consider the extension $H^{\cl}(\mu_{p^n}, p^{1/p})/H^{\cl, +}(\mu_{p^n})$, which is cyclic of order $2p$, thanks to the assumption $H^{\cl, +} (\mu_p) \not\supset H^{\cl}$.
By Chebotarev's density theorem, we find infinitely many prime ideals $\fL$ of $H^{\cl, +}(\mu_{p^n})$ whose degrees are one and that are inert in the extension.
Then the prime number $l$ lying below $\fL$ satisfies the conditions.
Shortly speaking, (1) is easy, being totally split in $H^{\cl, +}(\mu_{p^n})/\Q$ ensures (2), and being inert in $H^{\cl}(\mu_{p^n})/H^{\cl, +}(\mu_{p^n})$ (resp.~in $H^{\cl, +}(\mu_{p^n}, p^{1/p})/H^{\cl, +}(\mu_{p^n})$) ensures (3) (resp.~(4)).
\end{proof}

In order to show Theorem \ref{thm:current_main}(iii),
we may and do assume that (i) fails, that is, all $p$-adic primes of $F$ are wildly ramified in $H$.
Let $\Sigma$ and $\Sigma'$ be finite sets of places of $F$ satisfying (H1), (H2), (H3')$_p$, and (H4)$_p$.
Also we may assume that $\Sigma \supset S_p(F)$. 

We write $I_{v, H}$, $h_{v, H}$, $\theta_{\Sigma, H}^{\Sigma'}$, and $\Omega_{\Sigma, H}^{\Sigma'}$ for 
$I_{v}$, $h_{v}$, $\theta_{\Sigma}^{\Sigma'}$, and $\Omega_{\Sigma}^{\Sigma'}$
in order to clarify the relevant field (the base field $F$ is always unchanged).
Take an integer $r \geq 0$ such that
\begin{equation}\label{eq:def r}
\frac{p^r}{[H_w: \Q_p]} \in \Z_p
\end{equation}
for any $p$-adic prime $w$ of $H$.
Since $\theta_{\Sigma, H}^{\Sigma'}$ is a non-zero-divisor in $\Z_p[G]^-$, 
we can take a positive integer  $n > r$ such that
\begin{equation}\label{eq:def n}
p^{n-r} \in (\theta_{\Sigma, H}^{\Sigma'})
\end{equation}
holds in $\Z_p[G]^-$.
For this $n$, 
we take a prime number $l$ such that (1) -- (4) in Lemma \ref{lem:trivial_zero} hold.
We moreover require that $S_l \cap (\Sigma_f \cup \Sigma') = \emptyset$, where $S_l = S_l(F)$ denotes the set of $l$-adic primes of $F$.
We write $H_l = H E_l$ and $F_l = F E_l$. 
Note that $H \cap F_l = F$ since each prime above $l$ is totally ramified in $F_l / F$ and is unramified in $H/F$.
This shows that we have
\[
\Gal(H_l/F_l) \times \Gal(H_l/H) = \Gal(H_l/F).
\]

\begin{lem}[{\cite[Lemma 10.2]{Wil90}}]\label{lem:ATZ1}
Let $\psi$ be a character of $\Gal(H_l/F)$ whose restriction to $\Gal(H_l/H)$ has order $> p^r$.
Then we have $\psi(v) \neq 1$ for any $v \in S_p(F)$.
\end{lem}

\begin{proof}
We may assume that $\psi$ is unramified at $v$, since otherwise $\psi(v) = 0$.
We can decompose $\psi = \psi_1 \psi_2$, 
where $\psi_1$ (resp.~$\psi_2$) is a character of $\Gal(F_l/F)$ (resp.~$\Gal(H/F)$).
Note that, by (4) in Lemma \ref{lem:trivial_zero}, we have
\[
\ord(\Frob_v(F_l/F)) \geq \frac{p^n}{[F_v: \Q_p]_p}
\]
for each $v \in S_p(F)$, where $[F_v: \Q_p]_p$ denotes the largest $p$-power dividing $[F_v: \Q_p]$.
Then, by combining with the assumption, $\psi_1(v)$ is a $p$-power root of unity with
\[
\ord(\psi_1(v)) > \frac{p^r}{[F_v: \Q_p]_p}.
\]
On the other hand, we have $\ord(\psi_2(v)) \mid [H_w: F_v]$, where $w$ is a prime of $H$ lying above $v$.
If $\psi(v) = 1$, we must have $\ord(\psi_1(v)) = \ord(\psi_2(v))$, so the above formulas imply
\[
[H_w: F_v]_p > \frac{p^r}{[F_v: \Q_p]_p},
\]
where the the left hand side is defined similarly.
This contradicts the choice of $r$ in \eqref{eq:def r}, so the lemma follows.
\end{proof}

We define $\nu \in \Z_p[\Gal(H_l / F)]$ as the norm element of $\Gal(H_l/H)^{p^r}$.
Concretely, if $\tau$ is a generator of $\Gal(H_l / H)$, we put
\[
\nu = \sum_{j = 0}^{p^{n - r} - 1} \tau^{p^r j}
 \in \Z_p[\Gal(H_l / F)].
\]
Note that the condition on $\psi$ in Lemma \ref{lem:ATZ1} is equivalent to that $\psi(\nu) = 0$.

As an intermediate to the final goal, we prove the following.

\begin{claim}\label{claim:ATZ2}
We have
\begin{equation}
\Fitt_{\Z_p [\Gal (H_l / F)]^-} \left( {}_p\Omega_{\Sigma, H_l}^{\Sigma' \cup S_l} \right)
\subset \left(\theta_{\Sigma , H_l}^{\Sigma' \cup S_l} , \nu \right)^- . 
\end{equation} 
\end{claim}

\begin{proof}

For an integer $m \geq 0$, let $H_{l,m}$ be the $m$-th layer of the cyclotomic $\Z_p$-extension of $H_l$. 
We take an integer $m \geq 0$ such that
\begin{equation}\label{eq:def m}
\cfrac{\# I_{p, H_{l,m} / H_l}}{\# \Gal(H_l / F)} 
\in (\theta_{\Sigma, H_l}^{\Sigma' \cup S_l}) 
\end{equation}
holds in $\Z_p[\Gal (H_l / F)]^-$, where we write $I_{p, H_{l,m} / H_l} = \sum_{\fp \in S_p(F)} I_{\fp, H_{l,m} / H_l} \subset \Gal(H_{l, m}/H_l)$ for the sum of the inertia subgroups.
Then, by Theorem \ref{thm:current_main_p}(2), we have
\begin{equation}\label{eq:inclusion1}
 \Fitt_{\Z_p [\Gal(H_{l,m} / F)]^-} \left( {}_p\Omega_{\Sigma, H_{l,m}}^{\Sigma'\cup S_l } \right) \subset 
\left( \theta_{\Sigma , H_{l,m}}^{\Sigma' \cup S_l} , \nu_{I_{p, H_{l,m} / H_l}} \right)^- .
\end{equation}
Let
\[
\pi: \Q_p [\Gal (H_{l,m} / F)] \to \Q_p [\Gal (H_l / F)]
\]
be the natural map.
Note that $\pi(\nu_{I_{p, H_{l,m} / H_l}}) = p^a$ if we put $p^a = \# I_{p, H_{l,m} / H_l}$.
Then by \eqref{eq:theta_bar} and Proposition \ref{prop:fld_var}, we can deduce from \eqref{eq:inclusion1} that  
\begin{equation}\label{eq:inclusion3}
\Fitt_{\Z_p [\Gal (H_l / F)]^-} \left( {}_p\Omega_{\Sigma, H_l}^{\Sigma' \cup S_l} \right)
\subset \left(\theta_{\Sigma , H_l}^{\Sigma' \cup S_l} , \prod_{v \in S_p(F)} \cfrac{h_{v, H_l}}{\pi \parenth{ h_{v, H_{l,m}}}} \cdot p^a \right)^- .
\end{equation}
In the product, $v$ only has to run over the $p$-adic primes, since $h_{v, H_l } = \pi \parenth{ h_{v, H_{l,m}}}$ for $v \not \in S_p(F)$, for $H_{l, m}/H_l$ is unramified at $v$.

For a character $\psi$ of $\Gal(H_l/F)$ and $v \in S_p(F)$, by the definition of $h_v$ in \eqref{eq:hv}, we have
\[
\psi \parenth{\cfrac{h_{v, H_l}}{\pi \parenth{ h_{v, H_{l,m}}}}}
= \cfrac{1 - \psi(v)^{-1} + \# I_{v, H_l}}{1 - \psi(v)^{-1} + \# I_{v, H_{l, m}}}
\]
if $\psi$ is unramified at $v$, and otherwise the left hand side is $1$.
Recall that we assume $\neg$(i), so the orders of the inertia groups on the right hand side are divisible by $p$.
Therefore, as long as $\psi(v) \neq 1$, we have an integrality
\[
\psi \parenth{\cfrac{h_{v, H_l}}{\pi \parenth{ h_{v, H_{l,m}}}}}
\in \Z_p[\Imag (\psi)].
\]
In particular, by Lemma \ref{lem:ATZ1}, this integrality holds if $\psi(\nu) = 0$.
Therefore, we have 
\[
\# \Gal(H_l / F) \cdot \prod_{v \in S_p(F)} \cfrac{h_{v, H_l}}{\pi \parenth{ h_{v, H_{l,m}}}}
 \in \Z_p[\Gal(H_l / F)] + \Q_p[\Gal(H_l/F)] \nu.
\]
This, together with \eqref{eq:def m}, shows
\[
p^a \cdot \parenth{\prod_{v \in S_p(F)} \cfrac{h_{v, H_l}}{\pi \parenth{ h_{v, H_{l,m}}}}}^-
 \in (\theta_{\Sigma, H_l}^{\Sigma' \cup S_l}) + \Q_p[\Gal(H_l/F)]^- \nu.
\]
Then \eqref{eq:inclusion3} implies
\[
\Fitt_{\Z_p [\Gal (H_l / F)]^-} \left( {}_p\Omega_{\Sigma, H_l}^{\Sigma' \cup S_l} \right)
\subset (\theta_{\Sigma, H_l}^{\Sigma' \cup S_l}) + \Q_p[\Gal(H_l/F)]^- \nu.
\]
Note that both the left hand side and $(\theta_{\Sigma, H_l}^{\Sigma' \cup S_l})$ are contained in $\Z_p [\Gal (H_l / F)]^-$, and that we also have
\[
\Q_p[\Gal(H_l/F)] \nu \cap \Z_p[\Gal(H_l/F)] = \Z_p[\Gal(H_l/F)] \nu.
\]
Therefore, the claim follows.
\end{proof}

Using Claim \ref{claim:ATZ2}, we finish the proof of Theorem \ref{thm:current_main}, assuming (iii) and $\neg$(i).
From now on, $\pi$ denotes the natural map
\[
\pi: \Z_p [\Gal (H_l / F)] \to \Z_p [G].
\]
Note that $\pi(\nu) = p^{n - r}$.
Then, using \eqref{eq:theta_bar} and Proposition \ref{prop:fld_var}, we deduce from Claim \ref{claim:ATZ2} that
\begin{equation}\label{eq:inclusion4} 
\Fitt_{\Z_p [G]^-} \left( {}_p \Omega_{\Sigma, H}^{\Sigma' \cup S_l} \right)
\subset \left( \theta_{\Sigma , H}^{\Sigma' \cup S_l} , p^{n-r} \right)^- . 
\end{equation}
Here, we used the fact that any finite prime $v \notin S_l$ (in particular any $v \in \Sigma_f$) is unramified in $H_l / H$.
On the other hand, by \eqref{eq:theta_h_h'} and Proposition \ref{prop:Om_var}(2), 
we have 
\begin{align}
\Fitt_{\Z_p [G]^-} \left( {}_p \Omega_{\Sigma, H}^{\Sigma' \cup S_l} \right)
& = \prod_{v \in S_l} \left( 1-N(v) \varphi_v^{-1} \right)^- \cdot \Fitt_{\Z_p [G]^-} \left({}_p \Omega_{\Sigma, H}^{\Sigma'} \right)
\end{align}
and the corresponding formula for the Stickelberger elements.
By the conditions (2) and (3) in Lemma \ref{lem:trivial_zero}, 
we know that $\prod_{v \in S_l} \left( 1-N(v) \varphi_v^{-1} \right)^-$ is a unit in $\Z_p[G]^-$. 
Therefore, \eqref{eq:inclusion4} implies that 
$$ \Fitt_{\Z_p [G]^-} \left({}_p \Omega_{\Sigma, H}^{\Sigma'} \right) 
\subset \left( \theta_{\Sigma, H}^{\Sigma'} , p^{n-r} \right)^- .$$
By \eqref{eq:def n}, the right hand side equals to $(\theta_{\Sigma, H}^{\Sigma'})$.
Therefore, by applying Proposition \ref{prop:divisibility3}(1), we obtain the eTNC$_p^-$, as desired.
This completes the proof of Theorem \ref{thm:current_main} under (iii) (and $\neg$(i)). 

\subsection{Other formulations of the eTNC}\label{Ap:eTNC}

In this subsection, we prove that our formulation of the eTNC$^-$ (Conjecture \ref{conj:main}) is equivalent to 
a more standard one formulated by Burns, Kurihara, and Sano \cite[Conjecture 3.1]{BKS16}.
 We first review the statement.

Let $H/F$ be a finite abelian CM-extension and write $G = \Gal (H/F)$. 
Let $\Sigma$ and $\Sigma'$ be finite sets of places of $F$ satisfying the conditions (H1), (H2), (H3), and (H4'), where
\begin{itemize}
\item[(H4')]
$\Sigma \supset S_{\ram}(H/F)$.
\end{itemize}
Clearly (H4') is stronger than (H4).

We introduce a variant of the homomorphism \eqref{eq:defn_Omega} whose cokernel is $\Omega_{\Sigma}^{\Sigma'}$.
Take an auxiliary finite set $S'$ and $v_0 \in S'$ as in \S \ref{ss:Omega_constr}.
For each finite prime $w$ of $H$, 
we consider the map 
\begin{equation}\label{eq:g_w_defn}
g_w : W_w \rightarrow  \Z[G_w]
\end{equation}
which sends $(x,y)$ to $x$ (here we identify $W_w$ with the right hand side of the isomorphism in Proposition \ref{prop:W_str}). 
For each $v \in \Sigma_f$, 
we write $g_v = \oplus_{w \mid v} g_w : W_v \to \bigoplus_{w \mid v}\Z[G_w] \simeq \Z[G] $.  
Note that the last isomorphism depends the choice of a prime of $H$ above $v$.  
Using this, we consider a homomorphism 
\begin{equation}\label{eq:g_w}
g : \bigoplus_{v \in S'_f \setminus \{v_0\}} W_v^-  \to \bigoplus_{v \in S'_f \setminus \{v_0\}} \Z[G]^- 
\end{equation}
which is defined as $g_v^-$ at the components for $\Sigma_f$ and as $\iota_v^-$ at the components 
$v \in S'_f \setminus (\Sigma_f \cup \{ v_0 \} )$, 
where $\iota_v^- = (\oplus_{w \mid v} \iota_w)^-$ is defined in Proposition \ref{prop:str_W_2}. 
We define a homomorphism $\psi_{\Sigma}^{\Sigma'}: \Ker(\theta_V)^- \to \bigoplus_{v \in S'_f \setminus \{v_0\}} \Z[G]^-$
as the composite map 
\begin{equation}\label{eq:psi_defn}
\Ker(\theta_V)^- \to \bigoplus_{v \in S'_f \setminus \{v_0\}} W_v^- \overset{g}{\to} 
 \bigoplus_{v \in S'_f \setminus \{v_0\}} \Z[G]^-
 \end{equation}
where the first map is the middle arrow in the sequence \eqref{eq:snake_theta^-} with the identification as in Lemma \ref{lem:ker_theta_W}. 

We consider a complex of $\Z[G]^-$-modules 
$$C_{\Sigma}^{\Sigma'} = 
\left( \Ker(\theta_V)^- \overset{\psi_{\Sigma}^{\Sigma'}}{\to} \bigoplus_{v \in S'_f \setminus \{v_0\}} \Z[G]^- \right) ,$$
where the first term $\Ker(\theta_V)^-$ is placed in degree zero. 
By the hypotheses (H1),  (H2), (H3) and (H4'), 
we can show that $\Ker(\theta_V)^-$ is a projective $\Z[G]^-$-module 
in a similar was as the proof of Proposition \ref{prop:Omega_p_ct} in \S \ref{ss:ct}.  
Therefore, $C_{\Sigma}^{\Sigma'}$ is a perfect complex of $\Z[G]^-$-modules.

The cohomology groups of $C_{\Sigma}^{\Sigma'}$ are described as follows.
We put
$$ \OO_{H, \Sigma}^{\times, \Sigma'} 
= \{ x \in H^\times \mid \ord_w (x)= 0, \forall w \notin \Sigma_H 
{\;\rm and} \; x \equiv 1(\bmod w), \forall w \in \Sigma'_H\}. $$
Put $X_{H, \Sigma} = \Ker \left( \bigoplus_{w \in \Sigma_H} \Z \to \Z  \right) $. 
Let $\nabla_\Sigma^{\Sigma'} (H)$ be the transpose Selmer module (see \cite[Definition A.2]{DK20}). 
Then by \cite[\S A.1, (148)]{DK20}, we have an exact sequence
\begin{equation}\label{eq:Sel_X}
0 \to \Cl_{H, \Sigma}^{\Sigma'} \to \nabla_\Sigma^{\Sigma'} (H) \to X_{H, \Sigma} \to 0.
\end{equation}

\begin{prop}[{\cite[\S A.1]{DK20}}]\label{prop:perfect complex}
We have an exact sequence 
$$ 0 \to \left( \OO_{H, \Sigma}^{\times, \Sigma'} \right)^- \to  
\Ker(\theta_V)^- \overset{\psi_{\Sigma}^{\Sigma'}}{\to} \bigoplus_{v \in S'_f \setminus \{v_0\}} \Z[G]^-
\to \nabla_\Sigma^{\Sigma'} (H)^- \to 0.$$
In other words, we have isomorphisms 
$H^0 (C_{\Sigma}^{\Sigma'}) \simeq \left( \OO_{H, \Sigma}^{\times, \Sigma'} \right)^-$ and 
$H^1 (C_{\Sigma}^{\Sigma'}) \simeq \nabla_\Sigma^{\Sigma'} (H)^-$.
\end{prop}



Now we introduce relevant $L$-values and Stickelberger elements.
For any $\C$-valued character $\psi$ of $G$, the ``$\Sigma$-depleted, $\Sigma'$-smoothed'' $L$-function is defined by 
$$L_{\Sigma}^{\Sigma'} (\psi, s) = 
L (\psi, s) 
\cdot \prod_{v \in \Sigma_f, v \nmid \ff_{\psi}} \left( 1-\cfrac{\psi(v)}{N(v)^s} \right)
\cdot \prod_{v \in \Sigma'} \left( 1-\psi(v) N(v)^{1-s} \right) .$$
We define the Stickelberger element $\Theta_\Sigma^{\Sigma'} (H/F)$ by 
\[
\Theta_\Sigma^{\Sigma'} (H/F) 
= \sum_{\psi} L_{\Sigma}^{\Sigma'} (\psi^{-1}, 0) e_\psi 
\in \Q[G]^-,
\]
where $\psi$ runs over the odd characters of $G$.
Alternatively, we have
\[
\Theta_\Sigma^{\Sigma'} (H/F)
=  \prod_{v \in \Sigma_f} \left( 1- \cfrac{\nu_{I_v}}{\# I_v}\varphi_v^{-1} \right)^-
\cdot \omega^{\Sigma'}. 
\]
Since we are assuming (H1), (H2), (H3), and (H4'), by the work of Deligne and Ribet \cite{DR80} or Cassou-Nogu\`{e}s \cite{CN79},  
we know $\Theta_\Sigma^{\Sigma'} (H/F) \in \Z[G]^-$.

We put $r_{\psi, \Sigma} = \ord_{s=0} L_{\Sigma}^{\Sigma'} (\psi, s)$. 
Consider the leading term of $L_{\Sigma}^{\Sigma'} (\psi, s)$ at $s=0$ defined by
$$ L_{\Sigma}^{\Sigma'} (\psi, 0)^\ast = \lim_{s \to 0} 
\cfrac{L_{\Sigma}^{\Sigma'} (\psi, s)}{s^{r_{\psi, \Sigma}}} .$$
Then we consider the leading term of the equivariant $L$-function at $s=0$ defined by 
$$\Theta_{\Sigma}^{\Sigma', \ast}(H/F) = 
\sum_{\psi} L_{\Sigma}^{\Sigma'} (\psi^{-1}, 0)^\ast e_\psi \in \C[G]^-,$$  
where $\psi$ runs over the odd characters of $G$.
This element is a non-zero-divisor.

Using these ingredients, we introduce the zeta element.
Let $\det_{\Z[G]^-} (C_{\Sigma}^{\Sigma'})$ denote the determinant module of the complex $C_{\Sigma}^{\Sigma'}$. 
Then we have isomorphisms 

\begin{eqnarray}
\vartheta_{\Sigma, H}^{\Sigma'} : \detZ (C_{\Sigma}^{\Sigma'}) \otimes \C &\simeq&  \detC (C_{\Sigma}^{\Sigma'} \otimes \C) \\
 &\simeq& \detC \left( \left( \OO_{H, \Sigma}^{\times, \Sigma'} \right)^- \otimes \C \right) \otimes \detC^{-1} \left( X_{H, \Sigma}^- \otimes \C \right) \\
&\simeq& \C[G]^-,
\end{eqnarray}
where the second isomorphism comes from Proposition \ref{prop:perfect complex} and \eqref{eq:Sel_X}, and the last isomorphism is induced by the minus component of Dirichlet's regulator isomorphism 
$$ 
 \left( \OO_{H, \Sigma}^{\times, \Sigma'} \right)^- \otimes \C \simeq X_{H, \Sigma}^- \otimes \C $$
which sends $x^- \in \left( \OO_{H, \Sigma}^{\times, \Sigma'} \right)^-$ to $\left( -\sum_{w \in \Sigma_H} \log |x|_w w \right)^-$.
Here, $|-|_w$ denotes the normalized $w$-adic absolute value.

Using this, we define the minus component of the zeta element $z^-_{H/F, \Sigma, \Sigma'}$ by 
$$z^-_{H/F, \Sigma, \Sigma'} 
= \vartheta_{\Sigma, H}^{\Sigma', -1} (\Theta_{\Sigma}^{\Sigma', \ast}(H/F)) \in 
\detZ (C_{\Sigma}^{\Sigma'}) \otimes \C .$$

Then the more standard statement of the minus component of the eTNC is the following.

\begin{conj}[{cf. \cite[Conjecture 3.1]{BKS16}}]\label{conj:classical_eTNC}
We have
$$ \detZ (C_{\Sigma}^{\Sigma'}) = z_{H/F, \Sigma, \Sigma'}^- \Z[G]^- .$$
In other words, the zeta element $z^-_{H/F, \Sigma, \Sigma'}$ is a basis of the determinant module $\detZ (C_{\Sigma}^{\Sigma'})$.
\end{conj}

In the rest of this section, 
we show the following. 

\begin{thm}\label{thm:equivalence2}
Conjecture \ref{conj:classical_eTNC} is equivalent to Conjecture \ref{conj:main}.
\end{thm}

For the proof, it is convenient to make use of another formulation proposed by Kurihara \cite{Kur20}, which we now review.

We write $\Ker (\theta_V)_H, \psi_{\Sigma, H}^{\Sigma'}$ and $W_{v, H}$ for 
$\Ker (\theta_V), \psi_{\Sigma}^{\Sigma'}$ and $W_{v}$ in order to clarify the field concerned.
We also put ${}_p \Ker (\theta_V)_H^- = \Z_p \otimes_{\Z} \Ker (\theta_V)_H^- $, which is a free $\Z_p[G]^-$-module of rank $\# S'_f -1$. 
Then, for each intermediate CM-field $K$ of $H/F$, 
we have 
\[
\psi_{\Sigma, K}^{\Sigma'}: 
{}_p\Ker(\theta_V)_K^- \to \bigoplus_{v \in S'_f \setminus \{v_0\}} \Z_p[\Gal(K/F)]^-.
\]
By \cite[Lemma B.1]{DK20}, 
we have a canonical homomorphism 
\[
\res_{\theta_V, H/K} :  \Ker (\theta_V)_H \to \Ker (\theta_V)_K 
\]
which induces an isomorphism $\left( \Ker (\theta_V)_H \right)_{\Gal(H/K)} \simeq \Ker (\theta_V)_K $. 

Put $r = \# S'_f -1$. 
We fix a basis $\{ b_{i, H} \}_{1 \leq i \leq r}$ of 
$\bigoplus_{v \in S'_f \setminus \{v_0\}} \Z[G]^- $.
For instance, we may take the standard basis.
For each intermediate CM-field $K$ of $H/F$, we write
\[
\pi_{H/K}: \Z[G] \to \Z[\Gal(K/F)]
\]
for the natural restriction map.
By abuse of notation, we also write $\pi_{H/K}$ for the induced homomorphism 
\[
\bigoplus_{v \in S'_f \setminus \{v_0\}} \Z[G]^- \to \bigoplus_{v \in S'_f \setminus \{v_0\}} \Z[\Gal(K/F)]^-.
\]
Then we put $b_{i, K} = \pi_{H/K} (b_{i, H})$ for any $1 \leq i \leq r$, so $\{ b_{i, K} \}_{1 \leq i \leq r}$ is a basis of
$\bigoplus_{v \in S'_f \setminus \{v_0\}} \Z[\Gal(K/F)]^- $.

\begin{conj}[{\cite[Conjecture 3.4]{Kur20}}]\label{conj:Kurihara_conj}
For any odd prime $p$, there is a basis $\{ e_{i, H} \}_{1 \leq i \leq r} $ of ${}_p \Ker (\theta_V)_H^- $ as a $\Z_p[G]^-$-module
satisfying the following property.
For an intermediate CM-field $K$ of $H/F$, 
we put $e_{i,K} = \res_{\theta_V, H/K} (e_{i, H})$ for all $1 \leq i \leq r$. 
Note that $\{ e_{i, K} \}_{1 \leq i \leq r} $ is a basis of ${}_p \Ker (\theta_V)_K^- $ as a $\Z_p[\Gal(K/F)]^-$-module.
Then, for any intermediate CM-field $K$ of $H/F$ and any subset  
$\ol{\Sigma}$ of $\Sigma$ satisfying
$\ol{\Sigma} \supset S_\infty(F) \cup S_{\ram}(K/F)$, 
we have 
$$ \det (\psi_{\ol{ \Sigma},K}^{\Sigma'}) = \Theta_{\ol{\Sigma}}^{\Sigma'} (K/F)  ,$$  
where $\det (\psi_{\ol{ \Sigma},K}^{\Sigma'})$ is the determinant of $\psi_{\ol{ \Sigma},K}^{\Sigma'}$ with respect to the bases
$\{ e_{i, K} \}_{1 \leq i \leq r} $ of ${}_p \Ker (\theta_V)_K^- $ and 
$\{ b_{i, K} \}_{1 \leq i \leq r} $ of 
$\bigoplus_{v \in S'_f \setminus \{v_0\}} \Z_p[\Gal(K/F)]^- $. 
\end{conj}

We have the following.

\begin{prop}[{\cite[Proposition 3.5]{Kur20}}]\label{prop:equiv3}
Conjecture \ref{conj:Kurihara_conj} is equivalent to Conjecture \ref{conj:classical_eTNC}. 
\end{prop}

Thanks to this proposition, in order to prove Theorem \ref{thm:equivalence2}, it is enough to show the equivalence between Conjectures \ref{conj:main} and \ref{conj:Kurihara_conj}.
%


Recall that in Definition \ref{defn:Omega}, $\Omega_{\Sigma}^{\Sigma'}$ is defined as the cokernel of the injective homomorphism \eqref{eq:defn_Omega}.
We write
\[
\phi_{\Sigma}^{\Sigma'}: \Ker(\theta_V)^- \to \bigoplus_{v \in S'_f \setminus \{v_0\}} \Z[G]^-
\]
for the homomorphism \eqref{eq:defn_Omega}.

We show a lemma that is necessary for proving the equivalence.

\begin{lem}\label{lem:two_diagram2}
The following are true.
\begin{itemize}
\item[(1)] 
The following is commutative.
\[
\xymatrix{
      \Ker (\theta_V)^- \otimes \Q \ar[r]^-{\phi_{\Sigma}^{\Sigma'}} \ar[dr]_-{\psi_{\Sigma}^{\Sigma'}}
     &  \bigoplus_{v \in S'_f \setminus \{v_0\}} \Q[G]^- \ar[d] \\
	& \bigoplus_{v \in S'_f \setminus \{v_0 \} } \Q[G]^-. 
}
\]
Here, the vertical arrow is defined as the identity at the components for 
$v \notin \Sigma_f$ and as $\times \left(1 - \frac{\nu_{I_v}}{\# I_v}\varphi_v^{-1} \right) h_v^{-1} $ at the components 
$v \in \Sigma_f$ (recall that $h_v$ is defined in \eqref{eq:hv}).
\item[(2)]
Let $\ol{\Sigma}$ be a subset of $\Sigma$ such that 
$\ol{\Sigma} \supset S_\infty(F) \cup S_{\ram}(H/F)$. 
Then the following is commutative.
\[
\xymatrix{
      \Ker (\theta_V)^- \otimes \Q \ar[r]^-{\phi_{\Sigma}^{\Sigma'}} \ar[dr]_-{\phi_{\ol{\Sigma}}^{\Sigma'}}
     &  \bigoplus_{v \in S'_f \setminus \{v_0\}} \Q[G]^- \ar[d] \\
	& \bigoplus_{v \in S'_f \setminus \{v_0 \} } \Q[G]^-. 
}
\]
Here, the vertical arrow is defined as the identity at the components for 
$v \notin \Sigma \setminus \ol{\Sigma}$ and as $\times h_v^{-1} $ at the components 
$v \in \Sigma \setminus \ol{\Sigma}$. 
\item[(3)]
In order to clarify the field, we write $\phi_{\Sigma, H}^{\Sigma'}$ and $h_{v, H}$ for $\phi_{\Sigma}^{\Sigma'}$ and $h_v$.
Then, for an intermediate CM-field $K$ of $H/F$, 
the following is commutative.
\[ 
\xymatrix{
      \Ker (\theta_V)_H^- \otimes \Q \ar[r]^{\phi_{\Sigma, H}^{\Sigma'}} \ar[d]_{\res_{\theta_V, H/K}}
     &  \bigoplus_{v \in S'_f \setminus \{v_0\}} \Q[G]^- \ar[d] \\
	 \Ker (\theta_V)_K^- \otimes \Q \ar[r]^-{\phi_{\Sigma, K}^{\Sigma'}} 
	& \bigoplus_{v \in S'_f \setminus \{v_0 \} } \Q[\Gal (K/F)]^-. 
}
\]
Here, the right vertical arrow is defined as the natural restriction map $\pi_{H/K}$ at the components for 
$v \notin \Sigma_f $ and as the minus component of the composite map 
$$ \Q[G] \overset{\pi_{H/K}}{\to} \Q[\Gal(K/F)] 
\overset{\times h_{v, K} \cdot \pi_{H/K}(h_{v, H})^{-1} }{\longrightarrow} \Q[\Gal(K/F)] $$
at the components $v \in \Sigma_f $.
\end{itemize}
\end{lem}

\begin{proof}
(1)(2) 
Let $v$ be any finite prime of $F$. 
For any finite prime $w$ of $H$ above $v$, by Proposition \ref{prop:W_str}, 
$W_w \otimes \Q$ is a free $\Q[G_w]$-module of rank $1$ and 
$(1 - \frac{\nu_{I_v}}{\# I_v}\varphi_v^{-1} , 1)$ is a basis of $W_w \otimes \Q$. 
By the definitions of $\iota_w$, $f_w$, $g_w$ (see Proposition \ref{prop:str_W_2}(1), (2), and \eqref{eq:g_w_defn}, respectively), we have
\begin{align}
\iota_w \left( (1 - \frac{\nu_{I_v}}{\# I_v}\varphi_v^{-1} , 1) \right)  
& = 1,\\
f_w \left( (1 - \frac{\nu_{I_v}}{\# I_v}\varphi_v^{-1} , 1) \right)  
& = h_v,  \\
g_w \left( (1 - \frac{\nu_{I_v}}{\# I_v}\varphi_v^{-1} , 1) \right)
& = 1 - \frac{\nu_{I_v}}{\# I_v}\varphi_v^{-1}.
\end{align}
Therefore, the following are commutative.
\[
\xymatrix{
     W_w \otimes \Q \ar[r]^-{f_w} \ar[rd]_{g_w}
     &   \Q[G_w] \ar[d]^{\times \left(1 - \frac{\nu_{I_v}}{\# I_v}\varphi_v^{-1} \right) h_v^{-1}} \\
	&  \Q[G_w],
}
\qquad
\xymatrix{
     W_w \otimes \Q \ar[r]^-{f_w} \ar[rd]_{\iota_w} 
     &   \Q[G_w] \ar[d]^{\times h_v^{-1}} \\
	&  \Q[G_w]. 
}
\]
These diagrams imply the claims (1) and (2), respectively. 




(3)
By a similar proof as \cite[Lemma B.2]{DK20}, 
we obtain a commutative diagram 
\[ 
\xymatrix{
      \Ker (\theta_V)_H^- \ar[r] \ar[d]_{\res_{\theta_V, H/K}}
     &  \bigoplus_{v \in S'_f \setminus \{v_0\}} W_{v, H}^- \ar[d] \\
	 \Ker (\theta_V)_K^- \ar[r] 
	& \bigoplus_{v \in S'_f \setminus \{v_0 \} }  W_{v, K}^-. 
}
\]
Here, the two horizontal arrows are the first map in \eqref{eq:psi_defn} and 
the right vertical arrow is the natural restriction map. 
For any finite prime $v$ of $F$, 
the natural map $W_{v, H} \otimes \Q \to W_{v, K} \otimes \Q$ sends the basis $(1 - \frac{\nu_{I_{v, H}}}{\# I_{v, H}}\varphi_v^{-1} , 1)$ of $W_{v, H}^- \otimes \Q$ to the basis $(1 - \frac{\nu_{I_{v, K}}}{\# I_{v, K}}\varphi_v^{-1} , 1)$ of $W_{v, K}^- \otimes \Q$.
Therefore, the above formula on $f_w$ implies that the following is commutative.
\[
\xymatrix{
     W_{v, H} \otimes \Q \ar[r]^-{f_{v, H}} \ar[d]
     &   \Q[G] \ar[d]^{(\times h_{v, K} \cdot \pi_{H/K}(h_{v, H})^{-1}) \circ \pi_{H/K}}  \\
	 W_{v,K} \otimes \Q \ar[r]^-{f_{v, K}} 
	&  \Q[\Gal(K/F)],  
}
\]
where $f_{v, H}$ and $f_{v, K}$ are defined as in Definition \ref{defn:f_hom}. 
From these diagrams, the claim (3) follows.  
\end{proof}

Now we are ready to prove Theorem \ref{thm:equivalence2}. 

\begin{proof}[Proof of Theorem \ref{thm:equivalence2}] 
By Proposition \ref{prop:equiv3}, it is enough to show the equivalence between Conjectures \ref{conj:main} and \ref{conj:Kurihara_conj}.

Firstly, we prove that Conjecture \ref{conj:main} implies Conjecture \ref{conj:Kurihara_conj}.  
Assume that Conjecture \ref{conj:main} holds. 
We fix an odd prime number $p$.
Since $\Omega_{\Sigma}^{\Sigma'}$ is the cokernel of $\phi_{\Sigma}^{\Sigma'} = \phi_{\Sigma, H}^{\Sigma'}$, by the validity of Conjecture \ref{conj:main}, 
there exists a basis $\{ e_{i, H} \}_{1 \leq i \leq r} $ of ${}_p \Ker (\theta_V)_{H}^- $ such that 
\[
\det (\phi_{\Sigma, H}^{\Sigma'}) = \theta_{\Sigma}^{\Sigma'} = \theta_{\Sigma, H}^{\Sigma'}
\] 
with respect to the basis
$\{ e_{i, H} \}_{1 \leq i \leq r} $ and 
the fixed basis $\{ b_{i, H} \}_{1 \leq i \leq r} $ of 
$\bigoplus_{v \in S'_f \setminus \{v_0\}} \Z_p[G]^- $. 
For an intermediate CM-field $K$ of $H/F$, 
we define $e_{i,K} = \res_{\theta_V, H/K} (e_{i, H}) $ for all $1 \leq i \leq r$. 
For any subset  
$\ol{\Sigma}$ of $\Sigma$ such that 
$\ol{\Sigma} \supset S_\infty(F) \cup S_{\ram}(K/F)$, by Lemma \ref{lem:two_diagram2}(1), (2), and (3), respectively, 
we obtain the first, second, and third equalities of
\begin{align}
\det (\psi_{\ol{ \Sigma},K}^{\Sigma'}) 
&= \prod_{v \in \ol{ \Sigma}_f } \left[\left( 1-\cfrac{\nu_{I_{v, K}}}{\# I_{v,K}}\varphi_{v, K}^{-1} \right)^- 
	\cdot (h_{v,K}^-)^{-1} \right] \cdot \det (\phi_{\ol{ \Sigma},K}^{\Sigma'}) \\
&= \prod_{v \in \ol{ \Sigma}_f } \left[ \left( 1-\cfrac{\nu_{I_{v, K}}}{\# I_{v,K}}\varphi_{v, K}^{-1} \right)^- 
	\cdot (h_{v,K}^-)^{-1} \right] \cdot \prod_{v \in \Sigma \setminus \ol{ \Sigma}} (h_{v,K}^-)^{-1} \cdot \det (\phi_{\Sigma,K}^{\Sigma'}) \\
&= \prod_{v \in \ol{ \Sigma}_f } \left( 1-\cfrac{\nu_{I_{v, K}}}{\# I_{v,K}}\varphi_{v, K}^{-1} \right)^- \cdot \prod_{v \in \Sigma_f} (h_{v,K}^-)^{-1} \cdot \prod_{v \in \Sigma_f} h_{v,K}^{-} \cdot \pi_{H/K} \left( \prod_{v \in \Sigma_f} (h_{v,H}^-)^{-1} \cdot \det (\phi_{\Sigma, H}^{\Sigma'}) \right)  \\
&= \prod_{v \in \ol{ \Sigma}_f } \left( 1-\cfrac{\nu_{I_{v, K}}}{\# I_{v,K}}\varphi_{v, K}^{-1} \right)^- \cdot \pi_{H/K} (\omega^{\Sigma'}) \\
&= \Theta_{\ol{\Sigma}}^{\Sigma'} (K/F).
\end{align}
Therefore, Conjecture \ref{conj:Kurihara_conj} holds under the validity of 
Conjecture \ref{conj:main}. 

Next, we show that Conjecture \ref{conj:Kurihara_conj} implies Conjecture \ref{conj:main}. 
Assume that Conjecture \ref{conj:Kurihara_conj} holds. 
We fix an odd prime number $p$ and aim at showing the eTNC$_p^-$.
We take a basis 
$\{ e_{i, H} \}_{1 \leq i \leq r} $ of ${}_p \Ker (\theta_V)_H^- $ as a $\Z_p[G]^-$-module
satisfying the condition in Conjecture \ref{conj:Kurihara_conj}.
Then it is enough to show that we have $\det(\phi_{\Sigma, H}^{\Sigma'}) = \theta_{\Sigma}^{\Sigma'}$ as elements of $\Q_p[G]^-$, where the determinant is defined with respect to the bases $\{ e_{i, H} \}_{1 \leq i \leq r} $ and $\{ b_{i, H} \}_{1 \leq i \leq r} $.

For any odd character $\psi$ of $G$, 
we write $H^\psi = H^{\Ker \psi}$ and $\Sigma_\psi = S_{\ram} (H^\psi / F) $. 
Also for any $x \in \Q_p[G]^-$ or $x \in \Q_p[\Gal (H^\psi / F)]$, 
we write $x^\psi \in \Q_p(\Imag \psi)$ for the $\psi$-component of $x$. 
Then, for any odd character $\psi$ of $G$, 
Conjecture \ref{conj:Kurihara_conj}, Lemma \ref{lem:two_diagram2} (1), (2), and (3) respectively imply the fourth, third, second, and first equality of 
\begin{align}
\det (\phi_{\Sigma,H}^{\Sigma'})^\psi 
&= \prod_{v \in \Sigma_f} h_{v,H}^{\psi} \cdot \left( \prod_{v \in \Sigma_f} (h_{v,H^\psi}^{\psi})^{-1} \cdot \det (\phi_{\Sigma, H^\psi}^{\Sigma'})^\psi \right) \\
&= \prod_{v \in \Sigma_f} h_{v,H}^{\psi} \cdot \left( \prod_{v \in \Sigma_{\psi, f}} (h_{v,H^\psi}^{\psi})^{-1} \cdot \det (\phi_{\Sigma_\psi, H^\psi}^{\Sigma'})^\psi \right) \\
&= \prod_{v \in \Sigma_f} h_{v,H}^{\psi} \cdot \det (\psi_{\Sigma_\psi, H^\psi}^{\Sigma'})^\psi \\
&= \prod_{v \in \Sigma_f} h_{v,H}^{\psi} \cdot \Theta_{\Sigma_\psi}^{\Sigma'} (H^\psi/F)^\psi \\
&= \prod_{v \in \Sigma_f} h_{v,H}^{\psi} \cdot L_{\Sigma_\psi}^{\Sigma'}(\psi^{-1}, 0) = \theta_{\Sigma}^{\Sigma', \psi}, 
\end{align}
where $\det (\phi_{\Sigma,H}^{\Sigma'})$ (resp. $\det (\phi_{\Sigma, H^\psi}^{\Sigma'}), \det (\phi_{\Sigma_\psi, H^\psi}^{\Sigma'})$ and $\det (\psi_{\Sigma_\psi, H^\psi}^{\Sigma'})$) is defined with respect to the bases
$\{ e_{i, H} \}_{1 \leq i \leq r} $ 
and 
$\{ b_{i, H} \}_{1 \leq i \leq r} $ 
(resp. $\{ e_{i, H^\psi} \}_{1 \leq i \leq r} $
 and 
$\{ b_{i, H^\psi} \}_{1 \leq i \leq r} $).
Therefore, we have $\det(\phi_{\Sigma, H}^{\Sigma'}) = \theta_{\Sigma}^{\Sigma'}$ as desired.
This completes the proof of Theorem \ref{thm:equivalence2}. 
\end{proof}

\section{Modifications of the Eisenstein series}\label{s:Eis}

The rest of this paper is devoted to the proof of Theorem \ref{thm:current_main_p}.
In this section, we introduce group ring valued Eisenstein series that are suitable for our purpose.

In \S \ref{ss:HMF}, we briefly introduce notation on Hilbert modular forms.
In \S \ref{ss:Eisen}, we introduce the most basic Eisenstein series, and then in \S \ref{ss:Eisen_modif1} and \S \ref{ss:Eisen_modif2}, we modify them appropriately for our purpose, partly following an idea of Dasgupta and Kakde \cite{DK20}.
Finally in \S \ref{ss:Eisen_Hecke}, we compute the Hecke actions on the modified Eisenstein series.

In this section, let $F$ be a fixed totally real field.
We put $d = [F:\Q]$.
%
%
In general, for a property $P$ (e.g. the property $k = 1$ for a given integer $k$), we define $\delta_P = 1$ if $P$ is true and $\delta_P = 0$ otherwise.
This notation is introduced to avoid lengthy case-by-case argument.

\subsection{Notation on Hilbert modular forms}\label{ss:HMF}

We briefly introduce Hilbert modular forms.
In this paper we completely follow the notation of \cite[\S 7.2]{DK20}, so the readers are advised to refer to it for more details.

For each element $\lambda \in \Cl^+(F)$ of the narrow class group of $F$,
we implicitly fix a fractional ideal $\ft_{\lambda}$ that represents $\lambda$.
Let $\fn$ be an ideal of $F$ (by this we implicitly mean that $\fn$ is a nonzero integral ideal) and $k \geq 1$ an integer.

We write $M_k(\fn)$ for the space of $\C$-valued Hilbert modular forms of (parallel) weight $k$ and of level $\fn$.
For each $f \in M_k(\fn)$, we have the normalized Fourier coefficients $\coef{\fa}{f} \in \C$ for ideals $\fa$ of $F$, and 
$\coefz{\lambda}{f} \in \C$ for $\lambda \in \Cl^+(F)$.
These coefficients characterize the form $f$, so we have an injective $\C$-linear homomorphism
\[
M_k(\fn) \hookrightarrow \prod_{\lambda} \C \times \prod_{\fa} \C
\]
that sends $f$ to $((\coefz{\lambda}{f})_{\lambda}, (\coef{\fa}{f})_{\fa})$.
This map is called the Fourier expansion.

We have the set of cusps of level $\fn$, denoted by $\cusps(\fn)$.
An element of $\cusps(\fn)$ is written as the class $[\cA]$ of a pair
\[
\cA = (A, \lambda),
\]
 where $\lambda \in \Cl^+(F)$ and $A \in \GL_2^+(F)$ (the subgroup of $\GL_2(F)$ whose elements have totally positive determinants).
For such a pair $\cA$ and for $f \in M_k(\fn)$, we write $c_{\cA}(0, f)$ for the constant term of the normalized Fourier coefficients.

A cuspform (of the same weight and level) is defined as a form $f \in M_k(\fn)$ such that $c_{\cA}(0, f) = 0$ for any pair $\cA$.
The subspace of cuspforms is denoted by $S_k(\fn) \subset M_k(\fn)$.

Given a pair $\cA = (A, \lambda)$, following \cite[\S 7.2.5]{DK20}, we define a fractional ideal $\fb_{\cA}$ of $\OO_F$ by
\[
\fb_{\cA} = (a) + (c) (\ft_{\lambda} \fd_F)^{-1}.
\]
Here, $A = \begin{pmatrix} a & \ast \\ c & \ast \end{pmatrix}$ and $\fd_F$ denotes the different of $F$.
Then we define an integral ideal $\fc_{\cA}$ of $\OO_F$ by
\[
\fc_{\cA} = (c) (\ft_{\lambda} \fd_F \fb_{\cA})^{-1}.
\]
This ideal $\fc_{\cA}$ depends only on the cusp $[\cA]$.
For an ideal $\fb \mid \fn$, we define
\[
C_{\infty}(\fb, \fn) = \{ [\cA] \in \cusps(\fn) : \fb \mid \fc_{\cA} \} 
\]
and
\[
C_{0}(\fb, \fn) = \{ [\cA] \in \cusps(\fn) : \gcd(\fb, \fc_{\cA}) = 1 \}. 
\]
We also write $C_{\infty}(\fn) = C_{\infty}(\fn, \fn)$.

We have the diamond operator $S(\fa)$ on $M_k(\fn)$ and $S_k(\fn)$ for each element $\fa \in \Cl^+_{\fn}(F)$ of the narrow ray class group of $F$ of conductor $\fn$.
For each $\C$-valued character $\psi$ of $\Cl^+_{\fn}(F)$ which is totally odd (resp.~totally even) if $k$ is odd (resp.~$k$ is even), we put
\[
M_k(\fn, \psi) = \{ f \in M_k(\fn) \mid S(\fa) f = \psi(\fa) f, \forall \fa \in \Cl^+_{\fn}(F) \}
\]
and $S_k(\fn, \psi) = M_k(\fn, \psi) \cap S_k(\fn)$.
Then we have the direct sum decompositions
\[
M_k(\fn) = \bigoplus_{\psi} M_k(\fn, \psi),
\qquad S_k(\fn) = \bigoplus_{\psi} S_k(\fn, \psi).
\]

We write $M_k(\fn, \Z) \subset M_k(\fn)$ for the $\Z$-submodule of forms $f$ such that $c(\fa, f) \in \Z$ and $c_{\lambda}(0, f) \in \Z$ holds for all $\fa$ and $\lambda$.
Then, for a ring $R$, we define the $R$-module of Hilbert modular forms with coefficients in $R$ by $M_k(\fn, R) = M_k(\fn, \Z) \otimes_{\Z} R$.
When a character $\psi: \Cl_{\fn}^+(F) \to R^{\times}$ is given, we define
\[
M_k(\fn, R, \psi) = \{ f \in M_k(\fn, R) \mid S(\fa) f = \psi(\fa) f, \forall \fa \in \Cl^+_{\fn}(F) \}.
\]
We define $S_k(\fn, R)$ and $S_k(\fn, R, \psi)$ in a similar way.

In practice, we are interested in a finite abelian CM-extension $H/F$ (with $G = \Gal(H/F)$) and an odd prime number $p$.
Let us suppose that $k$ is odd and the conductor $\ff_{H/F}$ of $H/F$ divides $\fn$.
Let 
\[
\bpsi: \Cl_{\fn}^+(F) \twoheadrightarrow G \hookrightarrow \Z_p[G]^{\times} \twoheadrightarrow (\Z_p[G]^-)^{\times}
\]
 be the tautological character.
Then the above construction gives rise to spaces 
\[
S_k(\fn, \Z_p[G]^-, \bpsi) \subset M_k(\fn, \Z_p[G]^-, \bpsi).
\]
In this case, we have an interpretation of the elements of $M_k(\fn, \Z_p[G]^-, \bpsi)$ as families of forms, as follows (see \cite[Lemma 7.2]{DK20}, which comes from a result of Silliman \cite[Corollary 7.28]{Sil20}).
The Fourier expansion induces an injective $\Z_p[G]^-$-homomorphism
\[
M_k(\fn, \Z_p[G]^-, \bpsi)
\hookrightarrow \prod_{\lambda} \Z_p[G]^- \times \prod_{\fa} \Z_p[G]^-.
\]
Then an element $((c_{\lambda}(0))_{\lambda}, (c(\fa))_{\fa})$ of the target module is in the image of this map if and only if there exists a family $(f_{\psi})_{\psi} \in \prod_{\psi} M_k(\fn, \OO_{\psi}, \psi)$ ($\psi$ runs over the odd characters of $G$) such that $\psi(c_{\lambda}(0)) = c_{\lambda}(0, f_{\psi})$ and $\psi(c(\fa)) = c(\fa, f_{\psi})$ hold for any $\lambda$, $\fa$, and $\psi$.

\subsection{The Eisenstein series}\label{ss:Eisen}

In this subsection, we introduce the classical Eisenstein series.
See \cite[\S 7.3]{DK20}, \cite{DK20b}, or \cite[\S 2.2]{DDP11} for more details.

For each character $\psi$ of $F$, we write $\ff_{\psi}$ for its conductor.
When $\fn$ is an ideal divisible by $\ff_{\psi}$, let $\psi_{\fn}$ denote the character $\psi$ whose modulus is enlarged to $\fn$.
This means that $\psi_{\fn}(\fl) = 0$ as long as $\fl \mid \fn$.
We define the associated $L$-function by
\begin{align}
L(\psi_{\fn}, s) 
& = \prod_{v \nmid \fn} \parenth{1 - \frac{\psi(v)}{N(v)^s}}^{-1}\\
& = \prod_{v \mid \fn, v \nmid \ff_{\psi}} \parenth{1 - \frac{\psi(v)}{N(v)^s}} L(\psi, s).
\end{align}
Recall that we put $d = [F: \Q]$.

\begin{defn}\label{defn:Eisen}
Let $k \geq 1$ be an odd integer.
Let $\psi$ be a totally odd character of $G_F$.
Let $\fn$ be an ideal of $F$ which is divisible by $\ff_{\psi}$.
Let $\fR$ be an ideal of $F$ such that $(\fn, \fR) = 1$.
Let $1_{\fR}$ denote the trivial character whose modulus is enlarged to $\fR$.
Note that in practice $\fR$ is a product of $p$-adic primes (for a fixed $p$), and moreover we take $\fR = 1$ in most cases.

Then there exists a modular form $E_k(\psi_{\fn}, 1_{\fR}) \in M_k(\fn \fR, \Q(\psi), \psi)$, called the Eisenstein series, whose Fourier coefficients are given by 
\[
\coef{\fa}{E_k(\psi_{\fn}, 1_{\fR})}
= \sum_{\fr \mid \fa, (\fr, \fR) = 1} \psi_{\fn} \parenth{\frac{\fa}{\fr}} N(\fr)^{k - 1},
\]
where $\fr$ runs over the divisors of $\fa$ with $(\fr, \fR) = 1$, for each $\fa$, and
\[
 \coefz{\lambda}{E_k(\psi_{\fn}, 1_{\fR})}
= \delta_{\fn = 1} 2^{-d} \psi^{-1}(\lambda) L(\psi_{\fR}^{-1}, 1 - k)
+ \delta_{k = 1} \delta_{\fR = 1} 2^{-d} L(\psi_{\fn}, 0)
\]
for each $\lambda$.
(Recall that $\delta_{\fn = 1}$ is $1$ if $\fn = 1$ and is $0$ otherwise, and the other $\delta_P$'s are defined similarly.
Therefore, we have to consider the first term of the right hand side only when $\fn = 1$ and in particular $\ff_{\psi} = 1$, so $\psi^{-1}(\lambda)$ and $\psi_{\fR}$ are well-defined.
This kind of remark will be implicit in the remaining manuscript.)
\end{defn}

We next incorporate the Eisenstein series to a family.
Let $H/F$ be a finite abelian CM-extension.
For an odd integer $k \geq 1$ and an ideal $\fn$ of $F$ divisible by the conductor $\condu{H/F}$ of $H/F$, we define
\begin{equation}\label{eq:Theta}
\Theta_{\fn}(1 - k) 
= \sum_{\psi} L(\psi_{\fn}^{-1}, 1 - k) e_{\psi}
\in \Q[\Gal(H/F)]^-,
\end{equation}
where $\psi$ runs over the odd characters of $\Gal(H/F)$.

Let
\[
\bpsi = \bpsi^{H/F}: \Gal(H/F) \to (\Z[\Gal(H/F)]^-)^{\times}
\]
denote the tautological character which sends $\sigma$ to $\frac{1 - c}{2} \sigma$ (as usual $c$ denotes the complex conjugation).
For a prime $\fl$ which is unramified in $H/F$, we put $\bpsi(\fl) = \bpsi(\varphi_{\fl})$.
For an ideal $\fn$ of $F$ divisible by $\condu{H/F}$, we also define $\bpsi_{\fn}$ in a similar manner.

\begin{defn}\label{defn:Eisen_group}
Let $H/F$ be a finite abelian CM-extension.
Let $k \geq 1$ be an odd integer.
Let $\fn$ be an ideal of $F$ divisible by $\condu{H/F}$.
Let $\fR$ be an ideal of $F$ such that $(\fn, \fR) = 1$.

Then we define the group ring valued Eisenstein series
\[
E_k^{H/F}(\bpsi_{\fn}, \btr_{\fR}) 
\in M_k(\fn \fR, \Q[\Gal(H/F)]^-, \bpsi)
\]
by requiring that the specialization of $E_k^{H/F}(\bpsi_{\fn}, \btr_{\fR})$ at each odd character $\psi$ of $\Gal(H/F)$ equals to $E_k(\psi_{\fn}, 1_{\fR})$.
Explicitly, the form is constructed as
\[
\coef{\fa}{E_k^{H/F}(\bpsi_{\fn}, \btr_{\fR})}
= \sum_{\fr \mid \fa, (\fr, \fR) = 1} \bpsi_{\fn} \parenth{\frac{\fa}{\fr}} N(\fr)^{k - 1}
\]
for each $\fa$ and
\[
\coefz{\lambda}{E_k^{H/F}(\bpsi_{\fn}, \btr_{\fR})}
= \delta_{\fn = 1} 2^{-d} \bpsi^{-1}(\lambda) \Theta_{\fR}(1 - k)
+ \delta_{k = 1} \delta_{\fR = 1} 2^{-d} \Theta_{\fn}(0)^{\sharp}
\]
for each $\lambda$.
Here, $\sharp$ denotes the involution on group rings that inverts every group element.
\end{defn}

\subsection{A modification}\label{ss:Eisen_modif1}

In this subsection, we make the first modification of the above group ring valued Eisenstein series.
We essentially follow \cite{DK20}, but we change the exposition to some extent so that we will be able to make the second modification in the next subsection.

From now on let us fix an odd prime number $p$.
We shall consider a pair $(H/F, \fn)$ as follows.

\begin{setting}\label{setting:H/F}
Let $(H/F, \fn)$ be a pair satisfying:
\begin{itemize}
\item
$H/F$ is a finite abelian CM-extension such that $\Gal(H/F)'$, the maximal subgroup of $\Gal(H/F)$ whose order is prime to $p$, is cyclic.
\item
$\fn$ is an ideal of $F$ divisible by $\condu{H/F}$.
\end{itemize}
We suppose:
\begin{itemize}
\item[$(\star)$] 
When we write $\fn = \condu{H/F} \fQ \fL$ with $\fQ = \gcd( \condu{H/F}^{-1} \fn, p^{\infty})$, 
we have $(\fL, \ff_{H/F}) = 1$ and $\fL$ is square-free.
\end{itemize}
\end{setting}

Given such a pair $(H/F, \fn)$ and a faithful odd character $\chi$ of $\Gal(H/F)'$, we will work over the local ring $\Z_p[\Gal(H/F)]^{\chi}$.

\begin{defn}\label{defn:H^m}
Let $(H/F, \fn)$ be as in Setting \ref{setting:H/F}.
For an ideal $\fm \mid \fn$,
we write $H^{\fm}$ for the maximal intermediate field of $H/F$ such that every prime $v \mid \fm$ is unramified in $H^{\fm}/F$.
In other words, we have 
\[
\Gal(H/H^{\fm}) = \sum_{v \mid \fm} I_{v, H}
\]
as subgroups of $\Gal(H/F)$, where $I_{v, H}$ denotes the inertia group of $v$ in $H/F$.
We also put
\[
\nu_{\fm} = \prod_{v \mid \fm} \nu_{I_{v, H}} \in \Z[\Gal(H/F)].
\]
\end{defn}

Then $\nu_{\fm}$ is a multiple of the norm element of $\Gal(H/H^{\fm})$, 
so the multiplication by $\nu_{\fm}$ induces a well-defined homomorphism 
\[
\nu_{\fm}: \Z_p[\Gal(H^{\fm}/F)] \to \Z_p[\Gal(H/F)].
\]
We will modify the Eisenstein series by using lifts of Eisenstein series for $H^{\fm}/F$ for various $\fm \mid \fn$.
If there exists a prime $v \mid \fm$ such that $I_{v, H}$ is not a $p$-group, then we have $\nu_{\fm}^{\chi} = 0$ for each odd faithful character $\chi$ of $\Gal(H/F)'$.

In general, for an ideal $\fa$ of $F$, a Hall divisor of $\fa$ is defined as a divisor $\fb$ of $\fa$ satisfying $\gcd(\fb, \fa / \fb) = 1$.
In this case, we write $\fb \| \fa$.

\begin{defn}\label{defn:n_p}
Let $(H/F, \fn)$ be as in Setting \ref{setting:H/F}. 
We write $\fn_p$ for the Hall divisor of $\fn$ such that $v \mid \fn_p$ if and only if $I_{v, H}$ is a $p$-group.
\end{defn}

Let $\chi$ be an odd faithful character of $\Gal(H/F)'$.
Then we have $\nu_{\fm}^{\chi} = 0$ unless $\fm \mid \fn_p$.
If $\fm \mid \fn_p$ holds, the natural restriction map induces an isomorphism $\Gal(H/F)' \simeq \Gal(H^{\fm}/F)'$, so the character $\chi$ can be regarded as a character of $\Gal(H^{\fm}/F)'$, and we have
\begin{equation}\label{eq:norm_lift_m}
\nu_{\fm}: \Z_p[\Gal(H^{\fm}/F)]^{\chi} \to \Z_p[\Gal(H/F)]^{\chi}.
\end{equation}

The following lemmas are essentially observed in \cite[Lemma 8.13]{DK20}.

\begin{lem}\label{lem:np_sq-free}
Let $(H/F, \fn)$ be as in Setting \ref{setting:H/F}. 
Then the ideal $\fn_p / \gcd(\fn_p, p^{\infty})$ is square-free.
\end{lem}

\begin{proof}
Let us assume that there exists a prime $\fl \nmid p$ such that $\fl^2 \mid \fn_p$.
If $\fl$ is unramified in $H/F$, then $(\star)$ implies $\fl^2 \nmid \fn$, so we get a contradiction.
If $\fl$ is ramified in $H/F$, then $(\star)$ implies 
\[
2 \leq \ord_{\fl}(\fn) = \ord_{\fl}(\condu{H/F}),
\]
where $\ord_{\fl}$ denotes the normalized additive $\fl$-order.
On the other hand, $\fl \mid \fn_p$ says that $I_{\fl, H}$ is a $p$-group, so we have $\ord_{\fl}(\condu{H/F}) \leq 1$.
This is a contradiction.
\end{proof}

\begin{lem}\label{lem:8.13n}
Let $(H/F, \fn)$ be as in Setting \ref{setting:H/F}. 
The following are true.
\begin{itemize}
\item[(1)]
Let $H'$ be an intermediate field of $H/F$ such that $H/H'$ is a $p$-extension.
Then $(H'/F, \fn)$ also satisfies the condition of Setting \ref{setting:H/F}.
\item[(2)]
For each $\fm \mid \fn_p$, the pair $(H^{\fm}/F, \fn/\fm)$ also satisfies the condition of Setting \ref{setting:H/F}.
\item[(3)]
Let $\psi$ be an odd character of $\Gal(H/F)$ that is faithful on $\Gal(H/F)'$.
We write 
\[
\fn = \ff_{\psi} \fQ_{\psi} \fL_{\psi}
\]
 with $\ff_{\psi}$ the conductor of $\psi$ and $\fQ_{\psi} = \gcd( \ff_{\psi}^{-1} \fn, p^{\infty})$. 
Then we have $(\ff_{\psi}, \fL_{\psi}) = 1$ and $\fL_{\psi}$ is square-free.
\end{itemize}
\end{lem}

\begin{proof}
(1)
We have only to check $(\star)$ for $(H'/F, \fn)$.
Let us write 
\begin{equation}\label{eq:n/m}
\fn = \condu{H'/F} \fQ' \fL'
\end{equation}
 with 
$\fQ' = \gcd(\condu{H'/F}^{-1}\fn, p^{\infty})$.
We have to show that $\fL'$ is square-free and $(\fL', \ff_{H'/F}) = 1$.
Let us assume that there exists a prime $\fl \mid \fL'$ such that either $\fl^2 \mid \fL'$ or $\fl$ is ramified in $H'/F$.
In the both cases, we have $\fl^2 \mid \fn$.
Then the condition $(\star)$ for $(H/F, \fn)$ implies $\fl^2 \mid \condu{H/F}$ and $\ord_{\fl}(\condu{H/F}) = \ord_{\fl}(\fn)$.
Since $H/H'$ is a $p$-extension, $\fl^2 \mid \condu{H/F}$ implies
\[
\ord_{\fl}(\condu{H/F}) = \ord_{\fl}(\condu{H'/F}).
\]
These formulas combined with $\fl \mid \fL'$ contradict the $\fl$-order of the equation \eqref{eq:n/m}.

(2)
By the claim (1), $(H^{\fm}/F, \fn)$ satisfies the condition of Setting \ref{setting:H/F}, so $(H^{\fm}/F, \fn/\fm)$ also satisfies it.

(3)
Let us consider the intermediate field $H^{\psi} = H^{\Ker(\psi)}$ of $H/F$.
Note that $H/H^{\psi}$ is a $p$-extension since $\psi$ is faithful on $\Gal(H/F)'$.
Since $\ff_{\psi} = \condu{H^{\psi}/F}$, the claim follows from the claim (1).
\end{proof}

We briefly introduce the level raising operator.
Let $\fn$ and $\fm$ be ideals.
Then, for a form $f$ of level $\fn$, there exists a form $f|_{\fm}$ of level $\fn \fm$ whose Fourier coefficients are
\[
\begin{cases}
\coef{\fa}{f|_{\fm}} = \delta_{\fm \mid \fa} \coef{\fa/\fm}{f}\\
\coefz{\lambda}{f|_{\fm}} = \coefz{\lambda \fm}{f}.
\end{cases}
\]

Now we are ready to modify the Eisenstein series, following \cite{DK20}.
We use a formula in \cite[Proposition 8.14]{DK20} as the definition, and will observe the equivalence with \cite[Definition 8.2]{DK20} later.

\begin{defn}\label{defn:Wk}
Let $(H/F, \fn)$ be as in Setting \ref{setting:H/F}. 
Let $\chi$ be an odd faithful character of $\Gal(H/F)'$.
Let $k \geq 1$ be an odd integer.
Let $\fR$ be an ideal of $F$ such that $(\fn, \fR) = 1$.

For each $\fm \mid \fn_p$ with $(\fm, p) = 1$, we have a modular form
\[
E_k^{H^{\fm}/F}(\bpsi_{\fn/\fm}, \btr_{\fR})^{\chi}
\in M_k(\fm^{-1} \fn \fR, \Q_p[\Gal(H^{\fm}/F)]^{\chi}, \bpsi)
\]
by Lemma \ref{lem:8.13n}(2) and Definition \ref{defn:Eisen_group}.
Using these forms, the level raising operators, and \eqref{eq:norm_lift_m}, we define
$W_k^{H/F}(\bpsi_{\fn}, \btr_{\fR})^{\chi} \in M_k(\fn \fR, \Q_p[\Gal(H/F)]^{\chi}, \bpsi)$ by
\[
W_k^{H/F}(\bpsi_{\fn}, \btr_{\fR})^{\chi}
 = \sum_{\fm \mid \fn_p, (\fm, p) = 1}
	\nu_{\fm}
	\parenth{ \bpsi_{\fn/\fm}^{H^{\fm}/F}(\fm)
	\parenth{\prod_{v \mid \fm} \frac{1 - N(v)^k}{\# I_{v, H}}}
	E_k^{H^{\fm}/F}(\bpsi_{\fn/\fm}, \btr_{\fR})^{\chi}|_{\fm}}.
\]

For each odd character $\psi$ of $\Gal(H/F)$ with $\psi|_{G'} = \chi$, we write $W_k(\psi_{\fn}, 1_{\fR}) \in M_k(\fn \fR, \Q_p(\psi), \psi)$ for the $\psi$-component of $W_k^{H/F}(\bpsi_{\fn}, \btr_{\fR})^{\chi}$.
This is actually independent from the extension $H/F$.
\end{defn}

Note that the non-constant terms of $W_k^{H/F}(\bpsi_{\fn}, \btr_{\fR})^{\chi}$ lie in the integral part $\Z_p[\Gal(H/F)]^{\chi}$.
This is because the non-constant terms of the non-modified Eisenstein series are integral by the formula in Definition \ref{defn:Eisen_group} and, for $v \nmid p$, we have
\[
N(v) \equiv 1 \pmod {\# I_{v, H} \Z_p}
\]
by local class field theory.

The following proposition shows that our definition (when $\fR = 1$) coincides with \cite[Definition 8.2]{DK20}.

\begin{prop}\label{prop:W_spec}
Let $(H/F, \fn)$ be as in Setting \ref{setting:H/F}. 
Let $k \geq 1$ be an odd integer.
Let $\psi$ be an odd character of $\Gal(H/F)$ that is faithful on $\Gal(H/F)'$.
We write $\fn = \ff_{\psi} \fQ_{\psi} \fL_{\psi}$ as in Lemma \ref{lem:8.13n}(3).
Then we have
\[
W_k(\psi_{\fn}, 1)
= \sum_{\fm \mid \fL_{\psi}} \mu(\fm) N(\fm)^k \psi(\fm) E_k(\psi_{\ff_{\psi} \fQ_{\psi}}, 1)|_{\fm}.
\]
\end{prop}

\begin{proof}
This is essentially a rephrasing of \cite[Proposition 8.14]{DK20}. 
For the sake of completeness, we review the core computation.

First we observe that $\fL_{\psi} \mid \fn_p$.
For, let $\fl \mid \fn$ be a prime with $\fl \nmid \fn_p$.
Then the inertia group $I_{\fl, H/F}$ is not a $p$-group, so $\fl \mid \ff_{\psi}$ since $\psi$ is faithful on $\Gal(H/F)'$.
By Lemma \ref{lem:8.13n}(3), we then have $\fl \nmid \fL_{\psi}$.

For each $\fm \mid \fn_p$ with $(\fm, p) = 1$, we have $\psi(\nu_{\fm}) = 0$ unless $\fm \mid \fL_{\psi}$, and in that case we have $\psi(\nu_{\fm}) = \prod_{v \mid \fm} \# I_{v, H}$.
Therefore, we have
\begin{align}
W_k(\psi_{\fn}, 1)
& = \sum_{\fm \mid \fL_{\psi}}
	\parenth{ \psi(\fm)
	\parenth{\prod_{v \mid \fm} (1 - N(v)^k)}
	E_k(\psi_{\fn/\fm}, 1)|_{\fm}}.
\end{align}
We expand this to
\begin{align}
W_k(\psi_{\fn}, 1)
& = \sum_{\fm \mid \fL_{\psi}}
	\parenth{ \psi(\fm)
	\sum_{\fm' \mid \fm} \mu(\fm') N(\fm')^k
	E_k(\psi_{\fn/\fm}, 1)|_{\fm}}\\
& = \sum_{\fm' \mid \fL_{\psi}}
	\parenth{ \mu(\fm') N(\fm')^k \psi(\fm')
	\sum_{\fm'' \mid \fL_{\psi}/\fm'} \psi(\fm'')
	E_k(\psi_{\fn/\fm' \fm''}, 1)|_{\fm' \fm''}},
\end{align}
where we write $\fm = \fm' \fm''$.
Then it is enough to show
\[
\sum_{\fm'' \mid \fL_{\psi}/\fm'} \psi(\fm'')
	E_k(\psi_{\fn/\fm' \fm''}, 1)|_{\fm''}
= E_k(\psi_{\ff_{\psi} \fQ_{\psi}}, 1)
\]
for each $\fm' \mid \fL_{\psi}$.
When $\fm' = \fL_{\psi}$, this is clear.
In general, we can prove it by induction (we omit the detail).
\end{proof}

By \cite[Propositions 8.6 and 8.7]{DK20}, we have the following formula for the constant terms of $W_k(\psi_{\fn}, 1)$.

\begin{prop}\label{prop:const_terms}
Under the same notation as in Proposition \ref{prop:W_spec},
for a pair $\cA = (A, \lambda)$ with $A = \begin{pmatrix} a & \ast \\ c & \ast \end{pmatrix}$, we have
\begin{align}
& \coefz{\cA}{W_k(\psi_{\fn}, 1)}\\
& = \delta_{[\cA] \in C_0(\ff_{\psi} \fQ_{\psi}, \fn)}
	\frac{\tau(\psi)}{N(\ff_{\psi})^k} \sgn(\Nrm(-c)) \psi(\fc_{\cA}) 2^{-d} L(\psi^{-1}, 1 - k) \\
	& \qquad \times \prod_{\fp \mid \fQ_{\psi}} \parenth{1 - \frac{\psi(\fp)}{N(\fp)^k}}
	\prod_{\fq \mid \fL_{\psi}, [\cA] \in C_0(\fq, \fn)} (1 - \psi(\fq)) 
	\prod_{\fq \mid \fL_{\psi}, [\cA] \in C_{\infty}(\fq, \fn)} (1 - N(\fq)^k)\\
&\quad + \delta_{k = 1} \delta_{[\cA] \in C_{\infty}(\ff_{\psi} \fL_{\psi}, \fn)}
	\sgn(\Nrm(a)) \psi^{-1}(a \fb_{\cA}^{-1}) 2^{-d} L^{\fL_{\psi}}(\psi, 0) \\
	&  \qquad \times \prod_{\fp \mid \fQ_{\psi}, [\cA] \in C_0(\fp, \fn)} \parenth{1 - N(\fp)^{-1}}
	\prod_{\fp \mid \fQ_{\psi}, [\cA] \in C_{\infty}(\fp, \fn)} (1 - \psi(\fp)).
\end{align}
Here, in the products, $\fp$ and $\fq$ always denote finite primes of $F$.
We simply write $\Nrm$ for the norm from $F$ to $\Q$, and $\sgn$ for the sign of a rational number (see \cite[Remark 8.5]{DK20} for the convention for $\sgn(\Nrm(x))$ in the case $x = 0$).
\end{prop}

\begin{proof}
Note that $\fc_0$, $\fc$, $\fP$, $\fl$, and $T$ in \cite{DK20} 
correspond to $\ff_{\psi}$, $\ff_{\psi} \fQ_{\psi}$, $\fQ_{\psi}$, $\fL_{\psi}$, and $\prim(\fL_{\psi})$ in our notation, respectively.
Then the case where $k > 1$ directly follows from \cite[Proposition 8.6]{DK20}.
Let us assume $k = 1$ and use \cite[Proposition 8.7]{DK20}.
When $\ff_{\psi} \neq 1$, we have $C_0(\ff_{\psi} \fQ_{\psi}, \fn) \cap C_{\infty}(\ff_{\psi} \fL_{\psi}, \fn) = \emptyset$, and the proposition directly follows.
When $\ff_{\psi} = 1$, the right hand side of the proposition is 
\begin{align}
& \quad \delta_{[\cA] \in C_0(\fQ_{\psi}, \fn)}
	\tau(\psi) \sgn(\Nrm(-c)) \psi(\fc_{\cA}) 2^{-d} L(\psi^{-1}, 0) \\
	& \qquad \times \prod_{\fp \mid \fQ_{\psi}} \parenth{1 - \frac{\psi(\fp)}{N(\fp)}}
	\prod_{\fq \mid \fL_{\psi}, [\cA] \in C_0(\fq, \fn)} (1 - \psi(\fq)) 
	\prod_{\fq \mid \fL_{\psi}, [\cA] \in C_{\infty}(\fq, \fn)} (1 - N(\fq))\\
& \quad + \delta_{[\cA] \in C_{\infty}(\fL_{\psi}, \fn)}
	\sgn(\Nrm(a)) \psi^{-1}(a \fb_{\cA}^{-1}) 2^{-d} L^{\fL_{\psi}}(\psi, 0) \\
	&  \qquad \times \prod_{\fp \mid \fQ_{\psi}, [\cA] \in C_0(\fp, \fn)} \parenth{1 - N(\fp)^{-1}}
	\prod_{\fp \mid \fQ_{\psi}, [\cA] \in C_{\infty}(\fp, \fn)} (1 - \psi(\fp)).
\end{align}
By \cite[Remark 8.5]{DK20}, we have
\[
\sgn(\Nrm(-c)) \psi(\fc_{\cA}) = \psi^{-1}(\fd_F \ft_{\lambda} \fb_{\cA})
\]
and
\[
\sgn(\Nrm(a)) \psi^{-1}(a \fb_{\cA}^{-1}) = \psi(\fb_{\cA})
\]
(note that the existence of a totally odd character $\psi$ whose conductor is trivial forces the degree of $F/\Q$ to be even, so we have $\Nrm(-c) = \Nrm(c)$).
Therefore, the proposition coincides with the formulas in \cite[Proposition 8.7]{DK20}.
\end{proof}

\subsection{A further modification}\label{ss:Eisen_modif2}

In this subsection, we make a further modification of $W_1^{H/F}(\bpsi_{\fn}, \btr)^{\chi}$ (we have only to deal with the case where $\fR = 1$ and the weight is $1$).

Let $(H/F, \fn)$ be as in Setting \ref{setting:H/F} and $\chi$ an odd faithful character of $\Gal(H/F)'$.

\begin{defn}\label{defn:P_p}
We put $\fP = \gcd(\fn, p^{\infty})$.
We define Hall divisors $\fP_{\ba}$ and $\fP_p$ of $\fP$ as follows.
Similarly as in Definition \ref{defn:n_p}, we define $\fP_p$ as the Hall divisor of $\fP$ such that $\fp \mid \fP_p$ if and only if $I_{\fp, H}$ is a $p$-group. 
In other words, we have $\fP_p = \gcd(\fP, \fn_p)$.
We define $\fP_{\ba}$ as the Hall divisor of $\fP$ such that $\fp \mid \fP_{\ba}$ if and only if $G_{\fp}$ is a $p$-group.
Then we have
\[
\fP_{\ba} \| \fP_p \| \fP.
\]
\end{defn}

\begin{defn}\label{defn:olW}
For each Hall divisor $\fQ \| \fP_p$, we have a  modular form
\[
W_1^{H^{\fQ}/F}(\bpsi_{\fn/\fQ}, \btr)^{\chi} \in M_1(\fn/\fQ, \Q_p[\Gal(H^{\fQ}/F)]^{\chi}, \bpsi)
\]
by Lemma \ref{lem:8.13n}(2) and Definition \ref{defn:Wk}.
Using these forms and \eqref{eq:norm_lift_m}, we define 
$\ol{W}_1^{H/F}(\bpsi_{\fn}, \btr)^{\chi}
\in M_1(\fn, \Q_p[\Gal(H/F)]^{\chi}, \bpsi)$ by
\[
\ol{W}_1^{H/F}(\bpsi_{\fn}, \btr)^{\chi}
= \sum_{\fQ \| \fP_{\ba}} \nu_{\fQ} W_1^{H^{\fQ}/F}(\bpsi_{\fn / \fQ}, \btr)^{\chi}.
\]
Here, we do not use the level raising operators.
This definition is motivated by the relation between $\theta_{\Sigma}^{\Sigma'}$ and $\Theta_{\Sigma}^{\Sigma'}$ obtained in the proof of Proposition \ref{prop:theta_int}.

For each character $\psi$ of $G$ with $\psi|_{G'} = \chi$, 
we write $\ol{W}_1^{H/F}(\psi_{\fn}, 1) \in M_1(\fn, \Q_p(\psi), \psi)$ 
for the specialization of $\ol{W}_1^{H/F}(\bpsi_{\fn}, \btr)^{\chi}$.
This does depend on the extension $H/F$, so we should not omit the superscript $H/F$.
\end{defn}

Note that the non-constant terms of $\ol{W}_1^{H/F}(\bpsi_{\fn}, \btr)^{\chi}$ again lie in the integral part $\Z_p[\Gal(H/F)]^{\chi}$.

Using Proposition \ref{prop:const_terms}, we show the following formulas for the constant terms of $\ol{W}_1^{H/F}(\psi_{\fn}, 1)$ that involves our Stickelberger element $\theta_{\Sigma}^{\Sigma'}$ introduced in \eqref{eq:theta_defn}.

\begin{prop}\label{prop:bW1_const_terms}
Let $\psi$ be character of $\Gal(H/F)$ with $\psi|_{G'} = \chi$.
We write $\fn = \ff_{\psi} \fQ_{\psi} \fL_{\psi}$ as in Lemma \ref{lem:8.13n}(3).
For a pair $\cA = (A, \lambda)$ with $A = \begin{pmatrix} a & \ast \\ c & \ast \end{pmatrix}$ such that $[\cA] \in C_{\infty}(\fP, \fn)$, we have
\begin{align}
& \coefz{\cA}{\ol{W}_1^{H/F}(\psi_{\fn}, 1)}\\
= & \delta_{(\ff_{\psi}, p) = 1} \delta_{\fP_{\ba} = \fP} \parenth{\prod_{\fp \mid \fP} \# I_{\fp, H}} \delta_{[\cA] \in C_0(\ff_{\psi}, \fn)}
	\frac{\tau(\psi)}{N(\ff_{\psi})} \sgn(\Nrm(-c)) \psi(\fc_{\cA}) 2^{-d} L(\psi^{-1}, 0) \\
	& \qquad \times \prod_{\fq \mid \fL_{\psi}, [\cA] \in C_0(\fq, \fn)} (1 - \psi(\fq)) 
	\prod_{\fq \mid \fL_{\psi}, [\cA] \in C_{\infty}(\fq, \fn)} (1 - N(\fq))\\
&+ \delta_{[\cA] \in C_{\infty}(\fn)}
	\sgn(\Nrm(a)) \psi^{-1}(a \fb_{\cA}^{-1}) 2^{-d} 
	 \psi(\theta_{\Sigma}^{\Sigma', \sharp}) 
	\prod_{\fp \mid \fP / \fP_{\ba}} \psi \parenth{\frac{1 - \frac{\nu_{I_{\fp, H}}}{\# I_{\fp, H}} \varphi_{\fp}}{h_{\fp}^{\sharp}}}.
\end{align}

Here, we put 
$\Sigma = S_{\infty}(F) \cup \prim(\fP)$ and 
$\Sigma' = \prim(\fn / \fP)$. 
\end{prop}

\begin{proof}
We write $I_{\fp} = I_{\fp, H}$.
For each $\fQ \| \fP_{\ba}$,
we first note that 
\[
\psi(\nu_{\fQ}) = 
\begin{cases}
\prod_{\fp \mid \fQ} \# I_{\fp} & (\gcd(\ff_{\psi}, \fQ) = 1)\\
0 & (\gcd(\ff_{\psi}, \fQ) \neq 1).
\end{cases}
\]
This implies
\[
\ol{W}_1^{H/F}(\psi_{\fn}, 1) 
= \sum_{\fQ \| \fP_{\ba}, \gcd(\ff_{\psi}, \fQ) = 1} 
	\parenth{\prod_{\fp \mid \fQ} \# I_{\fp}} 
	W_1(\psi_{\fn / \fQ}, 1).
\]
We shall consider $\fQ \| \fP_{\ba}$ with $\gcd(\ff_{\psi}, \fQ) = 1$.
Then we have $\fQ \mid \fQ_{\psi}$ and $\fn/\fQ = \ff_{\psi} \frac{\fQ_{\psi}}{\fQ} \fL_{\psi}$ is the corresponding decomposition.
Then by Proposition \ref{prop:const_terms}, we have a formula for $\coefz{\cA}{W_1(\psi_{\fn/\fQ}, 1)}$; it simply replaces all $\fn$ and $\fQ_{\psi}$ in the right hand side of Proposition \ref{prop:const_terms} by $\fn/\fQ$ and $\fQ_{\psi}/\fQ$, respectively (and we take $k = 1$).
By the assumption $[\cA] \in C_{\infty}(\fP, \fn)$, we have $[\cA] \not\in C_0(\ff_{\psi} \fQ_{\psi}/\fQ, \fn/\fQ)$ unless $\gcd(\fP, \ff_{\psi} \fQ_{\psi} / \fQ) = 1$, which is equivalent to 
$(\ff_{\psi}, p) = 1$ and $\fQ = \fQ_{\psi} = \fP$.
Moreover, for any $\fp \mid \fP$, we have $[\cA] \in C_{\infty}(\fp, \fn/\fQ)$.
Therefore, we obtain
\begin{align}
& \coefz{\cA}{W_1(\psi_{\fn/\fQ}, 1)}\\
& = \delta_{(\ff_{\psi}, p) = 1} \delta_{\fQ = \fP}
	\delta_{[\cA] \in C_0(\ff_{\psi}, \fn)}
	\frac{\tau(\psi)}{N(\ff_{\psi})} \sgn(\Nrm(-c)) \psi(\fc_{\cA}) 2^{-d} L(\psi^{-1}, 0) \\
	& \qquad \times \prod_{\fq \mid \fL_{\psi}, [\cA] \in C_0(\fq, \fn)} (1 - \psi(\fq)) 
	\prod_{\fq \mid \fL_{\psi}, [\cA] \in C_{\infty}(\fq, \fn)} (1 - N(\fq))\\
& \quad + \delta_{[\cA] \in C_{\infty}(\ff_{\psi} \fL_{\psi}, \fn)}
	\sgn(\Nrm(a)) \psi^{-1}(a \fb_{\cA}^{-1}) 2^{-d} L^{\fL_{\psi}}(\psi, 0)
	\times \prod_{\fp \mid \fQ_{\psi}/\fQ} (1 - \psi(\fp)).
\end{align}
Note that $\fQ = \fP$ is possible if and only if $\fP_{\ba} = \fP$.
Note also that we have 
$\delta_{[\cA] \in C_{\infty}(\ff_{\psi} \fL_{\psi}, \fn)} = \delta_{[\cA] \in C_{\infty}(\fn)}$
by the assumption $[\cA] \in C_{\infty}(\fP, \fn)$.
Taking the sum with respect to $\fQ$, we then obtain
\begin{align}
& \coefz{\cA}{ \ol{W}_1^{H/F}(\psi_{\fn}, 1)}\\
& = \delta_{(\ff_{\psi}, p) = 1} \delta_{\fP_{\ba} = \fP} \parenth{\prod_{\fp \mid \fP} \# I_{\fp}}
	\delta_{[\cA] \in C_0(\ff_{\psi}, \fn)}
	\frac{\tau(\psi)}{N(\ff_{\psi})} \sgn(\Nrm(-c)) \psi(\fc_{\cA}) 2^{-d} L(\psi^{-1}, 0) \\
	& \qquad \times \prod_{\fq \mid \fL_{\psi}, [\cA] \in C_0(\fq, \fn)} (1 - \psi(\fq)) 
	\prod_{\fq \mid \fL_{\psi}, [\cA] \in C_{\infty}(\fq, \fn)} (1 - N(\fq))\\
& \quad +	\delta_{[\cA] \in C_{\infty}(\fn)}
	\sgn(\Nrm(a)) \psi^{-1}(a \fb_{\cA}^{-1}) 2^{-d} L^{\fL_{\psi}}(\psi, 0) \\
	&  \qquad \times \sum_{\fQ \| \fP_{\ba}, (\ff_{\psi}, \fQ) = 1} \parenth{\prod_{\fp \mid \fQ} \# I_{\fp}} \prod_{\fp \mid \fQ_{\psi}/\fQ} (1 - \psi(\fp)).
\end{align}
By the definition of $\fQ_{\psi}$, we have 
$\prod_{\fp \mid \fQ_{\psi}/\fQ} (1 - \psi(\fp)) 
= \prod_{\fp \mid \fP/\fQ} (1 - \psi(\fp))$.
Therefore, the final sum can be computed as
\begin{align}
& \sum_{\fQ \| \fP_{\ba}, (\ff_{\psi}, \fQ) = 1} 
	\parenth{\prod_{\fp \mid \fQ} \# I_{\fp}} 
	\prod_{\fp \mid \fQ_{\psi}/\fQ} (1 - \psi(\fp))\\
& \qquad = \prod_{\fp \mid \fP / \fP_{\ba}} (1 - \psi(\fp)) \cdot \sum_{\fQ \| \fP_{\ba}, (\ff_{\psi}, \fQ) = 1} 
	\parenth{\prod_{\fp \mid \fQ} \# I_{\fp}} 
	\prod_{\fp \mid \fP_{\ba}/\fQ, \fp \nmid \ff_{\psi}} (1 - \psi(\fp))\\
& \qquad = \prod_{\fp \mid \fP / \fP_{\ba}} (1 - \psi(\fp)) \cdot \prod_{\fp \mid \fP_{\ba}, \fp \nmid \ff_{\psi}} (1 - \psi(\fp) + \# I_{\fp})\\
& 
\qquad = \prod_{\fp \mid \fP / \fP_{\ba}} \psi \parenth{1 - \frac{\nu_{I_{\fp}}}{\# I_{\fp}} \varphi_{\fp}} \cdot \prod_{\fp \mid \fP_{\ba}} \psi(h_{\fp}^{\sharp}).
\end{align}
By the definitions of $\theta_{\Sigma}^{\Sigma'}$ and $\Sigma$, we have
\begin{equation}\label{eq:theta_spec}
\psi(\theta_{\Sigma}^{\Sigma', \sharp}) = 
L^{\Sigma'}(\psi, 0) 
\cdot \prod_{\fp \mid \fP} 
\psi(h_{\fp}^{\sharp}).
\end{equation}
By the choice of $\Sigma'$, we have $L^{\Sigma'}(\psi, 0) =  L^{\fL_{\psi}}(\psi, 0)$. 
Thus we obtain the proposition.
\end{proof}

\subsection{Hecke actions}\label{ss:Eisen_Hecke}

In this subsection, we gather formulas of the Hecke actions on various Eisenstein series that we constructed.

We briefly review the general formulas of the Hecke operators $U_{\fl}$ and $T_{\fl}$.
Let $f \in M_k(\fn, R, \psi)$ be a modular form with coefficients in $R$, of level $\fn$, weight $k \geq 1$, and nebentypus $\psi$ as in \S \ref{ss:HMF}.
Then for a prime $\fl \nmid \fn$, the form $T_{\fl} f \in M_k(\fn, R, \psi)$ is characterized by
\[
\begin{cases}
\coef{\fa}{T_{\fl} f} 
= \coef{\fl \fa}{f} + \delta_{\fl \mid \fa} \psi(\fl) N(\fl)^{k - 1} \coef{\fl^{-1} \fa}{f}\\
\coefz{\lambda}{T_{\fl} f} = \coefz{\fl^{-1} \lambda}{f} + \psi(\fl) N(\fl)^{k - 1} \coefz{\fl \lambda}{f}.
\end{cases}
\]
For any prime $\fl$, the form $U_{\fl} f \in M_k(\gcd(\fn, \fl), R, \psi)$ is characterized by
\[
\begin{cases}
\coef{\fa}{U_{\fl} f} = \coef{\fl \fa}{f}\\
\coefz{\lambda}{U_{\fl} f} = \coefz{\fl^{-1} \lambda}{f}.
\end{cases}
\]

Note that we have the following compatibility between the level raising operators and Hecke operators:
$U_{\fl}(f|_{\fm}) = (U_{\fl} f)|_{\fm}$
for $\fl \nmid \fm$
and
$T_{\fl}(f|_{\fm}) = (T_{\fl} f)|_{\fm}$
for $\fl \nmid \fm \fn$.
These can be checked directly by the formulas on Fourier coefficients.

\begin{prop}\label{prop:Hecke_E}
We keep the notation in Definition \ref{defn:Eisen_group}.
\begin{itemize}
\item[(1)]
For a prime $\fl \nmid \fn \fR$, we have
\[
T_{\fl} E_k^{H/F}(\bpsi_{\fn}, \btr_{\fR}) = (\bpsi(\fl) + N(\fl)^{k - 1}) E_k^{H/F}(\bpsi_{\fn}, \btr_{\fR}).
\]
\item[(2)]
For any prime $\fl$, we have
\begin{align}
U_{\fl} E_k^{H/F}(\bpsi_{\fn}, \btr_{\fR})
& = \delta_{\fl \nmid \fR} N(\fl)^{k - 1} E_k^{H/F}(\bpsi_{\fn}, \btr_{\fR})
	+ \bpsi_{\fn}(\fl) E_k^{H/F}(\bpsi_{\fn}, \btr_{\fl \fR})\\
& = \bpsi_{\fn}(\fl) E_k^{H/F}(\bpsi_{\fn}, \btr_{\fR})
	+ \delta_{\fl \nmid \fR} N(\fl)^{k - 1} E_k^{H/F}(\bpsi_{\fl \fn}, \btr_{\fR}).
\end{align}
\end{itemize}
\end{prop}

\begin{proof}
These formulas are standard.
We only compute the non-constant Fourier coefficients.

(1)
We have
\begin{align}
& \coef{\fa}{T_{\fl} E_k^{H/F}(\bpsi_{\fn}, \btr_{\fR})}\\
& = \coef{\fl \fa}{E_k^{H/F}(\bpsi_{\fn}, \btr_{\fR})}
	+ \delta_{\fl \mid \fa} \bpsi_{\fn}(\fl) N(\fl)^{k - 1} \coef{\fl^{-1} \fa}{E_k^{H/F}(\bpsi_{\fn}, \btr_{\fR})}\\
& = \sum_{\fr \mid \fl \fa, (\fr, \fR) = 1} \bpsi_{\fn}\parenth{\frac{\fl \fa}{\fr}} N(\fr)^{k - 1}
	+  \delta_{\fl \mid \fa} \bpsi_{\fn}(\fl) N(\fl)^{k - 1} 
		\sum_{\fr \mid \fl^{-1} \fa, (\fr, \fR) = 1} \bpsi_{\fn}\parenth{\frac{\fl^{-1} \fa}{\fr}} N(\fr)^{k - 1}.
\end{align}
%
We decompose the first sum regarding $\fr \mid \fa$ or $\fr \nmid \fa$.
In the latter case, we put $\fr = \fl \fr'$ with $\fr' \mid \fa$.
Then 
\begin{align}
& \coef{\fa}{T_{\fl} E_k^{H/F}(\bpsi_{\fn}, \btr_{\fR})}\\
& = \bpsi(\fl) \sum_{\fr \mid \fa, (\fr, \fR) = 1} \bpsi_{\fn}\parenth{\frac{\fa}{\fr}} N(\fr)^{k - 1}
	+ N(\fl)^{k - 1} \sum_{\fr' \mid \fa, \fl \fr' \nmid \fa, (\fl \fr', \fR) = 1} \bpsi_{\fn}\parenth{\frac{\fa}{\fr'}} N(\fr')^{k - 1}\\
	& \qquad +  \delta_{\fl \mid \fa} N(\fl)^{k - 1} 
		\sum_{\fr \mid \fl^{-1} \fa, (\fr, \fR) = 1} \bpsi_{\fn}\parenth{\frac{\fa}{\fr}} N(\fr)^{k - 1}.
\end{align}
The first term is nothing but $\bpsi(\fl) \coef{\fa}{E_k^{H/F}(\bpsi_{\fn}, \btr_{\fR})}$.
Moreover, the sum of the second and the third terms is $N(\fl)^{k - 1} \coef{\fa}{E_k^{H/F}(\bpsi_{\fn}, \btr_{\fR})}$.

(2)
We have
\begin{align}\label{eq:Ul_Eis}
\coef{\fa}{U_{\fl} E_k^{H/F}(\bpsi_{\fn}, \btr_{\fR})}
& = \coef{\fl \fa}{E_k^{H/F}(\bpsi_{\fn}, \btr_{\fR})}\\
& = \sum_{\fr \mid \fl \fa, (\fr, \fR) = 1} \bpsi_{\fn}\parenth{\frac{\fl \fa}{\fr}} N(\fr)^{k - 1}.
\end{align}

We decompose \eqref{eq:Ul_Eis} regarding $\fl \mid \fr$ or $\fl \nmid \fr$.
In the former case, we write $\fr = \fl \fr'$.
Then
\begin{align}
\coef{\fa}{U_{\fl} E_k^{H/F}(\bpsi_{\fn}, \btr_{\fR})}
& = \sum_{\fr' \mid \fa, (\fl \fr', \fR) = 1} \bpsi_{\fn}\parenth{\frac{\fa}{\fr'}} N(\fl)^{k - 1} N(\fr')^{k - 1}
	+ \sum_{\fr \mid \fa, (\fr, \fR) = 1, \fl \nmid \fr} \bpsi_{\fn}\parenth{\frac{\fl \fa}{\fr}} N(\fr)^{k - 1}\\
& = \delta_{\fl \nmid \fR} N(\fl)^{k - 1} \coef{\fa}{E_k^{H/F}(\bpsi_{\fn}, \btr_{\fR})}
	+ \bpsi_{\fn}(\fl) \coef{\fa}{E_k^{H/F}(\bpsi_{\fn}, \btr_{\fl \fR})}.
\end{align}

We next decompose \eqref{eq:Ul_Eis} regarding $\fr \mid \fa$ or $\fr \nmid \fa$.
In the latter case, we write $\fr = \fl \fr'$.
Then
\begin{align}
\coef{\fa}{U_{\fl} E_k^{H/F}(\bpsi_{\fn}, \btr_{\fR})}
& = \sum_{\fr \mid \fa, (\fr, \fR) = 1} \bpsi_{\fn}(\fl) \bpsi_{\fn}\parenth{\frac{\fa}{\fr}} N(\fr)^{k - 1}
	+ \sum_{\fr' \mid \fa, (\fl \fr', \fR) = 1, \fl \fr' \nmid \fa} \bpsi_{\fn}\parenth{\frac{\fa}{\fr'}} N(\fl)^{k - 1} N(\fr')^{k - 1}\\
& = \bpsi_{\fn}(\fl) \coef{\fa}{E_k^{H/F}(\bpsi_{\fn}, \btr_{\fR})}
	+ \delta_{\fl \nmid \fR} N(\fl)^{k - 1} \sum_{\fr' \mid \fa, (\fr', \fR) = 1, \fl \nmid \fa/\fr'} \bpsi_{\fn}\parenth{\frac{\fa}{\fr'}}  N(\fr')^{k - 1}\\
& = \bpsi_{\fn}(\fl) \coef{\fa}{E_k^{H/F}(\bpsi_{\fn}, \btr_{\fR})}
	+ \delta_{\fl \nmid \fR} N(\fl)^{k - 1} \coef{\fa}{E_k^{H/F}(\bpsi_{\fl \fn}, \btr_{\fR})}.
\end{align}
Therefore, we obtain the proposition.
\end{proof}

Let $p$ be a fixed odd prime number.
When $R$ is a finite $\Z_p$-algebra, for each $p$-adic prime $\fp \mid \fn$, we define Hida's ordinary operator $e_{\fp}^{\ord}$ on $M_k(\fn, R, \psi)$ by
\[
e_{\fp}^{\ord}= \lim_{n \to \infty} U_{\fp}^{n!}.
\]

\begin{prop}\label{prop:Hecke_W}
We keep the notation in Definition \ref{defn:Wk}.
\begin{itemize}
\item[(1)]
For a prime $\fl \nmid \fn \fR$, we have
\[
T_{\fl} W_k^{H/F}(\bpsi_{\fn}, \btr_{\fR})^{\chi} = (\bpsi(\fl) + N(\fl)^{k - 1}) W_k^{H/F}(\bpsi_{\fn}, \btr_{\fR})^{\chi}.
\]
\item[(2)]
For a prime $\fp \mid p$, we have
\begin{align}
 U_{\fp} W_k^{H/F}(\bpsi_{\fn}, \btr_{\fR})^{\chi}
& = \delta_{\fp \nmid \fR} N(\fp)^{k - 1} W_k^{H/F}(\bpsi_{\fn}, \btr_{\fR})^{\chi}
	+ \bpsi_{\fn}(\fp) W_k^{H/F}(\bpsi_{\fn}, \btr_{\fp \fR})^{\chi}\\
& = \bpsi_{\fn}(\fp) W_k^{H/F}(\bpsi_{\fn}, \btr_{\fR})^{\chi}
	+ \delta_{\fp \nmid \fR} N(\fp)^{k - 1} W_k^{H/F}(\bpsi_{\fp \fn}, \btr_{\fR})^{\chi}.
\end{align}
\item[(3)]
For a prime $\fp \mid p$ with $\fp \nmid \fR$, we have
\[
e_{\fp}^{\ord} W_k^{H/F}(\bpsi_{\fn}, \btr_{\fR})^{\chi} = 
\begin{cases}
	W_1^{H/F}(\bpsi_{\fn}, \btr_{\fR})^{\chi} & (k = 1)\\
	0 & (k > 1, \fp \mid \fn)\\
	\frac{1}{1 - N(\fp)^{k - 1} \bpsi(\fp)^{-1}} W_k^{H/F}(\bpsi_{\fn}, \btr_{\fp \fR})^{\chi} & (k > 1, \fp \nmid \fn).
\end{cases}
\]
\end{itemize}
\end{prop}

\begin{proof}
The claims (1) and (2) follow from Proposition \ref{prop:Hecke_E}(1) and (2), together with the observation that the Hecke operators commute with the level raising operators.
For (2), note also that the level raising operators are used for $\fm$ with $(\fm, p) = 1$, so we actually have $\fp \nmid \fm$.

(3)
We use the first formula of the claim (2).
When $\fp \mid \fn$, then the claim (3) immediately follows.
In the following, we suppose $\fp \nmid \fn$.
By applying the first formula of (2) inductively, with the help of $U_{\fp} W_k^{H/F}(\bpsi_{\fn}, \btr_{\fp \fR})^{\chi}
 = \bpsi(\fp) W_k^{H/F}(\bpsi_{\fn}, \btr_{\fp \fR})^{\chi}$ (again by the first formula of (2)), we have
\[
U_{\fp}^{n!} W_k^{H/F}(\bpsi_{\fn}, \btr_{\fR})^{\chi}
= N(\fp)^{(k - 1) n!} W_k^{H/F}(\bpsi_{\fn}, \btr_{\fR})^{\chi} 
	+ \parenth{\sum_{i = 0}^{n! - 1} N(\fp)^{(k - 1)i} \bpsi(\fp)^{n! - i}} W_k^{H/F}(\bpsi_{\fn}, \btr_{\fp \fR})^{\chi}.
\]
When $k = 1$, it is easy to see that $\sum_{i = 0}^{n! - 1} \bpsi(\fp)^{n! - i}$ goes to $0$ as $n \to \infty$, so the claim follows.
When $k > 1$, 
\[
\sum_{i = 0}^{n! - 1} N(\fp)^{(k - 1)i} \bpsi(\fp)^{n! - i}
= \bpsi(\fp)^{n!} \frac{1 - (N(\fp)^{k - 1} \bpsi(\fp)^{-1})^{n!}}{1 - N(\fp)^{k - 1} \bpsi(\fp)^{-1}},
\]
so the claim follows.
\end{proof}

\begin{prop}\label{prop:Hecke_olW}
We keep the notation in Definition \ref{defn:olW}.
\begin{itemize}
\item[(1)]
For a prime $\fl \nmid \fn$, we have
\[
T_{\fl} \ol{W}_1^{H/F}(\bpsi_{\fn}, \btr)^{\chi} 
= (\bpsi(\fl) + 1) \ol{W}_1^{H/F}(\bpsi_{\fn}, \btr)^{\chi}.
\]
\item[(2)]
For a prime $\fp \mid p$ with $\fp \nmid \fP_{\ba}$, we have
\[
U_{\fp} \ol{W}_1^{H/F}(\bpsi_{\fn}, \btr)^{\chi}
= \bpsi_{\fn}(\fp) \ol{W}_1^{H/F}(\bpsi_{\fn}, \btr)^{\chi}
	+ \ol{W}_1^{H/F}(\bpsi_{\fp \fn}, \btr)^{\chi}.
\]
\item[(3)]
For any prime $\fp \mid p$, we have
\[
e_{\fp}^{\ord} \ol{W}_1^{H/F}(\bpsi_{\fn}, \btr)^{\chi} = \ol{W}_1^{H/F}(\bpsi_{\fn}, \btr)^{\chi}.
\]
\end{itemize}
\end{prop}

\begin{proof}
By Definition \ref{defn:olW}, the claims (1), (2) and (3) follow from Proposition \ref{prop:Hecke_W}(1), (2), and (3), respectively.
%
\end{proof}

Note that the action of $U_{\fp}$ on $\ol{W}_1^{H/F}(\bpsi_{\fn}, \btr)^{\chi}$ when $\fp \mid \fP_{\ba}$ is complicated.
We will encounter this problem in \S \ref{ss:Fitt_B} (especially in Lemma \ref{lem:tW1_coeff}).

\section{Construction of the cusp form}\label{s:cuspform}

In this section, using the modified Eisenstein series constructed in the previous section, we construct a cuspform that plays a crucial role in the proof of Theorem \ref{thm:current_main_p}.


By the propositions in \S \ref{ss:func_state}, in order to prove Theorem \ref{thm:current_main_p}, we may assume that $\chi$ is faithful on $G'$ and moreover choose an arbitrary pair $(\Sigma, \Sigma')$.
For clarity, we describe the situation here.

\begin{setting}\label{setting}
Let $H/F$ be a finite abelian CM-extension.
Let $p$ be an odd prime number.
Let $G'$ denote the maximal subgroup of $G = \Gal(H/F)$ of order prime to $p$.
Let $\chi$ be an odd faithful character of $G'$.

Let $T$ be a finite set of finite primes of $F$ satisfying
\begin{itemize}
\item
$T \cap (S_p(F) \cup S_{\ram}(H/F)) = \emptyset$.
\item
$\OO_H^{\times, T}$ is $p$-torsion-free.
\end{itemize}
We put 
\[
\fn = \condu{H/F} \prod_{\fq \in T} \fq.
\]
Then $(H/F, \fn)$ satisfies the condition of Setting \ref{setting:H/F}.
As in Definition \ref{defn:P_p}, we put
\[
\fP = \gcd(\fn, p^{\infty}) = \gcd(\condu{H/F}, p^{\infty})
\]
and introduce the Hall divisors $\fP_{\ba} \| \fP_p$ of $\fP$.
Then we put
\[
\Sigma = S_\infty(F) \cup \prim(\fP), 
\qquad \Sigma' = \prim(\fn / \fP),
\]
which match the choices in Proposition \ref{prop:bW1_const_terms}.
Note that the conditions (H1), (H2), (H3')$_p^{\chi}$, and (H4)$_p^{\chi}$ hold. 
We simply write 
\[
\theta = \theta_{\Sigma}^{\Sigma', \chi^{-1}} \in \Z_p[G]^{\chi^{-1}},
\]
which is integral by Proposition \ref{prop:theta_int}, so we have $\theta^{\sharp} \in \Z_p[G]^{\chi}$.
Finally we put 
\[
\fP' = \prod_{\fp' \mid p, \fp' \nmid \fP} \fp',
\qquad \fP_{\ba}' = \prod_{\fp' \mid p, \fp' \nmid \fP_{\ba}} \fp',
\]
so $\fn \fP'$ is the minimal ideal divisible by both $\fn$ and all $p$-adic primes.
\end{setting}

Since $(H/F, \fn)$ satisfies the condition of Setting \ref{setting:H/F}, we can use the notation in the previous section.

\begin{defn}\label{defn:x}
When $\fP_{\ba} \neq \fP$, we simply put $\wtil{x_k} = 1 \in \Z_p[G]^{\chi}$ for any odd integer $k > 1$.
Suppose $\fP_{\ba} = \fP$ and recall the field $H^{\fP}$ defined in Definition \ref{defn:H^m}.
For each odd integer $k > 1$, we define
\[
x_k 
= \frac{\Theta({H^{\fP}/F}, 1 - k)^{\chi}}{\Theta({H^{\fP}/F}, 0)^{\chi}} \in \Q_p[\Gal(H^{\fP}/F)]^{\chi}.
\]
Here, for each odd integer $k \geq 1$, we define
\[
\Theta({H^{\fP}/F}, 1 - k) = \sum_{\psi} L(\psi^{-1}, 1 - k) e_{\psi} \in \Q[\Gal(H^{\fP}/F)]^-,
\]
where $\psi$ runs over the odd characters of $\Gal(H^{\fP}/F)$.
By \cite[Lemma 8.16]{DK20}, by taking $k \equiv 1 \pmod{ (p - 1)p^N}$ with $N$ large enough, we have $x_k \in \Z_p[\Gal(H^{\fP}/F)]^{\chi}$ and moreover $x_k$ is a non-zero-divisor.
Let us fix a lift $\wtil{x_k} \in \Z_p[G]^{\chi}$ of $x_k$ which is a non-zero-divisor. 
\end{defn}

We introduce the space of ordinary cuspforms.
Using Hida's ordinary operators, we put
\[
e_{p}^{\ord} = \prod_{\fp \mid p} e_{\fp}^{\ord}. 
\]
Then we define
\[
S_k(\fn \fP', \Z_p[G]^{\chi}, \bpsi)^{p \ordinary} = e_p^{\ord} S_k(\fn \fP', \Z_p[G]^{\chi}, \bpsi).
\]

The following is the goal of this section.

\begin{thm}\label{thm:cuspform}
For large enough $N$, for each odd integer $k > 1$ with $k \equiv 1 \pmod{(p - 1)p^N}$,
there exists a $p$-ordinary cuspform  
$F_k(\bpsi) \in S_k(\fn \fP', \Z_p[G]^{\chi}, \bpsi)^{p \ordinary}$ such that
\begin{equation}\label{eq:Fk}
F_k(\bpsi) 
\equiv \wtil{x_k} \ol{W}_1^{H/F}(\bpsi_{\fn}, \btr)^{\chi}
	- \delta_{\fP_{\ba} = \fP} \nu_{\fP} W_k^{H^{\fP}/F}(\bpsi_{\fn/\fP}, \btr_p)^{\chi}
	 \pmod{ (\wtil{x_k} \theta^{\sharp})}.
\end{equation}
\end{thm}

Here, the congruence between modular forms means the congruences between all the Fourier coefficients $\coef{\fa}{-}$ and $\coefz{\lambda}{-}$.
The proof of this theorem occupies the rest of this section.


We make essential use of the following theorems proved by Silliman \cite{Sil20}.

\begin{thm}[{Silliman \cite[Theorem 3 = Theorem 10.7]{Sil20}, cf. \cite[Theorem 8.8]{DK20}}]\label{thm:Vk}
Let $H/F$ be a finite abelian CM-extension.
Let $m > 0$ be a positive integer.
Then, for large enough $N$, for each odd integer $k > 1$ with $k \equiv 1 \pmod{(p - 1)p^N}$, there exists a modular form $V_{k - 1} \in M_{k - 1}((1), \Z_p, 1)$ (the last $1$ denotes the trivial character) such that
\[
V_{k - 1} \equiv 1 \pmod{p^{m}},
\]
which means that 
\[
\coefz{\lambda}{V_{k - 1}} \equiv 1 \pmod{p^{m}},
\qquad \coef{\fa}{V_{k - 1}} \equiv 0 \pmod{p^{m}}
\]
hold for any $\lambda$ and $\fa$.
\end{thm}

\begin{thm}[{Silliman \cite[Theorem 10.9]{Sil20}, cf. \cite[Theorem 8.9]{DK20}}]\label{thm:Gk}
Let $H/F$ be a finite abelian CM-extension.
Let $\fn$ be an ideal divisible by $\condu{H/F}$, and put $\fP = \gcd(\fn, p^{\infty})$.
For large enough odd integer $k > 1$, there exists a modular form $G_k(\bpsi) \in M_k(\fn, \Z_p[\Gal(H/F)]^-, \bpsi)$ such that
\[
c_{\cA}(0, G_k(\bpsi)) 
= \delta_{[\cA] \in C_{\infty}(\fn)} \sgn(\Nrm(a)) \bpsi^{-1}(a \fb_{\cA}^{-1})
\]
for all $\cA = (A, \lambda)$ with $A = \begin{pmatrix} a & \ast \\ \ast & \ast \end{pmatrix}$ such that $[\cA] \in C_{\infty}(\fP, \fn)$.
\end{thm}

Now we return to Setting \ref{setting}.

Firstly, we show the following lemma
on the fraction that appears in the formula of Proposition \ref{prop:bW1_const_terms}.

\begin{lem}\label{lem:unit}
For 
$\fp \mid \fP / \fP_{\ba}$,
we put
\[ y_{\fp} =  \left(1-\cfrac{\nu_{I_\fp}}{\# I_{\fp}}\varphi_{\fp} \right)^{\chi} 
\left( \left(1-\cfrac{\nu_{I_\fp}}{\# I_{\fp}}\varphi_{\fp} + \nu_{I_\fp} \right)^\chi \right)^{-1}
\in \Frac(\Z_p[G]^{\chi}).
\]
Then $y_{\fp}$ is a unit of $\Z_p[G]^\chi$. 
\end{lem}
\begin{proof}
Firstly, if $I_{\fp}$ is not a $p$-group, then $\nu_{I_{\fp}}^{\chi} = 0$ and so $y_{\fp} = 1$ since $\chi$ is faithful on $G'$.
Therefore, we may assume that $I_{\fp}$ is a $p$-group.
By the assumption $\fp \mid \fP/\fP_{\ba}$, this implies that $p \mid \# I_{\fp}$ and that $G_{\fp}/I_{\fp}$ is not a $p$-group.

We can easily compute
\[
y_{\fp} - 1
= - \frac{\nu_{I_{\fp}}^{\chi}}{\parenth{1 - \frac{\nu_{I_{\fp}}}{\# I_{\fp}}\varphi_{\fp} + \nu_{I_{\fp}}}^{\chi}}
\]
Moreover, since the element $\frac{\nu_{I_{\fp}}}{\# I_{\fp}}$ is an idempotent, we obtain
\[
y_{\fp} - 1
= - \frac{\nu_{I_{\fp}}^{\chi}}{\parenth{1 - \varphi_{\fp} + \nu_{I_{\fp}}}^{\chi}}.
\]
By the observations that $p \mid \# I_{\fp}$ and that $G_{\fp}/I_{\fp}$ is not a $p$-group,
we know that $(1-\varphi_{\fp} + \nu_{I_{\fp}})^\chi$ is a unit of $\Z_p[G]^\chi$
and that $\nu_{I_{\fp}}^{\chi}$ is in the maximal ideal of $\Z_p[G]^{\chi}$.
Therefore, $y_{\fp} - 1$ is in the maximal ideal of $\Z_p[G]^{\chi}$, so $y_{\fp}$ is a unit of $\Z_p[G]^{\chi}$.
\end{proof} 

The following proposition corresponds to \cite[Proposition 8.11]{DK20}.

\begin{prop}\label{prop:f_k}
Let $m > 0$ be a positive integer.
For large enough $N$, for each odd integer $k > 1$ with $k \equiv 1 \pmod{ (p - 1)p^N}$, we take $V_{k - 1}$ and $G_k(\bpsi)$ as in Theorems \ref{thm:Vk} and \ref{thm:Gk}, respectively.
(For $V_{k - 1}$, we actually have to enlarge $m$; see the proof below.)
Let us define $f_k(\bpsi) \in M_k(\fn, \Q_p[G]^{\chi}, \bpsi)$ by
\[
f_k(\bpsi) 
= \ol{W}_1^{H/F}(\bpsi_{\fn}, \btr)^{\chi} V_{k - 1} 
	- \delta_{\fP_{\ba} = \fP} \nu_{\fP} \parenth{x_k^{-1} 
	W_k^{H^{\fP}/F}(\bpsi_{\fn/\fP}, \btr)^{\chi}} 
	- 2^{-d} \prod_{\fp \mid \fP / \fP_{\ba}} y_{\fp} \cdot \theta^{\sharp} G_k(\bpsi)^{\chi}.
\]
Then, for a character $\psi$ of $G$ with $\psi|_{G'} = \chi$, the specialization $f_k(\psi)$ of $f_k(\bpsi)$ satisfies
\[
c_{\cA}(0, f_k(\psi)) \equiv 0 \pmod{ p^{m}}
\]
for any $\cA$ such that $[\cA] \in C_{\infty}(\fP, \fn)$.
\end{prop}

\begin{proof}
First, the properties of $V_{k-1}$ and $G_k(\bpsi)$ show
\begin{align}
& \coefz{\cA}{f_k(\psi)}\\
& \equiv \coefz{\cA}{\ol{W}_1^{H/F}(\psi_{\fn}, 1)}
	- \delta_{\fP_{\ba} = \fP} \delta_{(\ff_{\psi}, p) = 1} \parenth{\prod_{\fp \mid \fP} \# I_{\fp}}
	\frac{L(\psi^{-1}, 0)}{L(\psi^{-1}, 1 - k)} \coefz{\cA}{W_k(\psi_{\fn/\fP}, 1)}\\
&  \qquad - 2^{-d} \prod_{\fp \mid \fP / \fP_{\ba}} \psi(y_{\fp}) \cdot \psi(\theta^{\sharp}) \delta_{[\cA] \in C_{\infty}(\fn)} \sgn(\Nrm(a)) \psi^{-1}(a \fb_{\cA}^{-1})
\end{align}
modulo $(p^{m})$.
Here, since the constant terms of $\ol{W}_1^{H/F}(\psi_{\fn}, 1)$ are not integral in general, we have to choose $V_{k - 1}$ so that a bit higher congruence holds to cancel the denominators.
Then we can apply Propositions \ref{prop:const_terms} and \ref{prop:bW1_const_terms}.
An important point is that the two terms involving $\psi(\theta^{\sharp})$ (the final term of the formula of Proposition \ref{prop:bW1_const_terms} and the final term of the above displayed formula) cancel each other.
Then the remained terms give
\begin{align}
& \coefz{\cA}{f_k(\psi)}\\
& \quad \equiv \delta_{(\ff_{\psi}, p) = 1} \delta_{\fP_{\ba} = \fP} 
	\parenth{\prod_{\fp \mid \fP} \# I_{\fp}} \delta_{[\cA] \in C_0(\ff_{\psi}, \fn)}
	\frac{\tau(\psi)}{N(\ff_{\psi})} \sgn(\Nrm(-c)) \psi(\fc_{\cA}) 2^{-d} L(\psi^{-1}, 0) \\
	& \qquad \qquad \times \prod_{\fq \mid \fL_{\psi}, [\cA] \in C_0(\fq, \fn)} (1 - \psi(\fq)) 
	\prod_{\fq \mid \fL_{\psi}, [\cA] \in C_{\infty}(\fq, \fn)} (1 - N(\fq))\\
	& \qquad \qquad \times \parenth{1 - \frac{1}{N(\ff_{\psi})^{k - 1}} 
	\prod_{\fq \mid \fL_{\psi}, [\cA] \in C_{\infty}(\fq, \fn)} \frac{1 - N(\fq)^k}{1 - N(\fq)}}.
\end{align}
By $k \equiv 1 \pmod{ (p - 1)p^N}$, the last term in the parentheses is $p$-adically close to $0$.
Then the proposition follows.
\end{proof}

Using this proposition, and applying another key theorem of Silliman \cite{Sil20}, let us construct a cuspform.
We define $e_{\fP}^{\ord} = \prod_{\fp \mid \fP} e_{\fp}^{\ord}$ and $S_k(\fn, \Z_p[G]^{\chi}, \bpsi)^{\fP \ordinary} = e_{\fP}^{\ord}S_k(\fn, \Z_p[G]^{\chi}, \bpsi)$.

\begin{thm}\label{thm:cuspform2}
For large enough $N$, for each odd integer $k > 1$ with $k \equiv 1 \pmod{(p - 1)p^N}$,
there exists a cuspform 
$\wtil{F_k}(\bpsi) \in S_k(\fn, \Z_p[G]^{\chi}, \bpsi)^{\fP \ordinary}$ such that
\[
\wtil{F_k}(\bpsi) 
	\equiv \wtil{x_k} \ol{W}_1^{H/F}(\bpsi_{\fn}, \btr)^{\chi}
	- \delta_{\fP_{\ba} = \fP} \nu_{\fP}
	W_k^{H^{\fP}/F}(\bpsi_{\fn/\fP}, \btr_{\fP})^{\chi} \pmod{ (\wtil{x_k} \theta^{\sharp})}.
\]
\end{thm}

\begin{proof}
Let us take an integer $m > 0$ such that $p^{m} \in (\# G \cdot \theta^{\sharp}) \Z_p[G]^{\chi}$.
Such an $m$ exists since $\theta^{\sharp}$ is a non-zero-divisor of $\Rc$.
We then construct $f_k(\bpsi)$ as in Proposition \ref{prop:f_k} associated to this $m$.
By applying Silliman's work \cite[Theorem 10.10]{Sil20} (cf.~\cite[Theorem 8.15]{DK20}) to the family $f_k(\bpsi)$, we find a modular form $h_k(\bpsi) \in M_k(\fn, \Z_p[G]^{\chi}, \bpsi)$ such that the modular form $e_{\fP}^{\ord} \parenth{f_k(\bpsi) - \frac{p^{m}}{\# G} h(\bpsi)}$ is cuspidal: 
\[
e_{\fP}^{\ord} \parenth{f_k(\bpsi) - \frac{p^{m}}{\# G} h(\bpsi)} \in S_k(\fn, \Q_p[G]^{\chi}, \bpsi)^{\fP \ordinary}.
\]
We have
\begin{align}
& f_k(\bpsi) - \frac{p^{m}}{\# G} h(\bpsi)\\
& = \ol{W}_1^{H/F}(\bpsi_{\fn}, \btr)^{\chi} V_{k - 1}
	- \delta_{\fP_{\ba} = \fP} \nu_{\fP} x_k^{-1} W_k^{H^{\fP}/F}(\bpsi_{\fn/\fP}, \btr)^{\chi}
	- 2^{-d} \prod_{\fp \mid \fP / \fP_{\ba}} y_{\fp} \cdot \theta^{\sharp} G_k(\bpsi)^{\chi}
	- \frac{p^{m}}{\# G} h(\bpsi)\\
& \equiv \ol{W}_1^{H/F}(\bpsi_{\fn}, \btr)^{\chi} 
	- \delta_{\fP_{\ba} = \fP} \nu_{\fP} x_k^{-1} W_k^{H^{\fP}/F}(\bpsi_{\fn/\fP}, \btr)^{\chi}
	- \theta^{\sharp} \parenth{ 2^{-d} \prod_{\fp \mid \fP / \fP_{\ba}} y_{\fp} \cdot G_k(\bpsi)^{\chi} + \frac{p^{m}}{\# G \cdot \theta^{\sharp}} h(\bpsi)}
\end{align}
modulo $(p^{m})$ (we again take $V_{k - 1}$ satisfying higher congruences).
By the choice of $m$ and Lemma \ref{lem:unit}, 
the last term multiplied by $\theta^{\sharp}$ is integral.
Therefore, we have
\[
 \wtil{x_k} e_{\fP}^{\ord} \parenth{f_k(\bpsi) - \frac{p^{m}}{\# G} h(\bpsi)}
\equiv \wtil{x_k} e_{\fP}^{\ord}
	\parenth{\ol{W}_1^{H/F}(\bpsi_{\fn}, \btr)^{\chi} 
	- \delta_{\fP_{\ba} = \fP} \nu_{\fP} x_k^{-1} W_k^{H^{\fP}/F}(\bpsi_{\fn/\fP}, \btr)^{\chi}}\\
\]
modulo $(\wtil{x_k} \theta^{\sharp})$.
Using Propositions \ref{prop:Hecke_W}(3) and \ref{prop:Hecke_olW}(3), we then obtain
\begin{align}
& \wtil{x_k} e_{\fP}^{\ord} \parenth{f_k(\bpsi) - \frac{p^{m}}{\# G} h(\bpsi)}\\
&\qquad  \equiv \wtil{x_k} \ol{W}_1^{H/F}(\bpsi_{\fn}, \btr)^{\chi} 
	- \delta_{\fP_{\ba} = \fP} \nu_{\fP} \parenth{\prod_{\fp \mid \fP} \frac{1}{1 - N(\fp)^{k - 1} \bpsi(\fp)^{-1}}} W_k^{H^{\fP}/F}(\bpsi_{\fn/\fP}, \btr_{\fP})^{\chi}.
\end{align}

When $\fP_{\ba} \neq \fP$ (recall $\wtil{x_k} = 1$), we simply put 
\[
\wtil{F_k}(\bpsi) = \wtil{x_k} e_{\fP}^{\ord} \parenth{f_k(\bpsi) - \frac{p^{m}}{\# G} h(\bpsi)}.
\]
When $\fP_{\ba} = \fP$,
let us take a lift $\wtil{\bpsi(\fp)} \in G$ of $\bpsi(\fp) \in \Gal(H^{\fP}/F)$ for each $\fp \mid \fP$, and put
\[
\wtil{F_k}(\bpsi) 
= \parenth{\prod_{\fp \mid \fP} (1 - N(\fp)^{k - 1} \wtil{\bpsi(\fp)}^{-1})} 
	\wtil{x_k} e_{\fP}^{\ord} \parenth{f_k(\bpsi) - \frac{p^{m}}{\# G_p} h(\bpsi)}.
\]
Since $N(\fp)^{k-1} \in (\theta^{\sharp})$ as $N$ is large enough, this form satisfies the condition. 
\end{proof}

\begin{proof}[Proof of Theorem \ref{thm:cuspform}]
We have only to put 
\[
F_k(\bpsi) 
= \parenth{\prod_{\fp \mid \fP'} (1 - N(\fp)^{k - 1} \bpsi(\fp)^{-1}) }
	e_{\fP'}^{\ord} \wtil{F_k}(\bpsi)
\]
and apply Propositions \ref{prop:Hecke_W}(3) and \ref{prop:Hecke_olW}(3).
We again use $N(\fp)^{k-1} \in (\theta^{\sharp})$ as $N$ is large enough.
This completes the proof of Theorem \ref{thm:cuspform}.
\end{proof}

\section{The Galois representation and the proof of the main theorem}\label{s:Gal_rep}

In this section, using the Galois representation associated to the cuspform that we constructed in \S \ref{s:cuspform}, we prove Theorem \ref{thm:current_main_p}.
In \S \ref{ss:hom_Hecke}, we introduce Hecke algebras, homomorphisms associated to the cuspform, and an Eisenstein ideal.
In \S \ref{ss:Gal_rep}, we introduce the associated Galois representation and observe technical properties.
In \S \ref{ss:strategy}, we explain the strategy to prove Theorem \ref{thm:current_main_p}.
The idea is realized in \S \S \ref{ss:abc} -- \ref{ss:Fitt_B}.

We keep Setting \ref{setting}.
We fix $k > 1$ and $F_k(\bpsi)$ as in Theorem \ref{thm:cuspform}.
We write $\wtil{x} = \wtil{x_k}$.

\subsection{Hecke algebras and a homomorphism}\label{ss:hom_Hecke}

We introduce suitable Hecke algebras in a similar manner as in \cite[\S 8.5]{DK20}.

\begin{defn}\label{defn:Hecke_alg}
We define Hecke algebras 
\[
\bT' \subset \wtil{\bT} \subset \End_{\Z_p[G]^{\chi}} \parenth{S_k(\fn \fP', \Z_p[G]^{\chi}, \bpsi)^{p \ordinary}},
\]
which are $\Z_p[G]^{\chi}$-algebras, as follows.
The smaller algebra $\bT'$ is generated by the Hecke operators $T_{\fl}$ for the finite primes $\fl \nmid \fn \fP'$.
The larger algebra $\wtil{\bT}$ is generated by those operators and in addition by the Hecke operator $U_{\fp}$ for all $\fp \mid p$.
(We avoid using the notation $\bT$, which in \cite{DK20} denotes the intermediate ring generated by $\bT'$ and by $U_{\fp}$ for $\fp \mid \fP$.
In this paper we do not need to make use of this intermediate algebra.)
\end{defn}


Let us construct a $\Z_p[G]^{\chi}/(\wtil{x} \theta^{\sharp})$-algebra $W$ and a homomorphism $\varphi: \wtil{\bT} \to W$ as follows.

\begin{defn}\label{defn:psi_small}
Let us consider the Fourier expansion
\[
S_k(\fn \fP', \Z_p[G]^{\chi}, \bpsi)^{p \ordinary}
\hookrightarrow \prod_{\fa} \Z_p[G]^{\chi}
\]
that sends $f$ to $(\coef{\fa}{f})_{\fa}$.
We denote by $c$ the composite of this map and the projection to $\prod_{\fa} \Z_p[G]^{\chi}/(\wtil{x} \theta^{\sharp})$.
Then we have a natural homomorphism
\[
\wtil{\bT} \to \End_{\Z_p[G]^{\chi}/(\wtil{x} \theta^{\sharp})} \parenth{c \parenth{ (F_k(\bpsi))_{\wtil{\bT}}}},
\]
where $(F_k(\bpsi))_{\wtil{\bT}}$ denotes the $\wtil{\bT}$-submodule generated by $F_k(\bpsi)$.
We define a $\Z_p[G]^{\chi}/(\wtil{x} \theta^{\sharp})$-algebra $W$ as the image of this homomorphism, and define $\varphi: \wtil{\bT} \to W$ as the induced (surjective) homomorphism.
\end{defn}

Note that we do not claim that the structure morphism $\Z_p[G]^{\chi}/(\wtil{x} \theta^{\sharp}) \to W$ is injective.
This is different from \cite[Theorem 8.23]{DK20}.
The corresponding property will be implicitly observed in the proof of Theorem \ref{thm:Fitt_B}.

\begin{lem}\label{lem:eigen}
For each $\fl \nmid \fn \fP'$, we have $\varphi(T_{\fl}) = \bpsi(\fl) + N(\fl)^{k - 1}$ as elements of $W$.
\end{lem}

\begin{proof}
For each prime $\fl \nmid \fn \fP'$, we have
\[
T_{\fl} F_k(\bpsi) \equiv (\bpsi(\fl) + N(\fl)^{k - 1}) F_k(\bpsi)
\]
modulo $(\wtil{x} \theta^{\sharp})$.
This follows from \eqref{eq:Fk} and Propositions \ref{prop:Hecke_W}(1) and \ref{prop:Hecke_olW}(1).
Since $\wtil{\bT}$ is commutative, $T_{\fl}$ acts as $\bpsi(\fl) + N(\fl)^{k - 1}$ on $c((F_k(\bpsi))_{\wtil{\bT}})$.
Then the lemma follows.
\end{proof}

\begin{defn}\label{defn:Eis_ideal}
We define the Eisenstein ideal $I' \subset \bT'$ as the kernel of the composite map
\[
\bT' \hookrightarrow \wtil{\bT} \overset{\varphi}{\twoheadrightarrow} W.
\]
\end{defn}

Note that Lemma \ref{lem:eigen} says $T_{\fl} - (\bpsi(\fl) + N(\fl)^{k - 1}) \in I'$ for $\fl \nmid \fn \fP'$.
Therefore, the structure morphism $\Rc \to \bT'/I'$ is surjective, which implies that $\bT'/I'$ is a local ring unless we have $I' = \bT'$.

\subsection{The Galois representation}\label{ss:Gal_rep}

We define $\bK$ as the total ring of fractions of $\wtil{\bT}$.
As explained in the proof of the following theorem, $\bK$ is a finite product of $p$-adic fields.

Let $\varepsilon_{\cyc}: G_F \to \Z_p^{\times}$ be the cyclotomic character.
Thanks to the celebrated work of Hida and Wiles, we have a Galois representation as follows.

\begin{thm}\label{thm:Gal_rep}
There exists a finite extension $E$ of $\Q_p$ and a continuous two-dimensional Galois representation $\rho$ of $G_F$ over $\bK_E := E \otimes_{\Q_p} \bK$ satisfying the following.
\begin{itemize}
\item[(a)]
For each prime $\fl \nmid \fn \fP'$, $\rho$ is unramified at $\fl$ and we have 
\[
\cha(\rho(\varphi_{\fl}))(X) = X^2 - T_{\fl} X + \bpsi(\varphi_{\fl}) N(\fl)^{k - 1},
\]
where $\cha(\rho(\sigma))(X)$ denotes the characteristic polynomial of $\rho(\sigma)$ for $\sigma \in G_F$.
\item[(b)]
For each prime $\fp \mid p$, we have an equivalence of representations
\[
\rho|_{G_{F_{\fp}}} \sim
\begin{pmatrix}
\bpsi \eta_{\fp}^{-1} \varepsilon_{\cyc}^{k - 1}
& * \\
0
& \eta_{\fp}
\end{pmatrix},
\]
where $\eta_{\fp}$ denotes the unramified character $G_{F_{\fp}} \to \wtil{\bT}^{\times}$ that sends any lift of the arithmetic Frobenius to $U_{\fp}$.
\end{itemize}
\end{thm}

\begin{proof}
We explain how to deduce this theorem from the work of Hida and Wiles, following \cite[\S \S 8.5, 9.1]{DK20}.
For a character $\psi$ of $G$ such that $\psi|_{G'} = \chi$, let $M_{\psi}$ be the (finite) set of $p$-ordinary cuspidal new eigenforms of weight $k$, level dividing $\fn \fP'$, and nebentypus $\psi$.
We also put $M = \bigcup_{\psi} M_{\psi}$.
For each $f \in M_{\psi}$, let $f_p$ be its ordinary stabilization with respect to all $p$-adic primes.
Then $f_p$ is an eigenform with respect to both $T_{\fl}$ for $\fl \nmid \fn \fP'$ and $U_{\fp}$ for $\fp \mid p$.
The eigenvalue for $T_{\fl}$ is $\coef{\fl}{f_p} = \coef{\fl}{f}$, and that for $U_{\fp}$ is the unit root $\coef{\fp}{f_p}$ of $X^2 - \coef{\fp}{f} X + \psi(\fp)N(\fp)^{k - 1}$.

Let us take a finite extension $E$ of $\Q_p$ that contains all the Fourier coefficients $\coef{\fa}{f_p}$ for all $f \in M$.
Let $\OO$ be the integer ring of $E$.
Then we have a well-defined $\Z_p$-algebra homomorphism
\[
\wtil{\bT} \to \prod_{f \in M} \OO
\]
that sends $T_{\fl}$ to $(\coef{\fl}{f_p})_{f \in M}$ for $\fl \nmid \fn \fP'$ and $U_{\fp}$ to $(\coef{\fp}{f_p})_{f \in M}$ for $\fp \mid p$.
Then, as in \cite[\S 8.5]{DK20}, the induced $\OO$-algebra homomorphism $\OO \otimes_{\Z_p} \wtil{\bT} \to \prod_{f \in M} \OO$ is injective with finite cokernel.
Therefore, we have an isomorphism 
\begin{equation}\label{eq:EtimesK}
\bK_E \simeq \prod_{f \in M} E
\end{equation}
 as $E$-algebras.

For each $f \in M_{\psi}$, by \cite[Theorems 1 and 2]{Wil88}, there exists a continuous two-dimensional Galois representation $\rho_f$ of $G_F$ over $E$ satisfying the following.
\begin{itemize}
\item[(a)$_f$]
$\rho_f$ is unramified outside $\fn \fP'$ and, for each $\fl \nmid \fn \fP'$, we have 
\[
\cha(\rho(\varphi_{\fl}))(X) = X^2 - \coef{\fl}{f} X + \psi(\fl) N(\fl)^{k - 1}.
\]
\item[(b)$_f$]
For each prime $\fp \mid p$, we have an equivalence of representations
\[
\rho|_{G_{F_{\fp}}} \sim
\begin{pmatrix}
\psi \eta_{\fp, f}^{-1} \varepsilon_{\cyc}^{k - 1}
& * \\
0
& \eta_{\fp, f}
\end{pmatrix},
\]
where $\eta_{\fp, f}$ denotes the unramified character $G_{F_{\fp}} \to E^{\times}$ that sends any lift of the arithmetic Frobenius to $\coef{\fp}{f_p}$.
\end{itemize}
Then, by the isomorphism \eqref{eq:EtimesK}, the tuple $(\rho_f)_{f \in M}$ can be regarded as a continuous two-dimensional representation $\rho$ of $G_F$ over $\bK_E$ satisfying the desired properties (a) and (b).
\end{proof}

Note that it is indeed possible to show that the representation $\rho$ can be defined over $\bK$ instead of $\bK_E$.
However, we omit the proof and will work over $\bK_E$.

We investigate this Galois representation $\rho$.

\begin{lem}\label{lem:char_poly}
For any $\sigma \in G_F$, we have $\cha(\rho(\sigma))(X) \in \bT'[X]$ and
\[
\cha(\rho(\sigma))(X) \equiv (X - \bpsi(\sigma)) (X - \varepsilon_{\cyc}^{k - 1}(\sigma))
\pmod{ I'}.
\]
\end{lem}

\begin{proof}
The claims holds for $\sigma = \varphi_{\fl}$ for any $\fl \nmid \fn \fP'$ by the condition (a) and Lemma \ref{lem:eigen}.
Then the general case follows from Chebotarev's density theorem.
\end{proof}

From now on, let us fix an element $\tau \in G_F$ which is a lift of the complex conjugation $c \in \Gal(H/F)$ such that, for any $f \in M$ and any $\fp \mid p$, the one-dimensional subspace of $\rho_f$ which is stable under $G_{F_{\fp}}$ is not stable by $\tau$.
We can check the existence of such $\tau$ in the same way as \cite[Proposition 9.3]{DK20}.
Then, following \cite[\S 9.1 -- 9.2]{DK20}, we take a basis of the representation $\rho$ as follows.

\begin{lem}\label{lem:choice_basis}
There exists a basis of $\rho$ over $\bK_E$ such that the presentation matrix of $\tau$ is of the form
\[
\rho(\tau) =
\begin{pmatrix}
\lambda_1 & 0 \\ 0 & \lambda_2
\end{pmatrix}
\]
for some $\lambda_1, \lambda_2 \in \bT'$ such that
\[
\lambda_1 \equiv \varepsilon_{\cyc}(\tau)^{k - 1}, 
\qquad \lambda_2 \equiv \bpsi(\tau) \pmod{ I'}.
\]
\end{lem}

\begin{proof}
Since $\bpsi(\tau) \in \Z_p[G]^{\chi}$ is the image of the complex conjugation $c$, we actually have $\bpsi(\tau) = -1$.
This implies that $\bpsi(\tau)$ and $\varepsilon_{\cyc}^{k - 1}(\tau)$ are not congruent to each other modulo the Jacobson radical of $\bT'$.
By Lemma \ref{lem:char_poly}, applying Hensel's lemma, we then have two roots $\lambda_1, \lambda_2 \in \bT'$ of $\cha(\rho(\tau))(X)$ with the described congruences.
The existence of an appropriate basis follows from basic linear algebra.
\end{proof}

Let us fix a basis of $\rho$ over $\bK_E$ as in Lemma \ref{lem:choice_basis}.
For any $\sigma \in G_F$, we write
\[
\rho(\sigma) = 
\begin{pmatrix}
a(\sigma) & b(\sigma) \\ c(\sigma) & d(\sigma)
\end{pmatrix}
\in \GL_2(\bK_E).
\]

By the property (b) in Theorem \ref{thm:Gal_rep}, for each $\fp \mid p$, there exists a matrix
\[
M_{\fp} = 
\begin{pmatrix}
A_{\fp} & B_{\fp} \\ C_{\fp} & D_{\fp}
\end{pmatrix}
\in \GL_2(\bK_E)
\]
such that, for any $\sigma \in G_{F_{\fp}}$, we have
\begin{equation}\label{eq:ABCD}
\rho(\sigma) M_{\fp} = M_{\fp} \begin{pmatrix}
\bpsi \eta_{\fp}^{-1} \varepsilon_{\cyc}^{k - 1}
& * \\
0
& \eta_{\fp}
\end{pmatrix}.
\end{equation}
(There is a conflict between this element $A_{\fp}$ and the module defined in \eqref{eq:A_v}, but we have no afraid of confusion.)
By the choice of $\tau$, we have $A_{\fp}, C_{\fp} \in \bK_E^{\times}$.
We put 
\[
\AC{\fp} = A_{\fp}/C_{\fp} \in \bK_E^{\times}.
\]

For later use, let us observe some properties of $a(\sigma)$, $b(\sigma)$, and $d(\sigma)$.

\begin{lem}\label{lem:observ}
The following hold.
\begin{itemize}
\item[(1)]
For any $\sigma \in G_F$, we have
\[
a(\sigma) \equiv \varepsilon_{\cyc}^{k - 1}(\sigma), \qquad 
d(\sigma) \equiv \bpsi(\sigma) \pmod{ I'}.
\]
\item[(2)]
For each $\fp \mid p$ and $\sigma \in G_{F_{\fp}}$, we have
\[
b(\sigma) = (\bpsi \eta_{\fp}^{-1} \varepsilon_{\cyc}^{k - 1}(\sigma) - a(\sigma)) \AC{\fp}.
\]
\item[(3)]
For each $\fp \mid p$ and $\sigma \in G_{F_{\fp}}$, we have
\[
\bpsi(\sigma)^{-1} \varepsilon_{\cyc}^{1 - k}(\sigma) b(\sigma) 
\equiv (\eta_{\fp}(\sigma)^{-1} - \bpsi(\sigma)^{-1}) \AC{\fp}
\pmod{\AC{\fp} I'}.
\]
\end{itemize}
\end{lem}

\begin{proof}
(1)
By Lemma \ref{lem:char_poly}, we have 
\[
a(\sigma) + d(\sigma) 
\equiv \bpsi(\sigma) + \varepsilon_{\cyc}^{k - 1}(\sigma) \pmod{ I'}.
\]
By replacing $\sigma$ by $\sigma \tau$, we obtain
\[
a(\sigma) \lambda_1 + d(\sigma) \lambda_2 
\equiv \bpsi(\sigma) \lambda_2 + \varepsilon_{\cyc}^{k - 1}(\sigma) \lambda_1 \pmod{ I'}.
\]
Since $\lambda_1 - \lambda_2$ is a unit of $\bT'$, by solving these equations, we obtain the claim.

(2)
This follows immediately from \eqref{eq:ABCD}.

(3) 
This follows immediately from the claims (1) and (2).
\end{proof}

\subsection{The strategy}\label{ss:strategy}

In this subsection, we deduce Theorem \ref{thm:current_main_p} from two theorems that will be proved in the subsequent subsections.

We first define submodules $B_0$, $B$, $B'$, and $B''$ of $\bK_E$.
The definitions of the first three are directly inspired by the corresponding objects in \cite[\S \S 9.3 -- 9.4]{DK20}.
On the other hand, introducing the module $B''$ is a novel part of this paper.

Recall that $\fP_{\ba}$, $\fP_{\ba}'$, and $\fP'$ are introduced in Setting \ref{setting}.

\begin{defn}\label{defn:B}
We define $\bT'$-submodules $B' \subset B_0 \subset B$ of $\bK_E$ by
\begin{align}
B_0 &= (b(\sigma) : \sigma \in G_F)_{\bT'},\\
B &= B_0 + \parenth{ \AC{\fp} : \fp \mid \fP_{\ba} }_{\bT'},\\
B' &= \left(b(\sigma) : \sigma \in \bigcup_{\fp \mid \fP'} I_{F_\fp} \right)_{\bT'}. 
\end{align}
\end{defn}

Before introducing $B''$, we need a lemma in advance.

\begin{lem}\label{lem:B''}
For each $\fp \mid \fP_{\bad}$, we have
\[
U_{\fp}^{-1} \AC{\fp} \in B.
\]
\end{lem}

\begin{proof}
Let $\sigma \in G_{F_{\fp}}$ be a lift of the Frobenius.
Then Lemma \ref{lem:observ}(3) implies
\[
\bpsi(\sigma)^{-1} \varepsilon_{\cyc}^{1 - k}(\sigma) b(\sigma) 
\equiv (U_{\fp}^{-1} - \bpsi(\sigma)^{-1}) \AC{\fp}
\pmod{\AC{\fp} I'}.
\]
Then the lemma follows since the left hand side is in $B_0 \subset B$, we have $\AC{\fp} I' \subset \AC{\fp} \bT' \subset B$, and $\bpsi(\sigma)^{-1} \AC{\fp} \in B$.
\end{proof}

\begin{defn}\label{defn:B''}
We define a $\bT'$-submodule $B''$ of $\bK_E$ by
\[
B'' = \parenth{ (1 - \nu_{I_{\fp}}\bpsi(\varphi_{\fp})^{-1} - U_{\fp}^{-1}) \AC{\fp} : \fp \mid \fP_{\ba} }_{\bT'}.
\]
By Lemma \ref{lem:B''}, we have $B'' \subset B$.
The necessity of this module $B''$ will become clear in Lemma \ref{lem:gamma}.
\end{defn}

\begin{defn}
We define $\bT'$-modules $\ol{B_p}$ and $\ol{B_1}$ by
\begin{align}
\ol{B_p} &= B / (I'B +\theta^{\sharp} B + B' + B''),\\
\ol{B_1} &= B / (I'B +\theta^{\sharp} B + B_0 + B'').
\end{align}
We also define a $\bT'$-module $\ol{B_0} \subset \ol{B_p}$ as the image of $B_0$.
\end{defn}

By the definitions, we have a natural exact sequence of $\bT'$-modules
\begin{equation}\label{eq:Bs1}
0 \to \ol{B_0} \to \ol{B_p} \to \ol{B_1} \to 0.
\end{equation}
Note that, since $\ol{B_p}$ is annihilated by $I'$, this is a sequence of $\bT'/I'$-modules.
Moreover, $\bT'/I'$ is a $\Z_p[G]^{\chi}$-algebra whose structure morphism $\Z_p[G]^{\chi} \to \bT'/I'$ is surjective
(see the text after Definition \ref{defn:Eis_ideal}).
For these reasons, we regard \eqref{eq:Bs1} as a sequence over $\Z_p[G]^{\chi}$.

Recall that $\Sigma$ and $\Sigma'$ are chosen as in Setting \ref{setting}.
Now our goal, Theorem \ref{thm:current_main_p}, follows from the following two theorems whose proof is given in the subsequent subsections. 

\begin{thm}\label{thm:abc}
There exist surjective homomorphisms $\alpha$, $\beta$, and $\gamma$ which fit in a commutative diagram
\[
\xymatrix{
0 \ar[r]
& \Cl_H^{\Sigma', \chi^{-1}} \ar[r] \ar@{->>}[d]_{\alpha}
& \Omega_{\Sigma}^{\Sigma', \chi^{-1}} \ar[r] \ar@{->>}[d]_{\beta}
& \bigoplus_{v \in \Sigma_{f}} A_v^{\chi^{-1}} \ar[r] \ar@{->>}[d]_{\gamma}
& 0\\
0 \ar[r]
& \ol{B_0}^{\sharp} \ar[r]
& \ol{B_p}^{\sharp} \ar[r]
& \ol{B_1}^{\sharp} \ar[r]
& 0
}
\]
over $\Z_p[G]^{\chi^{-1}}$, where the upper sequence is the $\chi^{-1}$-component of \eqref{eq:fund} and the lower is the twist of \eqref{eq:Bs1}.
\end{thm}

\begin{thm}\label{thm:Fitt_B}
Recall that $I_p = \sum_{\fp \in S_p(F)} I_{\fp} \subset G$.
Then we have
\[
\Fitt_{\Rc}(\ol{B_p}) \subset (\theta^{\sharp}, \delta_{\fP_{\ba}' = 1} \nu_{I_p}^{\chi})
\]
as ideals of $\Z_p[G]^{\chi}$.
\end{thm}

Let us deduce Theorem \ref{thm:current_main_p} (for $\chi^{-1}$ instead of $\chi$) from these two theorems.
By Theorems \ref{thm:abc} and \ref{thm:Fitt_B},
we have the first and the last inclusions of
\[
\Fitt_{\Z_p[G]^{\chi^{-1}}}(\Omega_{\Sigma}^{\Sigma', \chi^{-1}})
\subset \Fitt_{\Z_p[G]^{\chi^{-1}}}(\ol{B_p}^{\sharp})
= \Fitt_{\Z_p[G]^{\chi}}(\ol{B_p})^{\sharp}
\subset (\theta_{\Sigma}^{\Sigma', \chi^{-1}}, \delta_{\fP_{\ba}' = 1} \nu_{I_p}^{\chi^{-1}})
\]
in $\Z_p[G]^{\chi^{-1}}$.
Recall that, by the definition, $\fP_{\ba}' = 1$ if and only if all the $p$-adic primes are ramified in $H/F$ and moreover the decomposition groups are $p$-groups.
Thus the $\chi^{-1}$-component of Theorem \ref{thm:current_main_p} follows.

\subsection{The existence of the homomorphisms}\label{ss:abc}

In this subsection, we prove Theorem \ref{thm:abc}.

\subsubsection{Construction of $\alpha$}\label{sss:a}

We define a map $\kappa: G_F \to \ol{B_0}^{\sharp}$ by
\[
\kappa(\sigma) = \bpsi(\sigma)^{-1} \ol{b(\sigma)}.
\]

\begin{lem}\label{lem:a_cocycle}
The map $\kappa$ is a cocycle, so we have $[\kappa] \in H^1(G_F, \ol{B_0}^{\sharp})$.
\end{lem}

\begin{proof}
Let $\sigma, \sigma' \in G_F$.
By $\rho(\sigma \sigma') = \rho(\sigma) \rho(\sigma')$, we have
\[
b(\sigma \sigma') = a(\sigma) b(\sigma') + b(\sigma) d(\sigma').
\]
By Lemma \ref{lem:observ}(1) and $\varepsilon_{\cyc}^{k - 1}(\sigma) \equiv 1 \pmod{\theta^{\sharp}}$, we obtain
\[
\ol{b(\sigma \sigma')} = \ol{b(\sigma')} + \bpsi(\sigma') \ol{b(\sigma)}
\]
as elements of $\ol{B_0}$.
This shows $\kappa(\sigma \sigma') = \kappa(\sigma) + \bpsi^{-1}(\sigma) \kappa(\sigma')$.
\end{proof}

Then we obtain the restriction 
\[
[\kappa]|_{G_H} \in H^1(G_H, \ol{B_0}^{\sharp})^G = \Hom_G(G_H, \ol{B_0}^{\sharp}).
\]
This cocycle induces the desired homomorphism $\alpha$ as follows.

\begin{prop}\label{prop:kappa_unram}
The restriction $[\kappa]|_{G_H} \in H^1(G_H, \ol{B_0}^{\sharp})$ is unramified at any places of $H$ not lying above a place in 
$\Sigma' = \prim(\fn / \fP)$.
Therefore, by class field theory, the cocycle $[\kappa]|_{G_H}$ induces a homomorphism
$\alpha: \Cl_H^{\Sigma', \chi^{-1}} \to \ol{B_0}^{\sharp}$.
\end{prop}

\begin{proof}
First we show that $[\kappa]|_{G_H}$ is unramified at places not lying above a place in $\prim(\fn/\fP_{\ba}) = \Sigma' \cup \prim(\fP/\fP_{\ba})$.
For any prime $\fl \nmid \fn \fP'$, since $\rho$ is unramified at $\fl$, we have $b(\sigma) = 0$ for $\sigma \in I_{F_{\fl}}$, so $\kappa(\sigma) = 0$.
For any $\fp \mid \fP'$ and $\sigma \in I_{F_{\fp}}$, since $b(\sigma) \in B'$ by the very definition of $B'$, we have $\kappa(\sigma) = 0$.
For any $\fp \mid \fP_{\bad}$, a place $\fp_H$ of $H$ lying above $\fp$, and $\sigma \in I_{H_{\fp_H}} \subset I_{F_{\fp}}$, we have by Lemma \ref{lem:observ}(3)
\[
\bpsi(\sigma)^{-1}b(\sigma) 
\equiv (\eta_{\fp}(\sigma)^{-1} - \bpsi(\sigma)^{-1}) \AC{\fp}
\pmod {\AC{\fp} I'}.
\]
Since $\eta_{\fp}(\sigma) = 1$ (by the unramifiedness of $\eta_{\fp}$) and $\bpsi(\sigma) = 1$, the right hand side vanishes.
This shows $\kappa(\sigma) = 0$.
For any $\fp \mid \fP / \fP_{\ba}$ and
a place $\fp_H$ of $H$ lying above $\fp$, 
we write $U^1_{H_{\fp_H}}$ for the principal local unit group of $H_{\fp_H}$. 
Then we have 
\[
\left( \prod_{\fp_H \mid \fp} U^1_{H_{\fp_H}} \right)^\chi 
\simeq \left( U^1_{H_{\fp_H}} \otimes \Z_p[G / G_{\fp}] \right)^\chi =0
\]
 since 
the decomposition group $G_{\fp} \subset G$ is not a $p$-group and $\chi$ is a faithful character. 
This implies that $H^1(I_{H_{\fp_H}}, \ol{B_0}^{\sharp})^G = \Hom_G(I_{H_{\fp_H}}, \ol{B_0}^{\sharp}) = 0$ by local class field theory.   
Therefore, $[\kappa]|_{G_H}$ is unramified at each $\fp \mid \fP / \fP_{\ba}$. 
Thus we obtain the proposition. 
\end{proof}

\begin{lem}
The homomorphism $\alpha$ is surjective.
\end{lem}

\begin{proof}
This can be shown in a similar way as \cite[Corollary 9.6]{DK20}.
\end{proof}

\subsubsection{Construction of $\gamma$}\label{sss:c}

We first note that $A_{\fp}^{\chi^{-1}} = 0$ for $\fp \in \Sigma_f \setminus \prim(\fP_{\ba}) = \prim(\fP / \fP_{\ba})$ (see the proof of Lemma \ref{lem:U_ct}).
Therefore, we have
\[
\bigoplus_{\fp \in \Sigma_f} A_{\fp}^{\chi^{-1}} = 
\bigoplus_{\fp \mid \fP_{\bad}} A_{\fp}^{\chi^{-1}}.
\]
We show a lemma which makes essential use of the definition of $B''$ in Definition \ref{defn:B''}.

\begin{lem}\label{lem:gamma}
Let $\fp \mid \fP_{\bad}$.
\begin{itemize}
\item[(a)]
For $\sigma \in I_{F_{\fp}}$, we have
\[
(1 - \bpsi(\sigma)^{-1}) \AC{\fp} = \kappa(\sigma)
\]
in $\ol{B_p}$.
\item[(b)]
For a lift $\wtil{\varphi_{\fp}}$ of the Frobenius, we have
\[
(1 - \nu_{I_{\fp}}\bpsi(\varphi_{\fp})^{-1} - \bpsi(\wtil{\varphi_{\fp}})^{-1}) \AC{\fp} 
= \kappa(\wtil{\varphi_{\fp}})
\]
in $\ol{B_p}$.
\end{itemize}
\end{lem}

\begin{proof}
For any $\sigma \in G_{F_{\fp}}$, we have by Lemma \ref{lem:observ}(3)
\[
(\eta_{\fp}(\sigma)^{-1} - \bpsi(\sigma)^{-1}) \AC{\fp} 
= \bpsi(\sigma)^{-1} \ol{b(\sigma)}
= \kappa(\sigma)
\]
in $\ol{B_p}$.
Then the claim (a) immediately follows since $\eta_{\fp}(\sigma) = 1$.
The claim (b) also follows, since $\eta_{\fp}(\wtil{\varphi_{\fp}}) = U_{\fp}$ and $(1 - \nu_{I_{\fp}}\bpsi(\varphi_{\fp})^{-1} - U_{\fp}^{-1}) \AC{\fp} = 0$ in $\ol{B_p}$ by the definition of $B''$.
\end{proof}

%

For each $\fp \mid \fP$, recall that the ideal $J_{\fp}$ of $\Z[G]$ defined in \S \ref{ss:ext} is generated by 
$\bpsi(\sigma) - 1$ for $\sigma \in I_{F_{\fp}}$ and $\bpsi(\wtil{\varphi_{\fp}}) - 1 + \nu_{I_{\fp}} \bpsi(\wtil{\varphi_{\fp}})$ for any lift $\wtil{\varphi_{\fp}} \in G_{F_{\fp}}$ of the Frobenius.
For $\fp \mid \fP_{\bad}$, by Lemma \ref{lem:gamma}, the element $- \AC{\fp}$ is annihilated by $J_{\fp}$ in $\ol{B_1}^{\sharp}$.
Therefore, we can define a homomorphism
\[
\gamma: 
\bigoplus_{\fp \mid \fP_{\bad}} A_{\fp}^{\chi^{-1}} \to \ol{B_1}^{\sharp}
\]
by sending $1$ at $\fp \mid \fP_{\ba}$ to $-\AC{\fp}$.
Then $\gamma$ is clearly surjective.

\subsubsection{Construction of $\beta$}\label{sss:b}

Recall that in \S \ref{ss:Gal_rep}, 
we fixed a lift $\tau \in G_F$ of the complex conjugation $c \in \Gal(H/F)$ that satisfies a certain condition. 
We take $\wtil{c} = \tau$ in Lemma \ref{lem:desc_phi} and define a cocycle $\phi$ accordingly.  
We first observe a lemma.

\begin{lem}\label{lem:phi_kappa}
We have
\[
\alpha \circ \phi = \kappa
\]
as maps from $G_F$ to $\ol{B_0}^{\sharp}$.
\end{lem}

\begin{proof}
Recall the class $[\phi] \in H^1(G_F, \Cl_H^{\Sigma', -})$ in Definition \ref{defn:phi_cocycle}.
Then we have
\[
\alpha_* [\phi] = [\kappa] \in H^1(G_F, \ol{B_0}^{\sharp}),
\]
since both elements are sent by the restriction to the same class in $H^1(G_H, \ol{B_0}^{\sharp})^G$.
This means that there exists $\ol{b_0} \in \ol{B_0}$ such that
\[
\alpha(\phi(\sigma)) - \kappa(\sigma) = \bpsi(\sigma)^{-1} \ol{b_0} - \ol{b_0}
\]
for any $\sigma \in G_F$.
Let us put $\sigma = \tau$ in this formula.
Since $\phi(\tau) = 0$, $\kappa(\tau) = 0$, and $\bpsi(\tau) = -1$, we then obtain $0 = -2 \ol{b_0}$, so $\ol{b_0} = 0$.
This completes the proof of the lemma.
\end{proof}

As in Definition \ref{defn:eta}, let
\[
\eta^{\chi^{-1}} \in \Ext^1_{\Z_p[G]^{\chi^{-1}}} \parenth{\bigoplus_{\fp \mid \fP_{\bad}} A_{\fp}^{\chi^{-1}}, \Cl_H^{\Sigma', \chi^{-1}}}
\]
denote the extension class represented by the upper sequence in the diagram in Theorem \ref{thm:abc}.
Let us write $\eta_{\Omega} = \eta^{\chi^{-1}}$.
Similarly, let 
\[
\eta_B \in \Ext^1_{\Z_p[G]^{\chi^{-1}}} \parenth{\ol{B_1}^{\sharp}, \ol{B_0}^{\sharp}}
\]
denote the extension class represented by the lower sequence.

Then, for the construction of $\beta$, it is enough to show the following (cf.~\cite[Theorem 9.8]{DK20}).

\begin{prop}
We have
\[
\alpha_*\eta_{\Omega} = \gamma^* \eta_B
\in \Ext^1_{\Z_p[G]^{\chi^{-1}}} \parenth{\bigoplus_{\fp \mid \fP_{\bad}} A_{\fp}^{\chi^{-1}}, \ol{B_0}^{\sharp}}.
\]
\end{prop}

\begin{proof}
We fix $\fp \mid \fP_{\bad}$ and study the $\fp$-components $\eta_{\Omega, \fp}$ and $(\gamma^* \eta_B)_{\fp}$ of $\eta_{\Omega}$ and $\gamma^* \eta_B$.
We will always work over $\Z_p[G]^{\chi^{-1}}$.
It is convenient to write a commutative diagram over 
\[
\xymatrix{
	\Hom(J_{\fp}^{\chi^{-1}}, \Cl_H^{\Sigma', \chi^{-1}}) \ar@{->>}[r]^{\delta_{\Cl}} \ar[d]_{\alpha_*}
	& \Ext^1(A_{\fp}^{\chi^{-1}}, \Cl_H^{\Sigma', \chi^{-1}}) \ar[d]^{\alpha_*}\\
	\Hom(J_{\fp}^{\chi^{-1}}, \ol{B_0}^{\sharp}) \ar@{->>}[r]_{\delta_{B}}
	& \Ext^1(A_{\fp}^{\chi^{-1}}, \ol{B_0}^{\sharp}),
}
\]
where the horizontal arrows are induced by \eqref{eq:J_v}, so $\delta_{\Cl}$ is nothing but the $\chi^{-1}$-component of $\delta_{\fp}$ in \eqref{eq:Jv_ext}.
By Proposition \ref{prop:ext_desc}, we have an element
\[
\wtil{\eta_{\Omega, \fp}} \in \Hom(J_{\fp}^{\chi^{-1}}, \Cl_H^{\Sigma', \chi^{-1}})
\]
such that $\delta_{\Cl}(\wtil{\eta_{\Omega, \fp}}) = \eta_{\Omega, \fp}$ and that the properties labeled (b) and (c) hold.
Let $\alpha_* \wtil{\eta_{\Omega, \fp}} \in \Hom(J_{\fp}^{\chi^{-1}}, \ol{B_0}^{\sharp})$ be the push of $\wtil{\eta_{\Omega, \fp}}$.
Then, by the above commutative diagram, the following hold.
\begin{itemize}
\item[(a)]
We have $\delta_B(\alpha_* \wtil{\eta_{\Omega, \fp}}) = \alpha_* \eta_{\Omega, \fp}$.
\item[(b)]
For any $\sigma \in I_{F_{\fp}}$, 
we have $(\alpha_* \wtil{\eta_{\Omega, \fp}})(\bpsi(\sigma) - 1) = \alpha(\phi(\sigma))$.
\item[(c)]
For any lift $\wtil{\varphi_{\fp}} \in G_{F_{\fp}}$ of $\varphi_{\fp}$, 
we have $(\alpha_* \wtil{\eta_{\Omega, \fp}})(\bpsi(\wtil{\varphi_{\fp}}) - 1 + \nu_{I_{\fp}} \bpsi(\wtil{\varphi_{\fp}})) = \alpha(\phi(\wtil{\varphi_{\fp}}))$.
\end{itemize}

On the other hand, the element $(\gamma^*\eta_B)_{\fp} \in \Ext^1\parenth{A_{\fp}^{\chi^{-1}}, \ol{B_0}^{\sharp}}$ is described as follows.
Let us consider the homomorphism $\Z_p[G]^{\chi^{-1}} \to \ol{B_p}^{\sharp}$ which sends $1$ to $- \AC{\fp}$, which induces a homomorphism $(\wtil{\gamma^*\eta_B})_{\fp}: J_{\fp}^{\chi^{-1}} \to \ol{B_0}^{\sharp}$.
Then the following hold.
\begin{itemize}
\item[(a')]
We have $\delta_{B}((\wtil{\gamma^*\eta_B})_{\fp}) = (\gamma^*\eta_B)_{\fp}$ by tracing the maps involved in the snake lemma.
\item[(b')]
For $\sigma \in I_{F_{\fp}}$,
we have $(\wtil{\gamma^*\eta_B})_{\fp}(\bpsi(\sigma) - 1) = \kappa(\sigma)$ by Lemma \ref{lem:gamma}(a).
\item[(c')]
For any lift $\wtil{\varphi_{\fp}} \in G_{F_{\fp}}$ of $\varphi_{\fp}$, 
we have $(\wtil{\gamma^*\eta_B})_{\fp}(\bpsi(\wtil{\varphi_{\fp}}) - 1 + \nu_{I_{\fp}} \bpsi(\wtil{\varphi_{\fp}})) = \kappa(\wtil{\varphi_{\fp}})$ by Lemma \ref{lem:gamma}(b).
\end{itemize}

By Lemma \ref{lem:phi_kappa} and comparing the properties (b)(c) with (b')(c'), we have
\[
\alpha_* \wtil{\eta_{\Omega, \fp}} = (\wtil{\gamma^*\eta_B})_{\fp}.
\]
Therefore, by the properties (a) and (a'), we obtain
\[
\alpha_* \eta_{\Omega, \fp} = (\gamma^*\eta_B)_{\fp},
\]
as desired.
\end{proof}

\subsection{Computation of the Fitting ideal}\label{ss:Fitt_B} 

In this subsection, we prove Theorem \ref{thm:Fitt_B}.

We put $r = \# \prim(\fP') \geq 0$ and label the elements as $\prim(\fP') = \{\fp'_1, \dots, \fp'_r\}$.
For each $1 \leq j \leq r$, 
let us fix $\sigma_j \in G_{F_{\fp'_j}}$ which is a lift of the arithmetic Frobenius.
Put
\[
c_j = \bpsi(\sigma_j) \varepsilon_{\cyc}^{1 - k}(\sigma_j^{-1}) b(\sigma_j^{-1}) \in B_0 \subset B.
\] 
By Lemma \ref{lem:observ}(3), there exists an element $x_j \in I'$ such that 
\begin{equation}\label{eq:c_desc}
c_j = (U_{\fp'_j} - \bpsi(\sigma_j) + x_j) \AC{\fp'_j}.
\end{equation}
We also put $s = \# \prim(\fP_{\bad})$ and label $\prim(\fP_{\bad}) = \{ \fp_1, \dots, \fp_s\}$.

Let $b_1, \dots, b_n \in B$ be a set of generator of $B$ as a $\Rc$-module such that
$b_1, \dots, b_n$ are non-zero-divisor in $\bK_E$ (we may assume this as in \cite[Lemma 9.9]{DK20}).
We use the classes of 
\[
c_1, \dots, c_r, b_1, \dots, b_n, \AC{\fp_1}, \dots, \AC{\fp_s}
\]
as a generator of $\ol{B_p}$ over $\Rc$.
Compared to the work \cite[\S 9.5]{DK20}, the additional elements $\AC{\fp_l}$ for $1 \leq l \leq s$ are a novel ingredient.
We introduce them in order to manage the contribution of the additional module $B''$ in the definition of $\ol{B_p}$.
They make the computation even harder.

Let $M \in M_{r + n + s}(\Rc)$ be any square matrix which is a relation of the generators in $\ol{B_p}$, that is, if we write
\[
M = (W \mid Z \mid U) = 
\parenth{(w_{ij})_{i, 1 \leq j \leq r} \mid (z_{ik})_{i, 1 \leq k \leq n} \mid (u_{il})_{i, 1 \leq l \leq s} },
\]
we have
\begin{equation}\label{eq:Fitt_cond}
\sum_{j=1}^r w_{ij} c_j + \sum_{k=1}^n z_{ik} b_k + \sum_{l=1}^s u_{il} \AC{\fp_l}
\in I'B + \theta^{\sharp} B + B' + B''
\end{equation}
for each $1 \leq i \leq r + n + s$.
Since the Fitting ideal $\Fitt_{\Rc}(\ol{B_p})$ is generated by $\det(M) \in \Rc$ for $M$ with this property, 
in order to prove Theorem \ref{thm:Fitt_B}, it is enough to show
\begin{equation}\label{eq:Fitt_M}
\det(M) \in (\theta^{\sharp}, \delta_{\fP_{\ba}' = 1} \nu_{I_p}^{\chi}).
\end{equation}

\begin{lem}\label{lem:t}
There exists a family of elements $y_{ik}, \alpha_{il} \in \Rc$ such that, if we define an element $t \in \wtil{\bT}$ by
\[
t = \prod_{j = 1}^r (U_{\fp'_j}- \bpsi(\fp_j')) 
\cdot \det \parenth{(w_{ij})_{i, j} 
\mid (z_{ik} - \theta^{\sharp} y_{ik})_{i, k} 
\mid \parenth{u_{il} - \alpha_{il} (1 - \nu_{I_{\fp_l}} \bpsi({\fp_l})^{-1} - U_{\fp_l}^{-1})}_{i, l} },
\]
then we have
\[
t F_k(\bpsi) \equiv 0 \pmod{ (\wtil{x} \theta^{\sharp})}.
\]
\end{lem}

\begin{proof}
We have $B' \subset \sum_{j = 1}^r \AC{\fp'_j} I'$ since, for any $\fp \mid \fP'$ and $\sigma \in I_{F_{\fp}}$, Lemma \ref{lem:observ}(3) implies $b(\sigma) \in \AC{\fp} I'$. 
Therefore, by \eqref{eq:Fitt_cond}, for each $i$, we have
\begin{align}
& \sum_{j=1}^r w_{ij} c_j + \sum_{k=1}^n z_{ik} b_k + \sum_{l=1}^s u_{il} \AC{\fp_l}\\
& = \sum_{j = 1}^r \AC{\fp'_j} x_{ij}'' 
+ \sum_{k = 1}^n (x_{ik}' + \theta^{\sharp} y_{ik}) b_k 
+ \sum_{l = 1}^s t_{il} (1 - \nu_{I_{\fp_l}} \bpsi(\fp_l)^{-1} - U_{\fp_l}^{-1}) \AC{\fp_l}
\end{align}
for some $x_{ij}'' \in I'$, $x_{ik}' \in I'$, $y_{ik} \in \Rc$, and $t_{il} \in \bT'$.
Using \eqref{eq:c_desc}, we rewrite this formula as
\begin{align}\label{eq:Fitt_a1}
& \sum_{j=1}^r \parenth{w_{ij} (U_{\fp'_j}- \bpsi(\fp_j') + x_j) - x_{ij}''}\AC{\fp'_j} \\
& + \sum_{k=1}^n \parenth{z_{ik} - x_{ik}' - \theta^{\sharp} y_{ik}} b_k 
 + \sum_{l=1}^s \parenth{u_{il} - t_{il} (1 - \nu_{I_{\fp_l}} \bpsi({\fp_l})^{-1} - U_{\fp_l}^{-1})} \AC{\fp_l}
 = 0.
\end{align}

We define a matrix $M' \in M_{r + n + s}(\wtil{\bT})$ by
\[
M' = 
(W' \mid Z' \mid U')
\]
with
\begin{align}
W' &= \parenth{w_{ij} (U_{\fp'_j}- \bpsi(\fp_j') + x_j) - x_{ij}''}_{i, 1 \leq j \leq r},\\
Z' &= \parenth{z_{ik} - x_{ik}' - \theta^{\sharp} y_{ik}}_{i, 1 \leq k \leq n},\\
U' &=  \parenth{u_{il} - t_{il} (1 - \nu_{I_{\fp_l}} \bpsi({\fp_l})^{-1} - U_{\fp_l}^{-1})}_{i, 1 \leq l \leq s}.
\end{align}
Then by the formula \eqref{eq:Fitt_a1}, since $b_1, \dots, b_n$ are non-zero-divisors, we have $\det(M') = 0$.
In particular, we have $\det(\varphi(M')) = \varphi(\det(M')) = 0$ in $W$.
Since $x_j, x_{ik}', x_{ij}''$ are sent to zero by $\varphi$, we have
\begin{align}
\varphi(M') = \parenth{\parenth{w_{ij} (\varphi(U_{\fp'_j})- \bpsi(\fp_j'))}_{i, j}
	\mid \parenth{z_{ik} - \theta^{\sharp} y_{ik}}_{i, k}
	\mid \parenth{u_{il} - \varphi(t_{il}) (1 - \nu_{I_{\fp_l}} \bpsi({\fp_l})^{-1} - \varphi(U_{\fp_l})^{-1})
}_{i, l}}.
\end{align}

Since the structure morphism $\Rc \to \bT'/I'$ is surjective, we can choose an element $\alpha_{il} \in \Rc$ such that $t_{il} - \alpha_{il} \in I'$.
We then define an element $t \in \wtil{\bT}$ by the formula in the statement.
Then we have $\varphi(t) = \det(\varphi(M')) = 0$.
By the definition of $\varphi$, we have $\varphi(t) = 0$ in $W$ if and only if $t F_k(\bpsi) \equiv 0 \pmod{ (\wtil{x} \theta^{\sharp})}$, so the lemma holds.
\end{proof}

We keep the notation in Lemma \ref{lem:t}. 
For simplicity, we write
\begin{equation}\label{eq:defn_Z}
\cZ = \parenth{(w_{ij})_{i, j} \mid (z_{ik} - \theta^{\sharp} y_{ik})_{i, k} },
\end{equation}
so
\[
t = \prod_{j = 1}^r (U_{\fp'_j}- \bpsi(\fp_j')) 
\cdot \det \parenth{\cZ \mid (u_{il} - \alpha_{il} (1 - \nu_{I_{\fp_l}} \bpsi({\fp_l})^{-1} - U_{\fp_l}^{-1}))_{i, l} }.
\]
%

In order to utilize Lemma \ref{lem:t}, we have to investigate $t F_k(\bpsi)$.
By \eqref{eq:Fk}, we are led to study $t \ol{W}_1^{H/F}(\bpsi_{\fn}, \btr)^{\chi}$ and $t \nu_{\fP} W_k^{H^{\fP}/F}(\bpsi_{\fn/\fP}, \btr_p)^{\chi}$ in case $\fP_{\ba} = \fP$.
The former will be studied in Proposition \ref{prop:tW1}, using Lemma \ref{lem:tW1_coeff}.
Actually, we need only the Fourier coefficients at square-free ideals $\fa$ with $\fa \mid \fP_{\ba}$.
On the other hand, as written in \S \ref{ss:outline}, we cannot manage the contribution of the latter completely, so we will deal with only when the latter vanishes in some sense.
The precise statement is the following.

\begin{lem}\label{lem:tW1}
Let us take $t \in \wtil{\bT}$ as in Lemma \ref{lem:t}.
Then we have
\begin{equation}\label{eq:tW1_cong}
\coef{\fa}{t \ol{W}_1^{H/F}(\bpsi_{\fn}, 1)^{\chi}} \equiv 0 \pmod{(\theta^{\sharp}, \delta_{\fP_{\ba}' = 1} \nu_{I_p}^{\chi})}
\end{equation}
for any ideal $\fa$.
\end{lem}

\begin{proof}
By $t F_k(\bpsi) \equiv 0 \pmod{ (\wtil{x} \theta^{\sharp})}$ and \eqref{eq:Fk}, we have
\begin{equation}\label{eq:tFk2}
\wtil{x} t \ol{W}_1^{H/F}(\bpsi_{\fn}, \btr)^{\chi}
- \delta_{\fP_{\ba} = \fP} \nu_{\fP} t W_k^{H^{\fP}/F}(\bpsi_{\fn/\fP}, \btr_p)^{\chi} \equiv 0
 \pmod{(\wtil{x} \theta^{\sharp})}.
\end{equation}
By Proposition \ref{prop:Hecke_W}(2), 
 $W_k^{H^{\fP}/F}(\bpsi_{\fn/\fP}, \btr_p)^{\chi}$ is annihilated by $U_{\fp_j'} - \bpsi(\fp_j')$ for $1 \leq j \leq r$ (recall that these operators appear in the definition of $t$).
Therefore, the contribution of $W_k^{H^{\fP}/F}(\bpsi_{\fn/\fP}, \btr_p)^{\chi}$ vanishes unless $\fP_{\ba} = \fP$ and $r = 0$. 
The last two conditions simultaneously hold if and only if $\fP_{\ba}' = 1$.
Even if $\fP_{\ba}' = 1$, $\nu_{\fP}$ is divisible by $\nu_{I_p}$, and also recall that the non-constant terms of $t W_k^{H^{\fP}/F}(\bpsi_{\fn/\fP}, \btr_p)^{\chi}$ are all integral.
Therefore, in any case, we have
\begin{equation}
\coef{\fa}{\wtil{x} t \ol{W}_1^{H/F}(\bpsi_{\fn}, \btr)^{\chi}} \equiv 0
 \pmod{(\wtil{x} \theta^{\sharp}, \delta_{\fP_{\ba}' = 1} \nu_{I_p}^{\chi})}.
\end{equation}
The lemma follows from this.
\end{proof}

\begin{lem}\label{lem:tW1_coeff}
For a square-free ideal $\fa \mid \fP_{\bad}$, the following are true.
\begin{itemize}
\item[(a)]
We have
\[
\coef{\fa}{\ol{W}_1^{H/F}(\bpsi_{\fn \fP'}, \btr)^{\chi}}
= \prod_{\fp \mid \fP_{\ba}} (1 + \nu_{I_{\fp}}(1 + \delta_{\fp \mid \fa} \bpsi(\fp))).
\]
\item[(b)]
For a divisor $\fP_0 \| \fP_{\ba}$, we have
\begin{align}
& \coef{\fa}{\prod_{\fp \mid \fP_0}(1 - \nu_{I_{\fp}} \bpsi(\fp)^{-1} - U_{\fp}^{-1}) \ol{W}_1^{H/F}(\bpsi_{\fn \fP'}, \btr)^{\chi}}\\
& = \coef{\fa}{\ol{W}_1^{H/F}(\bpsi_{\fn \fP'}, \btr)^{\chi}}
\times \prod_{\fp \mid \fP_0} C_{\fp}^{\fa},
\end{align}
where, for $\fp \mid \fP_{\bad}$, we put
\[
C_{\fp}^{\fa} = \nu_{I_{\fp}} (1 - \# I_{\fp} \bpsi(\fp)^{-1} - \bpsi(\fp)^{-1})
	 \frac{1 + \delta_{\fp \mid \fa} \bpsi(\fp)}{1 + \# {I_{\fp}}(1 + \delta_{\fp \mid \fa} \bpsi(\fp))}.
\]
\end{itemize}
\end{lem}

\begin{proof}
Recall that 
\[
\ol{W}_1^{H/F}(\bpsi_{\fn \fP'}, \btr)^{\chi}
= \sum_{\fQ \| \fP_{\ba}} \nu_{\fQ} W_1^{H^{\fQ}/F}(\bpsi_{\fn \fP' / \fQ}, \btr)^{\chi}.
\]
(a)
We decompose $\fP_{\ba} = \fP_1 \fP_2$ so that 
\[
\prim(\fP_1) \subset \prim(\fa),
\qquad \prim(\fP_2) \cap \prim(\fa) = \emptyset.
\]
For each $\fQ \| \fP_{\ba}$, we decompose $\fQ = \fQ_1 \fQ_2$ correspondingly.
We first compute $\coef{\fa}{W_1^{H^{\fQ}/F}(\bpsi_{\fn \fP'/\fQ}, \btr)^{\chi}}$ by the formula in Definition \ref{defn:Wk}.
Since $\fa \mid p$, by the formula of the level raising operators, only $\fm = 1$ contributes to the sum, so
\begin{align}
\coef{\fa}{W_1^{H^{\fQ}/F}(\bpsi_{\fn \fP'/\fQ}, \btr)^{\chi}}
& = \coef{\fa}{E_1^{H^{\fQ}/F}(\bpsi_{\fn \fP'/ \fQ}, \btr)^{\chi}}
 = \sum_{\fr \mid \fa} \bpsi_{\fn \fP'/\fQ} \parenth{\frac{\fa}{\fr}}\\
& = \sum_{\fr \mid (\fa, \fQ)} \bpsi(\fr)
 = \prod_{\fp \mid \fQ_1} (1 + \bpsi(\fp)).
\end{align}
By taking the sum with respect to $\fQ \| \fP_{\ba}$, we obtain
\begin{align}
\coef{\fa}{\ol{W}_1^{H/F}(\bpsi_{\fn \fP'}, \btr)^{\chi}}
& = \sum_{\fQ \| \fP_{\ba}} 
	\nu_{\fQ}  \coef{\fa}{W_1^{H^{\fQ}/F}(\bpsi_{\fn \fP' / \fQ}, \btr)^{\chi}}\\
& = \sum_{\fQ_1 \| \fP_1} \sum_{\fQ_2 \| \fP_2} 
	\nu_{\fQ_1} \nu_{\fQ_2}
	\prod_{\fp \mid \fQ_1} (1 + \bpsi(\fp))\\
& = \prod_{\fp \mid \fP_1} (1 + \nu_{I_{\fp}}(1 + \bpsi(\fp)))
	\prod_{\fp \mid \fP_2} (1 + \nu_{I_{\fp}}).
\end{align}

(b)
By Proposition \ref{prop:Hecke_W}(2), for each $\fp \mid \fP_0$ and $\fQ \| \fP_{\ba}$, we have
\[
U_{\fp} W_1^{H^{\fQ}/F}(\bpsi_{\fn \fP' / \fQ}, \btr)^{\chi}
= \bpsi_{\fn \fP' / \fQ}(\fp) W_1^{H^{\fQ}/F}(\bpsi_{\fn \fP' / \fQ}, \btr)^{\chi}
 + W_1^{H^{\fQ}/F}(\bpsi_{\fp \fn \fP' / \fQ}, \btr)^{\chi}
\]
and
\[
U_{\fp} W_1^{H^{\fQ}/F}(\bpsi_{\fp \fn \fP' / \fQ}, \btr)^{\chi}
=W_1^{H^{\fQ}/F}(\bpsi_{\fp \fn \fP' / \fQ}, \btr)^{\chi}.
\]
Note that $\bpsi_{\fn \fP' / \fQ}(\fp) = \delta_{\fp \mid \fQ} \bpsi(\fp)$.
Then these formulas imply
\[
U_{\fp}^{-1} W_1^{H^{\fQ}/F}(\bpsi_{\fn \fP' / \fQ}, \btr)^{\chi}
= \begin{cases}
	W_1^{H^{\fQ}/F}(\bpsi_{\fn \fP' / \fQ}, \btr)^{\chi}
	& (\fp \nmid \fQ)\\
	\bpsi(\fp)^{-1} 
	\parenth{W_1^{H^{\fQ}/F}(\bpsi_{\fn \fP' / \fQ}, \btr)^{\chi}
		- W_1^{H^{\fQ}/F}(\bpsi_{\fp \fn \fP' / \fQ}, \btr)^{\chi}
	}
	& (\fp \mid \fQ).
\end{cases}
\]
Then, for each $\fp \mid \fP_0$, we have
\begin{align}
& (1 - \nu_{I_{\fp}} \bpsi(\fp)^{-1} - U_{\fp}^{-1})
	\sum_{\fQ \| \fP_{\ba}} \nu_{\fQ} W_1^{H^{\fQ}/F}(\bpsi_{\fn \fP' / \fQ}, \btr)^{\chi}\\
& = \sum_{\fQ \| \fP_{\ba}, \fp \mid \fQ} \nu_{\fQ} 
	\parenth{(1 - \nu_{I_{\fp}} \bpsi(\fp)^{-1} - \bpsi(\fp)^{-1}) W_1^{H^{\fQ}/F}(\bpsi_{\fn \fP' / \fQ}, \btr)^{\chi}
  		+ \bpsi(\fp)^{-1} W_1^{H^{\fQ}/F}(\bpsi_{\fp \fn \fP' / \fQ}, \btr)^{\chi}}\\
& \qquad + \sum_{\fQ \| \fP_{\ba}, \fp \nmid \fQ} \nu_{\fQ} 
	\parenth{ - \nu_{I_{\fp}} \bpsi(\fp)^{-1}  W_1^{H^{\fQ}/F}(\bpsi_{\fn \fP' / \fQ}, \btr)^{\chi}
	}.
\end{align}
Several terms in this expression cancel each other, and we obtain
\begin{align}
& = \sum_{\fQ \| \fP_{\ba}, \fp \mid \fQ} \nu_{\fQ} 
	(1 - \nu_{I_{\fp}} \bpsi(\fp)^{-1} - \bpsi(\fp)^{-1}) W_1^{H^{\fQ}/F}(\bpsi_{\fn \fP' / \fQ}, \btr)^{\chi}\\
& = \nu_{I_{\fp}} (1 - \# I_{\fp} \bpsi(\fp)^{-1} - \bpsi(\fp)^{-1})
	\sum_{\fQ \| \fP_{\ba}, \fp \nmid \fQ} \nu_{\fQ} 
	W_1^{H^{\fp^{d_{\fp}} \fQ}/F}(\bpsi_{\fn \fP' / \fp^{d_{\fp}} \fQ}, \btr)^{\chi}.
\end{align}
Here, we write $d_{\fp} = \ord_{\fp}(\fP)$, so  $\fp^{d_{\fp}} \fQ \| \fP$ for $\fQ \| \fP$ with $\fp \nmid \fQ$.

By similar computation, we can inductively show
\begin{align}
& \parenth{\prod_{\fp \mid \fP_0} (1 - \nu_{I_{\fp}} \bpsi(\fp)^{-1} - U_{\fp}^{-1})}
	\ol{W}_1^{H/F}(\bpsi_{\fn \fP'}, \btr)^{\chi}\\
& = \parenth{\prod_{\fp \mid \fP_0} \nu_{I_{\fp}} (1 - \# I_{\fp} \bpsi(\fp)^{-1} - \bpsi(\fp)^{-1})}
	\sum_{\fQ  \| \fP_{\ba}, (\fQ, \fP_0) = 1} \nu_{\fQ} 
	W_1^{H^{\fP_0 \fQ}/F}(\bpsi_{\fn \fP' / \fP_0 \fQ}, \btr)^{\chi}.
\end{align}

Let us decompose $\fP_{\ba} = \fP_0 \fP_1 \fP_2$ so that
\[
\prim(\fP_1) \subset \prim(\fa), \qquad \prim(\fP_2) \cap \prim(\fa) = \emptyset.
\]
For each $\fQ \| \fP_{\ba}$ with $(\fQ, \fP_0) = 1$, we decompose $\fQ = \fQ_1 \fQ_2$ correspondingly.
Then as in (a) we have
\begin{align}
\coef{\fa}{W_1^{H^{\fP_0\fQ}/F}(\bpsi_{\fn \fP' / \fP_0 \fQ}, \btr)^{\chi}}
& = \prod_{\fp \mid \fa, \fp \mid \fP_0 \fQ} (1 + \bpsi(\fp))\\
& = \prod_{\fp \mid \fP_0, \fp \mid \fa} (1 + \bpsi(\fp))
	\prod_{\fp \mid \fQ_1} (1 + \bpsi(\fp)).
\end{align}
Therefore,
\begin{align}
& \coef{\fa}{\parenth{\prod_{\fp \mid \fP_0} \nu_{I_{\fp}}} \sum_{\fQ  \| \fP_{\ba}, (\fQ, \fP_0) = 1} \nu_{\fQ} 
	W_1^{H^{\fP_0 \fQ}/F}(\bpsi_{\fn \fP' / \fP_0 \fQ}, \btr)^{\chi}}\\
& = 	\prod_{\fp \mid \fP_0} \nu_{I_{\fp}} \sum_{\fQ \| \fP_{\ba}, (\fQ, \fP_0) = 1}
	\nu_{\fQ}
	 \prod_{\fp \mid \fP_0, \fp \mid \fa} (1 + \bpsi(\fp))
	\prod_{\fp \mid \fQ_1} (1 + \bpsi(\fp))\\
& =	\prod_{\fp \mid \fP_0, \fp \nmid \fa} \nu_{I_{\fp}} 
	\prod_{\fp \mid \fP_0, \fp \mid \fa} \nu_{I_{\fp}} (1 + \bpsi(\fp))
	\sum_{\fQ_1 \| \fP_1} \sum_{\fQ_2 \| \fP_2}
	\nu_{\fQ_1} \nu_{\fQ_2}	
	\prod_{\fp \mid \fQ_1} (1 + \bpsi(\fp))\\
& = 	\prod_{\fp \mid \fP_0, \fp \nmid \fa} \nu_{I_{\fp}} 
	\prod_{\fp \mid \fP_0, \fp \mid \fa} \nu_{I_{\fp}} (1 + \bpsi(\fp))
	\prod_{\fp \mid \fP_1} (1 + \nu_{I_{\fp}}(1 + \bpsi(\fp)))
	\prod_{\fp \mid \fP_2} (1 + \nu_{I_{\fp}})\\
& =  \coef{\fa}{\ol{W}_1^{H/F}(\bpsi_{\fn \fP'}, \btr)^{\chi}}
	\times \prod_{\fp \mid \fP_0, \fp \mid \fa} \nu_{I_{\fp}} \frac{1 + \bpsi(\fp)}{1 + \#{I_{\fp}}(1 + \bpsi(\fp))}
	\times \prod_{\fp \mid \fP_0, \fp \nmid \fa} \nu_{I_{\fp}} \frac{1}{1 + \#{I_{\fp}}}.
\end{align}
The last equality follows from (a).
Thus we have the claim (b).
\end{proof}

\begin{prop}\label{prop:tW1}
For a square-free ideal $\fa \mid \fP_{\bad}$, we have
\[
\coef{\fa}{t \ol{W}_1^{H/F}(\bpsi_{\fn}, \btr)^{\chi}}
= \det \parenth{ \cZ \mid (v^{\fa}_{il})_{i, l}},
\]
where we put
\[
v^{\fa}_{il} = 
	(1 + \nu_{I_{\fp_l}}(1 + \delta_{\fp_l \mid \fa} \bpsi(\fp_l)))u_{il} 
	- \alpha_{il} \cdot \nu_{I_{\fp_l}} (1 - \# I_{\fp_l} \bpsi(\fp_l)^{-1} - \bpsi(\fp_l)^{-1})
	 \cdot (1 + \delta_{\fp_l \mid \fa} \bpsi(\fp_l)).
\]
\end{prop}

\begin{proof}
By Proposition \ref{prop:Hecke_olW}(2), we have
\[
\parenth{\prod_{\fp' \mid \fP'} (U_{\fp'} - \bpsi(\fp'))}
\ol{W_1}^{H/F}(\bpsi_{\fn}, \btr)^{\chi}
= \ol{W_1}^{H/F}(\bpsi_{\fn \fP'}, \btr)^{\chi},
\]
so
\begin{equation}\label{eq:tW1_det}
t \ol{W_1}^{H/F}(\bpsi_{\fn}, \btr)^{\chi} = 
\det \parenth{\cZ \mid (u_{il} - \alpha_{il} (1 - \nu_{I_{\fp_l}} \bpsi({\fp_l})^{-1} - U_{\fp_l}^{-1}))_{i, l} } 
	\cdot \ol{W_1}^{H/F}(\bpsi_{\fn \fP'}, \btr)^{\chi}.
\end{equation}

By Lemma \ref{lem:tW1_coeff}(b) and considering the expansion of the determinant, we have
\[
\coef{\fa}{t \ol{W}_1^{H/F}(\bpsi_{\fn}, \btr)^{\chi}}
= \coef{\fa}{\ol{W}_1^{H/F}(\bpsi_{\fn \fP'}, 1)}
\times \det(\cZ \mid (u_{il} - \alpha_{il} C_{\fp_l}^{\fa})_{i, l}).
\]
By Lemma \ref{lem:tW1_coeff}(a), this is equal to
\[
\det(\cZ \mid (1 + \nu_{I_{\fp_l}}(1 + \delta_{\fp_l \mid \fa} \bpsi(\fp_l)))(u_{il} - \alpha_{il} C_{\fp_l}^{\fa}))_{i, l}).
\]
Now it is easy to see 
\[
(1 + \nu_{I_{\fp_l}}(1 + \delta_{\fp_l \mid \fa} \bpsi(\fp_l)))(u_{il} - \alpha_{il} C_{\fp_l}^{\fa})
 = v^{\fa}_{il}.
\]
Thus we get the proposition.
\end{proof}

Now we are ready to finish the proof of Theorem \ref{thm:Fitt_B}.

\begin{proof}[Proof of Theorem \ref{thm:Fitt_B}]
Recall that our goal is to show \eqref{eq:Fitt_M}, which is by \eqref{eq:defn_Z} equivalent to
\begin{equation}\label{eq:Fitt_goal}
\det \parenth{\cZ \mid (u_{il})_{i, l}} \equiv 0 
\pmod{(\theta^{\sharp}, \delta_{\fP_{\ba}' = 1} \nu_{I_p}^{\chi})}.
\end{equation}

For each square-free ideal $\fa \mid \fP_{\bad}$, let us consider $v_{il}^{\fa} \in \Z_p[G]^{\chi}$ defined in Proposition \ref{prop:tW1}.
For any $0 \leq l' \leq l$ and any $\fa \mid \fp_{l'+1} \dots \fp_l$, we shall show
\begin{equation}\label{eq:det_1}
\det \parenth{ \cZ \mid \parenth{u_{i, 1}, \cdots, u_{i, l'}, v^{\fa}_{i, l'+1}, \cdots, v^{\fa}_{i, l}}_{i}} \equiv 0 \pmod{ (\theta^{\sharp}, \delta_{\fP_{\ba}' = 1} \nu_{I_p}^{\chi})}
\end{equation}
by induction on $l'$.
When $l' = 0$, the claim \eqref{eq:det_1} follows immediately from Lemma \ref{lem:tW1} and Proposition \ref{prop:tW1}.
When $l' = l$, the claim \eqref{eq:det_1} is nothing but the goal \eqref{eq:Fitt_goal}.

Let $1 \leq l' \leq l$ and let $\fa \mid \fp_{l' + 1} \dots \fp_l$.
We choose a lift $\wtil{\bpsi(\fp_{l'})} \in \Gal(H/F)$ of $\bpsi(\fp_{l'}) \in \Gal(H^{\fp_{l'}}/F)$.
Then by the induction hypothesis, we have
\begin{align}
& (1 + \wtil{\bpsi(\fp_{l'})}) 
\det \parenth{ \cZ \mid \parenth{u_{i, 1}, \cdots, u_{i, l' - 1}, v^{\fa}_{i, l'}, v^{\fa}_{i, l'+1}, \cdots, v^{\fa}_{i, l}}_{i}} \\
& - \det \parenth{ \cZ \mid \parenth{u_{i, 1}, \cdots, u_{i, l' - 1}, v^{\fp_{l'} \fa}_{i, l'}, v^{\fp_{l'} \fa}_{i, l'+1}, \cdots, v^{\fp_{l'} \fa}_{i, l}}_{i}} 
\equiv 0 \pmod{ (\theta^{\sharp}, \delta_{\fP_{\ba}' = 1} \nu_{I_p}^{\chi})}.
\end{align}
By direct computation, we have
\[
(1 + \wtil{\bpsi(\fp_{l'})}) v^{\fa}_{i, l'} - v^{\fp_{l'} \fa}_{i, l'}
= \wtil{\bpsi(\fp_{l'})} u_{i, l'}.
\]
Therefore, the left hand side is 
\[
\wtil{\bpsi(\fp_{l'})} \cdot
\det \parenth{ \cZ \mid \parenth{u_{i, 1}, \cdots, u_{i, l' - 1}, u_{i, l'}, v^{\fa}_{i, l'+1}, \cdots, v^{\fa}_{i, l}}_{i}}.
\]
This proves \eqref{eq:det_1} by induction.

Thus we have proved \eqref{eq:det_1}, so in particular \eqref{eq:Fitt_goal}.
This completes the proof of Theorem \ref{thm:Fitt_B}.
\end{proof}

This completes the proof of Theorem \ref{thm:current_main_p}.
By the discussion in \S \ref{ss:pf_thm}, this also completes the proof of Theorems \ref{thm:current_main} and \ref{thm:main_p_2}.

\bibliographystyle{abbrv}
\bibliography{biblio}

\begin{thebibliography}{10}

\bibitem{AK21}
M.~Atsuta and T.~Kataoka.
\newblock Fitting ideals of class groups for {CM} abelian extensions.
\newblock {\em preprint, arXiv:2104.14765}, 2021.

\bibitem{Ble06}
W.~Bley.
\newblock Equivariant {T}amagawa number conjecture for abelian extensions of a
  quadratic imaginary field.
\newblock {\em Doc. Math.}, 11:73--118, 2006.

\bibitem{BBDS}
D.~Bullach, D.~Burns, A.~Daoud, and S.~Seo.
\newblock Dirichlet {$L$}-series at $s = 0$ and the scarcity of {E}uler
  systems.
\newblock {\em preprint, arXiv:2111.14689}, 2021.

\bibitem{Bur11}
D.~Burns.
\newblock Congruences between derivatives of geometric {$L$}-functions.
\newblock {\em Invent. Math.}, 184(2):221--256, 2011.
\newblock With an appendix by Burns, King Fai Lai and Ki-Seng Tan.

\bibitem{BuFl01}
D.~Burns and M.~Flach.
\newblock Tamagawa numbers for motives with (non-commutative) coefficients.
\newblock {\em Doc. Math.}, 6:501--570, 2001.

\bibitem{BuGr03}
D.~Burns and C.~Greither.
\newblock On the equivariant {T}amagawa number conjecture for {T}ate motives.
\newblock {\em Invent. Math.}, 153(2):303--359, 2003.

\bibitem{BKS16}
D.~Burns, M.~Kurihara, and T.~Sano.
\newblock On zeta elements for {$\Bbb G_m$}.
\newblock {\em Doc. Math.}, 21:555--626, 2016.

\bibitem{BKS17}
D.~Burns, M.~Kurihara, and T.~Sano.
\newblock On {I}wasawa theory, zeta elements for {$\Bbb G_m$}, and the
  equivariant {T}amagawa number conjecture.
\newblock {\em Algebra {\&} Number Theory}, 11(7):1527--1571, 2017.

\bibitem{CN79}
P.~Cassou-Nogu\`es.
\newblock Valeurs aux entiers n\'{e}gatifs des fonctions z\^{e}ta et fonctions
  z\^{e}ta {$p$}-adiques.
\newblock {\em Invent. Math.}, 51(1):29--59, 1979.

\bibitem{DDP11}
S.~Dasgupta, H.~Darmon, and R.~Pollack.
\newblock Hilbert modular forms and the {G}ross-{S}tark conjecture.
\newblock {\em Ann. of Math. (2)}, 174(1):439--484, 2011.

\bibitem{DK20}
S.~Dasgupta and M.~Kakde.
\newblock On the {B}rumer-{S}tark conjecture.
\newblock {\em preprint, arXiv:2010.00657}, 2020.

\bibitem{DK20b}
S.~Dasgupta and M.~Kakde.
\newblock On {C}onstant {T}erms of {E}isenstein {S}eries.
\newblock {\em Acta Arithmetica}, 200:119--147, 2021.

\bibitem{DKV18}
S.~Dasgupta, M.~Kakde, and K.~Ventullo.
\newblock On the {G}ross-{S}tark conjecture.
\newblock {\em Ann. of Math. (2)}, 188(3):833--870, 2018.

\bibitem{DR80}
P.~Deligne and K.~A. Ribet.
\newblock Values of abelian {$L$}-functions at negative integers over totally
  real fields.
\newblock {\em Invent. Math.}, 59(3):227--286, 1980.

\bibitem{Flach11}
M.~Flach.
\newblock On the cyclotomic main conjecture for the prime 2.
\newblock {\em J. Reine Angew. Math.}, 661:1--36, 2011.

\bibitem{Grei00}
C.~Greither.
\newblock Some cases of {B}rumer's conjecture for abelian {CM} extensions of
  totally real fields.
\newblock {\em Math. Z.}, 233(3):515--534, 2000.

\bibitem{Gre07}
C.~Greither.
\newblock Determining {F}itting ideals of minus class groups via the
  equivariant {T}amagawa number conjecture.
\newblock {\em Compos. Math.}, 143(6):1399--1426, 2007.

\bibitem{GKK_09}
C.~Greither, T.~Kataoka, and M.~Kurihara.
\newblock Fitting ideals of $p$-ramified {I}wasawa modules over totally real
  fields.
\newblock {\em to appear in Selecta Mathematica}.

\bibitem{GP15}
C.~Greither and C.~D. Popescu.
\newblock An equivariant main conjecture in {I}wasawa theory and applications.
\newblock {\em J. Algebraic Geom.}, 24(4):629--692, 2015.

\bibitem{Gru76}
K.~W. Gruenberg.
\newblock {\em Relation modules of finite groups}.
\newblock American Mathematical Society, Providence, R.I., 1976.
\newblock Conference Board of the Mathematical Sciences Regional Conference
  Series in Mathematics, No. 25.

\bibitem{JN20}
H.~Johnston and A.~Nickel.
\newblock An unconditional proof of the abelian equivariant {I}wasawa main
  conjecture and applications.
\newblock {\em preprint, arXiv:2010.03186}, 2020.

\bibitem{Kata_05}
T.~Kataoka.
\newblock Fitting invariants in equivariant {I}wasawa theory.
\newblock {\em Development of Iwasawa Theory -- the Centennial of K.~Iwasawa's
  Birth}, 2020.

\bibitem{Kur20}
M.~Kurihara.
\newblock Notes on the dual of the ideal class groups of {CM}-fields.
\newblock {\em to appear in J. Th\'{e}or. Nombres Bordeaux}.

\bibitem{Nic11I}
A.~Nickel.
\newblock On the equivariant {T}amagawa number conjecture in tame
  {CM}-extensions.
\newblock {\em Math. Z.}, 268(1-2):1--35, 2011.

\bibitem{Nickel21}
A.~Nickel.
\newblock The strong {S}tark conjecture for totally odd characters.
\newblock {\em preprint, arXiv:2106.05619}, 2021.

\bibitem{Rib76}
K.~A. Ribet.
\newblock A modular construction of unramified {$p$}-extensions of {$Q(\mu
  \sb{p})$}.
\newblock {\em Invent. Math.}, 34(3):151--162, 1976.

\bibitem{RW96}
J.~Ritter and A.~Weiss.
\newblock A {T}ate sequence for global units.
\newblock {\em Compositio Math.}, 102(2):147--178, 1996.

\bibitem{Sil20}
J.~Silliman.
\newblock Group {R}ing {V}alued {H}ilbert {M}odular {F}orms.
\newblock {\em preprint, arXiv:2009.14353}, 2020.

\bibitem{Ven15}
K.~Ventullo.
\newblock On the rank one abelian {G}ross-{S}tark conjecture.
\newblock {\em Comment. Math. Helv.}, 90(4):939--963, 2015.

\bibitem{Wil88}
A.~Wiles.
\newblock On ordinary {$\lambda$}-adic representations associated to modular
  forms.
\newblock {\em Invent. Math.}, 94(3):529--573, 1988.

\bibitem{Wil90}
A.~Wiles.
\newblock The {I}wasawa conjecture for totally real fields.
\newblock {\em Ann. of Math. (2)}, 131(3):493--540, 1990.

\end{thebibliography}

\end{document}